\documentclass[10pt]{amsart} 
\usepackage[english]{babel}
\usepackage{graphicx,epsfig}
\usepackage{bbm}
\usepackage[noadjust]{cite}
\usepackage{amsmath,amssymb,latexsym, amsfonts, amscd, amsthm}
\usepackage{enumerate}
\usepackage{mathrsfs}
\usepackage[T1]{fontenc}
\usepackage{lmodern}

\usepackage{mathtools}
\usepackage[margin=1.2in]{geometry}

\usepackage[dvipsnames]{xcolor}

\usepackage{xspace} 

\usepackage{cellspace}
\usepackage{tikz-cd}
\usetikzlibrary{cd}

\usepackage[mathscr]{euscript} 

\usepackage{caption}

\usepackage{titlesec}
\usepackage{titletoc}

\titleformat{\section}
{\Large\bfseries\boldmath}{\thesection.}{0.5em}{}
\titlespacing{\section}{0em}{0.75cm}{0.4cm}

\titlecontents{section}
[1.5em] 
{}
{\contentslabel{1.5em}}
{\hspace*{-1.5em}}
{\titlerule*[1pc]{ }\contentspage}

\titlecontents{subsection}
[2.3em] 
{}
{\contentslabel{2.3em}}
{\hspace*{-2.3em}}
{\titlerule*[1pc]{ }\contentspage}

\titleformat{\subsection}
{\bfseries\boldmath}{\thesubsection.}{0.5em}{}
\titlespacing{\subsection}{0em}{0.5cm}{0.2cm}

\usepackage{relsize}
\usepackage[bbgreekl]{mathbbol}
\DeclareSymbolFontAlphabet{\mathbb}{AMSb} 
\DeclareSymbolFontAlphabet{\mathbbl}{bbold}
\newcommand{\Prism}{{\mathlarger{\mathbbl{\Delta}}}}

\makeindex 

\definecolor{ggreen}{rgb}{0.07, 0.55, 0.0}
\definecolor{rred}{rgb}{0.86, 0.18, 0.10}
\definecolor{bblue}{RGB}{15,70,150}

\usepackage[linktocpage=true]{hyperref}
\hypersetup{pdftoolbar=true, pdftitle={CrsEtaleRamification}, pdfauthor={Pavel Coupek}, pdffitwindow=true, colorlinks=true, citecolor=ggreen, filecolor=ggreen, linkcolor=rred, urlcolor=ggreen, hypertexnames=false}

\usepackage{tablefootnote}

\newtheoremstyle{vetalike}
{8pt}
{6pt}
{\slshape}
{}
{\bfseries}
{.}
{1em}
{}

\newtheoremstyle{poznslike}
{8pt}
{6pt}
{}
{}
{\bfseries}
{.}
{\newline}
{}

\newtheoremstyle{otherlike}
{8pt}
{6pt}
{}
{}
{\bfseries}
{.}
{1em}
{}

\newtheoremstyle{deflike}
{8pt}
{6pt}
{}
{}
{\bfseries}
{.}
{1em}
{}

\theoremstyle{vetalike}
\newtheorem{thm}{Theorem}[section]
\newtheorem{prop}[thm]{Proposition}
\newtheorem{lem}[thm]{Lemma}
\newtheorem{cor}[thm]{Corollary}

\theoremstyle{poznslike}
\newtheorem{rems}[thm]{Remarks}

\theoremstyle{deflike}
\newtheorem{deff}[thm]{Definition}
\newtheorem{nott}[thm]{Notation}
\newtheorem{constr}[thm]{Construction}
\newtheorem{rem}[thm]{Remark}

\theoremstyle{otherlike}

\newtheorem{pr}[thm]{Example}

\newcommand{\ainf}{A_{\mathrm{inf}}}
\newcommand{\Ainf}[1]{\mathbb{A}_{\mathrm{inf}}(#1)}
\newcommand{\BK}{\mathrm{BK}}
\newcommand{\et}{\text{\'{e}t}}

\newcommand{\spf}{\mathrm{Spf}}
\newcommand{\spec}{\mathrm{Spec}}
\newcommand{\shv}{\mathrm{Shv}}

\newcommand{\Crs}{\textnormal{($\mathrm{Cr}_s$)}\xspace}
\newcommand{\Cr}[1]{\textnormal{($\mathrm{Cr}_{#1}$)}\xspace}

\newcommand{\Crrs}{\textnormal{($\mathrm{Cr}'_s$)}\xspace}
\newcommand{\Crr}[1]{\textnormal{($\mathrm{Cr}'_{#1}$)}\xspace}

\newcommand{\oh}{\mathcal{O}}

\renewcommand{\H}{\mathrm{H}}
\newcommand{\R}{\mathsf{R}}

\newcommand{\cornerdown}{\rotatebox{45}{$\llcorner$}}

\newcommand{\lin}{\mathrm{lin}}

\newcommand{\Es}{\mathfrak{S}}

\begin{document}

\author{Pavel \v{C}oupek}
\address[Pavel \v{C}oupek]{Department of Mathematics, Michigan State University}
\email{coupekpa@msu.edu}
\date{}
\title{Crystalline condition for $\boldsymbol{\ainf}$--cohomology and ramification bounds}
\begin{abstract}
\noindent For a prime $p>2$ and a smooth proper $p$--adic formal scheme $\mathscr{X}$ over $\oh_K$ where $K$ is a $p$--adic field, we study a series of conditions  \Crs, $s\geq 0$ that partially control the $G_K$--action on the image of the associated Breuil--Kisin prismatic cohomology $\R \Gamma_{\Prism}(\mathscr{X}/ \Es)$ inside the $\ainf$--prismatic cohomology $\R \Gamma_{\Prism}(\mathscr{X}_{\ainf}/ \ainf)$. The condition \Cr{0} is a criterion for a Breuil--Kisin--Fargues $G_K$--module to induce a crystalline representation used by Gee and Liu in \cite[Appendix F]{EmertonGee2}, and thus leads to a proof of crystallinity of $\H^i_{\et}(\mathscr{X}_{\overline{\eta}}, \mathbb{Q}_p)$ that avoids the crystalline comparison. The higher conditions $\Crs$ are used to adapt the strategy of Caruso and Liu from \cite{CarusoLiu} to establish ramification bounds for the mod $p$ representations $\H^{i}_{\et}(\mathscr{X}_{\overline{\eta}}, \mathbb{Z}/p\mathbb{Z}),$ for arbitrary $e$ and $i$, which extend or improve existing bounds in various situations.
\end{abstract}
\maketitle

\tableofcontents

\section{Introduction}

Let $k$ be a perfect field of characteristic $p>2$ and $K'=W(k)[1/p]$ the associated absolutely unramified field. Let $K/K'$ be a totally ramified finite extension with ramification index $e$, and denote by $G_K$ its absolute Galois group. The goal of the present paper is to provide new bounds for ramification of the mod $p$ representations of $G_K$ that arise as the \'{e}tale cohomology groups $\H^i_{\et}(\mathscr{X}_{\overline{\eta}}, \mathbb{Z}/p\mathbb{Z})$ in terms of $p, i$ and $e$, where $\mathscr{X}$ is a smooth proper $p$--adic formal scheme over $\oh_K$ (and $\mathscr{X}_{\overline{\eta}}$ is the geometric adic generic fiber). Concretely, let us denote by  $G_K^{\mu}$ the $\mu$--th ramification group of $G_K$ in the upper numbering (in the standard convention, e.g. \cite{SerreLocalFields}) and $G_K^{(\mu)}=G_K^{\mu-1}$. The main result is as follows.

\begin{thm}[Theorem~\ref{thm:FinalRamification}]\label{thm:IntroMain}
Set $$\alpha=\left\lfloor\mathrm{log}_p\left( \mathrm{max} \left\{\frac{ip}{p-1}, \frac{(i-1)e}{p-1}\right\}\right)\right\rfloor+1, \;\;\; \beta=\frac{1}{p^\alpha}\left(\frac{iep}{p-1}-1\right).$$
Then:
\begin{enumerate}[(1)]
\item{The group $G_K^{(\mu)}$ acts trivially on $\H^i_{\et}(\mathscr{X}_{\overline{\eta}}, \mathbb{Z}/p\mathbb{Z})$ when $\mu>1+e\alpha+\mathrm{max}\left\{\beta, \frac{e}{p-1}\right\}.$}
\item{Denote by $L$ the field $\overline{K}^H$ where $H$ is the kernel of the $G_K$--representation $\rho$ given by $\H^i_{\et}(\mathscr{X}_{\overline{\eta}}, \mathbb{Z}/p\mathbb{Z})$. Then
$$v_K(\mathcal{D}_{L/K})<1+e\alpha+\beta,$$
where $\mathcal{D}_{L/K}$ denotes the different of the extension $L/K$ and $v_K$ denotes the additive valuation on $K$ normalized so that $v_K(K^{\times})=\mathbb{Z}.$}
\end{enumerate}
\end{thm}
\noindent In particular, unlike in previous results of this type (discussed below), there are no restrictions on the size of $e$ and $i$ with respect to $p$. 

\begin{rem}
As the constants $\alpha, \beta$ appearing in Theorem~\ref{thm:IntroMain} are quite complicated, let us draw some non--optimal, but more tractable consequences. The group $G_K^{(\mu)}$ acts trivially on $\H^i_{\et}(\mathscr{X}_{\overline{\eta}}, \mathbb{Z}/p\mathbb{Z})$ when one of the following occurs:
\begin{enumerate}[(1)]
\item{$e \leq p$ and $\mu>1+e\left(\left\lfloor \mathrm{log}_p\left(\frac{ip}{p-1}\right)\right\rfloor+1\right)+e,$}
\item{$e>p$ and $\mu>1+e\left(\left\lfloor \mathrm{log}_p\left(\frac{ie}{p-1}\right)\right\rfloor+1\right)+p,$\footnote{Strictly speaking, to obtain this precise form one has to replace $(i-1)e$ in $\alpha$ from Theorem~\ref{thm:IntroMain} by $ie,$ and modify $\beta$ appropriately; one can show that such form of Theorem~\ref{thm:IntroMain} is still valid.}}
\item{$i=1$ ($e, p$ are arbitrary) and $\mu>1+e\left(1+\frac{1}{p-1}\right).$} 
\end{enumerate}
\end{rem}

\vspace{1em}

Let us briefly summarize the history of related results. Questions of this type originate in Fontaine's paper \cite{Fontaine}, where he proved that for a finite flat group scheme $\Gamma$ over $\oh_K$ that is annihilated by $p^n$, $G_K^{(\mu)}$ acts trivially on $\Gamma(\overline{K})$ when $\mu>e(n+1/(p-1))$; this is a key step in his proof that there are no non--trivial abelian schemes over $\mathbb{Z}$. In the same paper, Fontaine conjectured that general $p^n$--torsion cohomology would follow the same pattern: given a proper smooth variety $X$ over $K$ with good reduction, $G_K^{(\mu)}$ should act trivially on $\H^i_{\et}(X_{\overline{K}}, \mathbb{Z}/p^n\mathbb{Z})$ when $\mu>e(n+i/(p-1))$. 

This conjecture has been subsequently proved by Fontaine himself (\cite{Fontaine2}) in the case when $e=n=1, i<p-1$ and by Abrashkin (\cite{Abrashkin}; see also \cite{Abrashkin2}) when $e=1, i<p-1$ and $n$ is arbitrary. This is achieved by using Fontaine--Laffaille modules (introduced in \cite{FontaineLaffaille}), which parametrize quotients of pairs of $G_K$--stable lattices in crystalline representations with Hodge--Tate weights in $[0, i]$ (such as $\H^i_{\et}({X}_{\overline{K}}, \mathbb{Q}_p)^{\vee}$). The (duals of the) representations $\H^i_{\et}({X}_{\overline{K}}, \mathbb{Z}/p^n\mathbb{Z})$ are included among these thanks to a comparison theorem of Fontaine--Messing (\cite{FontaineMessing}). Similarly to the orginal application, these ramification bounds lead to a scarcity result for existence of smooth proper $\mathbb{Z}$--schemes.

Various extensions to the semistable case subsequently followed. Under the asumption $i<p-1$ (and arbitrary $e$), Hattori proved in \cite{Hattori} a ramification bound for $p^n$--torsion quotients of lattices in semistable representations with Hodge--Tate weights in the range $[0, i],$ using (a variant of) Breuil's filtered $(\phi_r, N)$--modules. Thanks to a comparison result between log--crystalline and \'{e}tale cohomology by Caruso (\cite{CarusoLogCris}), this results in a ramification bound for $\H^i_{\et}({X}_{\overline{K}}, \mathbb{Z}/p^n\mathbb{Z})$ when ${X}$ is proper with semistable reduction, assuming $ie<p-1$ when $n=1$ and $(i+1)e<p-1$ when $n \geq 2$ \footnote{Recently, in \cite{LiLiu} Li and Liu extended Caruso's result to the range $ie<p-1$ regardless of $n$, for $\mathscr{X}/\oh_K$ proper and smooth (formal) scheme. In view of this, results of \cite{Hattori} should apply in these situations as well.}.

These results were further extended by Caruso and Liu in \cite{CarusoLiu} for all $p^n$--torsion quotients of pairs of semistable lattices with Hodge--Tate weights in $[0, i]$, without any restriction on $i$ or $e$. The proof uses the theory of $(\varphi , \widehat{G})$--modules, which are objects suitable for description of lattices in semistable representations. Roughly speaking, a $(\varphi , \widehat{G})$--module consists of a Breuil--Kisin module $M$ and the datum of an action of $\widehat{G}=\mathrm{Gal}(K(\mu_{p^{\infty}}, \pi^{1/p^\infty})/K)$ on $\widehat{M}=M \otimes_{\Es, \varphi} \widehat{\mathcal{R}}$ where $\widehat{\mathcal{R}}$ is a suitable subring of Fontaine's period ring $\ainf=W(\oh_{\mathbb{C}_{K}^\flat})$ (and $\pi\in K$ is a fixed choice of a uniformizer). An obstacle to applying the results of \cite{CarusoLiu} to the torsion \'{e}tale cohomology groups $\H^i_{\et}(X_{\overline{K}}, \mathbb{Z}/p\mathbb{Z})$ is that it is not quite clear when (duals of) such representations come as a quotient of two semistable lattices with Hodge--Tate weights in $[0, i].$ This is indeed the case in the situation when $e=1$, $i<p-1$ and $X$ has good reduction by the aforementioned Fontaine--Messing theorem, and it was also shown in the case $i=1$ (no restriction on $e, p$) for $X$ with semistable reduction by Emerton and Gee in \cite{EmertonGee1}, but in general the question seems open.

\vspace{1em}

Nevertheless, the idea of the proof of Theorem~\ref{thm:IntroMain} is to follow the general strategy of Caruso and Liu. While one does not necessarily have semistable lattices and the associated $(\varphi, \widehat{G})$--modules to work with, a suitable replacement comes from the recently developed cohomology theories of Bhatt--Morrow--Scholze and Bhatt--Scholze (\cite{BMS1, BMS2, BhattScholze}). Concretely, to a smooth $p$--adic formal scheme $\mathscr{X}$ one can associate the ``$p^n$-torsion prismatic cohomologies'' $$\R\Gamma_{\Prism, n}(\mathscr{X}/ \Es)=\R\Gamma_{\Prism}(\mathscr{X}/ \Es)\stackrel{{\mathsf{L}}}{\otimes}\mathbb{Z}/p^n\mathbb{Z}, \;\;\;\;\;\;\R\Gamma_{\Prism, n}(\mathscr{X}_{\ainf}/ \ainf)=\R\Gamma_{\Prism}(\mathscr{X}_{\ainf}/ \ainf)\stackrel{{\mathsf{L}}}{\otimes}\mathbb{Z}/p^n\mathbb{Z}$$
where $\R\Gamma_{\Prism}(\mathscr{X}_{\ainf}/ \ainf), \R\Gamma_{\Prism}(\mathscr{X}/ \Es)$ are the prismatic avatars of the $\ainf$-- and Breuil--Kisin cohomologies from \cite{BMS1} and \cite{BMS2}, resp.
Taking $M_{\BK}=\H^i_{\Prism, 1}(\mathscr{X}/ \Es)$ and $M_{\inf}=\H^i_{\Prism, 1}(\mathscr{X}/ \ainf),$  Li and Liu showed in \cite{LiLiu} that $M_{\BK}$ is a $p$--torsion Breuil--Kisin module, $M_{\inf}$ is a $p$--torsion Breuil--Kisin--Fargues $G_K$--module, and that these modules recover the \'{e}tale cohomology group $\H^i_{\et}(\mathscr{X}_{\overline{\eta}}, \mathbb{Z}/p\mathbb{Z})$ essentially due to the \'{e}tale comparison theorem for prismatic cohomology from \cite{BhattScholze}. The pair $(M_{\BK}, M_{\inf})$ then serves as a suitable replacement of a $(\varphi, \widehat{G})$--module in our context.

The most significant deviation from the strategy of \cite{CarusoLiu} then stems from the fact that the pair $(M_{\BK}, M_{\inf})$ obtained this way is ``inherently $p$--torsion'', that is, it does not come equipped with any apparent lift to analogous objects in characteristic $0$. This is not the case in \cite{CarusoLiu}, where all torsion modules ultimately originate from a free $(\varphi, \widehat{G})$--module $(M, \widehat{M})$. A key technical input in \textit{loc. cit.} is to establish a partial control on the Galois action on $M$ inside $\widehat{M},$ namely, a condition of the form 
\begin{equation}\label{IntroPhiGiHatCondition}
\forall g \in G_{K(\pi^{1/p^s})},\;\;\forall x \in M:\;\;g(x)-x \in (J_{n,s}+p^n\ainf) (\widehat{M}\otimes_{\widehat{\mathcal{R}}}\ainf).
\end{equation} 
Here $J_{n, s} \subseteq \ainf$ are certain ideals (that are shrinking with growing $s$). This is a ``rational'' fact, in the sense that this claim is a consequence of the description of the Galois action in terms of the monodromy operator on the associated Breuil module $\mathcal{D}(\widehat{M})$ (cf. \cite{Breuil}, \cite[\S 3.2]{LiuLatticesNew}), a vector space over the characteristic $0$ field $K'$. 

As a starting point for replacing (\ref{IntroPhiGiHatCondition}) in our context, we turn to a result by Gee and Liu in \cite[Appendix F]{EmertonGee2} (see also \cite[Theorem~3.8]{Ozeki}). Given a finite free Breuil--Kisin module $M_{\BK}$ (of finite height) and a compatible structure of Breuil--Kisin--Fargues $G_K$--module on $M_{\inf}=M_{\BK}\otimes_{\Es}\ainf,$ such that the image of $M_{\BK}$ under the natural map lands in $(M_{\inf})^{G_{K(\pi^{1/p^{\infty}})}}$, the \'{e}tale realization of $M_{\inf}$ is crystalline if and only if
\begin{equation}\tag{$\mathrm{Cr}_0$}\label{IntroCrysCondition}
\forall g \in G_K,\;\; \forall x \in M_{\BK}: g(x)-x \in \varphi^{-1}([\underline{\varepsilon}]-1)[\underline{\pi}]M_{\inf}.
\end{equation}
Here $[-]$ denotes the Teichm\"{u}ller lift and $\underline{\varepsilon}, \underline{\pi}$ are the elements of $\oh_{\mathbb{C}_K^\flat}$ given by a collection $(\zeta_{p^n})_n$ of (compatible) $p^n$--th roots of unity and a collection $(\pi^{1/p^n})_n$ of $p^n$--th roots of the chosen uniformizer $\pi$, resp. We call condition~(\ref{IntroCrysCondition}) the \emph{crystalline condition}. As the considered formal scheme $\mathscr{X}$ is assumed to be smooth over $\oh_K$, it is reasonable to expect that the same condition applies to the pair $M_{\BK}=\H^i_{\Prism}(\mathscr{X}/ \Es)$ and $M_{\inf}=\H^i_{\Prism}(\mathscr{X}_{\ainf}/ \ainf)$, despite the fact that the Breuil--Kisin and Breuil--Kisin--Fargues modules coming from prismatic cohomology are not necessarily free. 

This is indeed the case and, moreover, it can be shown that the crystalline condition even applies to the embedding of the chain complexes $\R\Gamma_{\Prism}(\mathscr{X}/ \Es)\rightarrow \R\Gamma_{\Prism}(\mathscr{X}_{\ainf}/ \ainf)$: to make sense of this claim, we model the cohomology theories by their associated \v{C}ech--Alexander complexes. These were introduced in \cite{BhattScholze} in the case that $\mathscr{X}$ is affine, but can be extended to (at least) arbitrary separated smooth $p$--adic formal schemes. We are then able to verify the condition termwise for this pair of complexes. More generally, we introduce a decreasing series of ideals $I_s$, $s \geq 0$  where $I_0=\varphi^{-1}([\underline{\varepsilon}]-1)[\underline{\pi}]\ainf,$ and then formulate and prove the analogue of (\ref{IntroCrysCondition}) for $I_s$ and the action of $G_{K(\pi^{1/p^s})}.$ As a consequence, we obtain:

\begin{thm}[Theorem~\ref{thm:CrsForCechComplex}, Corollary~\ref{cor:CrystallineForCohomologyGrps}, Proposition~\ref{prop:CrystallineCohMopPN}]\label{thm:IntroCrysCondition}
Let $\mathscr{X}$ be a smooth separated $p$--adic formal scheme over $\oh_K$.
\begin{enumerate}[(1)]
\item{Fix a compatible choice of \v{C}ech--Alexander complexes $\check{C}_{\BK}^{\bullet}\subseteq \check{C}_{\inf}^{\bullet}$ that compute $\R\Gamma_{\Prism}(\mathscr{X}/ \Es)$ and $\R\Gamma_{\Prism}(\mathscr{X}_{\ainf}/ \ainf)$, resp. Then for all $s \geq 0$, the pair $(\check{C}_{\BK}^{\bullet},\check{C}_{\inf}^{\bullet})$ satisfies (termwise) the condition
\begin{equation}\tag{$\mathrm{Cr}_s$}
\forall g \in G_{K(\pi^{1/p^s})}, \;\; \forall x \in \check{C}_{\BK}^{\bullet}: g(x)-x \in I_s  \check{C}_{\inf}^{\bullet}.
\end{equation}}
\item{The associated prismatic cohomology groups satisfy the crystalline condition, that is, the condition
$$
\forall g \in G_{K}, \;\; \forall x \in \H^i_{\Prism}(\mathscr{X}/ \Es): \;\; g(x)-x \in \varphi^{-1}([\underline{\varepsilon}]-1)[\underline{\pi}] \H^i_{\Prism}(\mathscr{X}_{\ainf}/ \ainf).
$$}
\item{For all pairs of integers $s, n$ with $s+1\geq n \geq 1$, the $p^n$--torsion prismatic cohomology groups satisfy the condition
$$
\forall g \in G_{K(\pi^{1/p^s})}, \;\; \forall x \in \H^i_{\Prism, n}(\mathscr{X}/ \Es): \;\; g(x)-x \in \varphi^{-1}([\underline{\varepsilon}]-1)[\underline{\pi}]^{p^{s+1-n}} \H^i_{\Prism, n}(\mathscr{X}_{\ainf}/ \ainf).
$$}
\end{enumerate}
\end{thm}    

Theorem~\ref{thm:IntroCrysCondition}~(3) specialized to $n=1$ provides the desired analogue of the property~(\ref{IntroPhiGiHatCondition}) of $(\varphi, \widehat{G})$--modules and allows us to carry out the proof of Theorem~\ref{thm:IntroMain}.

As a consequence of Theorem~\ref{thm:IntroCrysCondition} (2), we obtain a proof of crystallinity of the cohomology groups $\H^i_{\et}(\mathscr{X}_{\overline{\eta}}, \mathbb{Q}_p)$ in the proper case partially by means of ``formal'' $p$--adic Hodge theory (Corollary~\ref{cor:EtaleCohomologyCrystalline}). This fact is usually established via a a direct comparison between crystalline and \'{e}tale cohomology, and in this generality is originally due to Bhatt, Morrow and Scholze (\cite{BMS1}). Of course, since our setup relies on the machinery of prismatic cohomology and especially the \'{e}tale comparison, the proof can be considered independent of the one from \cite{BMS1} only in that it avoids the crystalline comparison theorem for (prismatic) $\ainf$--cohomology.

\vspace{1em}

The bounds of Theorem~\ref{thm:IntroMain} compare to the already known bounds as follows. Whenever the bounds of ``semistable type'' are known to apply to the situation of $\H^i_{\et}(\mathscr{X}_{\overline{\eta}}, \mathbb{Z}/p\mathbb{Z})$ (e.g. \cite{CarusoLiu} when $i=1$, \cite{Hattori} when $ie<p-1$ and $\mathscr{X}$ is a scheme), the bounds from Theorem~\ref{thm:IntroMain} agree with those bounds. The bounds tailored to crystalline representations (\cite{Fontaine2, Abrashkin}) are slightly better but their applicability is quite limited ($e=1$ and $i<p-1$).

The fact that the cohomology groups $\H^i_{\et}(\mathscr{X}_{\overline{\eta}}, \mathbb{Z}/p^n\mathbb{Z})$ have an associated Breuil--Kisin module yields one more source of ramification estimates: in \cite{Caruso}, Caruso provides a very general bound for $p^n$--torsion $G_K$--modules based on their restriction to $G_{K(\pi^{1/p^\infty})}$ via Fontaine's theory of \'{e}tale $\oh_{\mathcal{E}}$--modules. Using the Breuil--Kisin module $\H^i_{\Prism,n }(\mathscr{X}/ \Es)$ attached to $\H^i_{\et}(\mathscr{X}_{\overline{\eta}}, \mathbb{Z}/p^n\mathbb{Z})$, this bound becomes explicit (as discussed in more detail in Remark~\ref{rem:CarusoBound}). Comparing this result to Theorem~\ref{thm:IntroMain} is more ambiguous due to somewhat different shapes of the estimates, but roughly speaking, the estimate of Theorem~\ref{thm:IntroMain} is approximately the same for $e \leq p$, becomes worse when $K$ is absolutely tamely ramified with large ramification degree, and is expected to outperform Caruso's bound in case of large wild absolute ramification (rel. to the tame part of the ramification). 

In future work, we intend to extend the result of Theorem~\ref{thm:IntroMain} to the case of arbitrary $n$. This seems plausible thanks to the full statement of Theorem~\ref{thm:IntroCrysCondition} (3). In a different direction, we plan to extend the results of the present paper to the case of semistable reduction, using the log--prismatic cohomology developed by Koshikawa in \cite{Koshikawa}. An important facts in this regard are that the $\ainf$--log--prismatic cohomology groups are still Breuil--Kisin--Fargues $G_K$--modules by a result of \v{C}esnavi\v{c}ius and Koshikawa (\cite{CesnaviciusKoshikawa}) and that by results of Gao, a variant of the condition \Cr{0} might exist in the semistable case (\cite{GaoBKGK}; see Remark~\ref{rem:CrystConditionProof} (3) below for details).

\vspace{1em}

The outline of the paper is as follows. In \S\ref{sec:prelim} we establish some necessary technical results. Namely, we discuss non--zero divisors and regular sequences on derived complete and completely flat modules with respect to the weak topology of $\ainf$, and establish \v{C}ech--Alexander complexes in the case of a separated and smooth formal scheme. Next, \S\ref{sec:crs} introduces the conditions \Crs, studies their basic algebraic properties and discusses in particular the crystalline condition \Cr{0} in the case of Breuil--Kisin--Fargues $G_K$--modules. In \S\ref{sec:CrsCohomology} we prove the conditions \Crs for the Alexander--\v{C}ech complexes of a separated smooth  $p$--adic scheme $\mathscr{X}$ over $\Es$ and $\ainf$, and draw some consequences for the inidividual cohomology groups (especially when $\mathscr{X}$ is proper), proving Theorem~\ref{thm:IntroCrysCondition}. Finally,  in \S\ref{sec:bounds} we establish the ramification bounds for mod $p$ \'{e}tale cohomology, proving Theorem~\ref{thm:IntroMain}. Subsequently, we discuss in detail how the bounds from Theorem~\ref{thm:IntroMain} compare to the various known bounds from the literature.

\vspace{1em}

Let us set up some basic notation used throughout the paper. We fix a perfect field $k$ of characteristic $p>0$ and a finite totally ramified extension $K / K'$ of degree $e$ where $K'=W(k)[1/p]$. We fix a uniformizer $\pi \in \mathcal{O}_K$, and a compatible system $(\pi_n)_n$ of $p^n$--th roots of $\pi$ in $\mathbb{C}_K$, the completion of algebraic closure of $K$. Setting $\Es=W(k)[[u]]$, the choice of $\pi$ determines a surjective map $\Es \rightarrow \mathcal{O}_{K}$ via $u \mapsto \pi$; the kernel of this map is generated by an Eisenstein polynomial $E(u)$ of degree $e$. $\Es$ is endowed with a Frobenius lift (hence a $\delta$--structure) extending the one on $W(k)$ by $u \mapsto u^p$.

Denote $\ainf=\Ainf{\mathcal{O}_{\mathbb{C}_K}}=W(\mathcal{O}_{\mathbb{C}_K^{\flat}})$ where $W(-)$ denotes the Witt vectors construction and $\mathcal{O}_{\mathbb{C}_K^{\flat}}=\mathcal{O}_{\mathbb{C}_K}^{\flat}$ is the tilt of $\oh_{\mathbb{C}_K}$, $\mathcal{O}_{\mathbb{C}_K}^{\flat}=\varprojlim_{x \mapsto x^p}\mathcal{O}_{\mathbb{C}_K}/p$. The choice of the system $(\pi_n)_n$ describes an element $\underline{\pi} \in \oh_{\mathbb{C}_K^\flat}\simeq\varprojlim_{x \mapsto x^p}\oh_{\mathbb{C}_K}$, and hence an embedding of $\Es$ into $\ainf$ via $u \mapsto [\underline{\pi}]$ where $[-]$ denotes the Teichm\"{u}ller lift. Under this embedding, $E(u)$ is sent to a generator $\xi$ of the kernel of the canonical map $\theta: \ainf\rightarrow \mathcal{O}_{\mathbb{C}_K}$ that lifts the canonical projection $\mathrm{pr}_0:\mathcal{O}_{\mathbb{C}_K}^{\flat}=\varprojlim_{\varphi}\mathcal{O}_{\mathbb{C}_K}/p \rightarrow \mathcal{O}_{\mathbb{C}_K}/p.$ Consequently, $(\Es, (E(u)))\rightarrow (\ainf, \mathrm{Ker}\,\theta)$ is a map of prisms. It is known that under such embedding, $\ainf$ is faithfully flat over $\Es$ (see e.g. \cite[Proposition~2.2.13]{EmertonGee2}).

Similarly, we fix a choice of a compatible system of primitive $p^{n}$--th roots of unity $(\zeta_{p^n})_{n \geq 0}$. This defines an element $\underline{\varepsilon}$ of $\oh_{\mathbb{C}_K^\flat}$ in an analogous manner, and the embedding $\Es\hookrightarrow \ainf$ extends to a map (actually still an embedding by \cite[Proposition~1.14]{Caruso}) $W(k)[[u, v]] \rightarrow \ainf$ by additionally setting $v \mapsto [\underline{\varepsilon}]-1$. Additionally, we denote by $\omega$ the element $([\underline{\varepsilon}]-1)/([\underline{\varepsilon}^{1/p}]-1)=[\underline{\varepsilon}^{1/p}]^{p-1}+\dots+[\underline{\varepsilon}^{1/p}]+1$. It is well--known that this is another generator of $\mathrm{Ker}\,\theta$, therefore $\omega/\xi$ is a unit in $\ainf$.

The choices of $\pi, \pi_n$ and $\zeta_{p^n}$, hence also the maps $\Es\hookrightarrow \ainf$ and $W(k)[[u,v]]\hookrightarrow \ainf$, remain fixed throughout. For this reason, we often refer to $[\underline{\pi}]$ as $u$, $[\underline{\varepsilon}]-1$ as $v$, $\xi$ as $E(u),$ etc.

Throughout the paper, we use freely the language of prisms and $\delta$--rings from \cite{BhattScholze}, and we adopt much of the related notation and conventions. In particular, a formal scheme $\mathscr{X}$ over a $p$--adically complete ring $A$ always means a $p$--adic formal scheme, and it is called smooth if it is locally of the form $\mathrm{Spf}\,R$ for a (derived\footnote{As we will always consider the base $A$ to have bounded $p^{\infty}$--torsion, there is no distinction between derived $p$--completion and $p$--adic completion in this case.}) $p$--completely smooth $A$--algebra $R$ -- that is, a $p$--complete $A$--algebra such that $R/p$ is a smooth $A/p$--algebra and $\mathrm{Tor}_i^A(R, A/p)=0$ for all $i>0$. By the results of Elkik \cite{Elkik} and the discussion in \cite[\S 1.2]{BhattScholze}, $R$ is equivalently the $p$--adic completion of a smooth $A$-algebra.

\vspace{1em}

\textbf{Acknowledgements.} I would like to express my gratitude to my Ph.D. advisor Tong Liu for suggesting the topic of this paper, his constant encouragement and many comments, suggestions and valuable insights. Many thanks go to Deepam Patel and Shuddhodan Kadattur Vasudevan for organizing the prismatic cohomology learning seminar at Purdue University in Fall 2019, and to Donu Arapura for a useful discussion of \v{C}ech theory. I would like to thank Xavier Caruso, Shin Hattori and  Shizhang Li for reading an earlier version of this paper, and for providing me with useful comments and questions. The present paper is an adapted version of the author's Ph.D. thesis at Purdue University. During the preparation of the paper, the author was partially supported by the Ross Fellowship and the Bilsland Fellowship of Purdue University, as well as Graduate School Summer Research Grants of Purdue University during summers 2019--2021.

\section{Preparations}\label{sec:prelim}

\subsection{Regularity on $(p, E(u))$--completely flat modules}\label{subsec:regularity}
The goal of this section is to prove that every $(p, E(u))$--complete and $(p, E(u))$--completely flat $\ainf$--module is torsion--free, and that any sequence $p, x$ with $x \in \ainf\setminus(\ainf^\times \cup p\ainf)$ is regular on such modules.

Regarding completions and complete flatness, we adopt the terminology of \cite[091N]{stacks}, \cite{BhattScholze}, but since we apply these notions mostly to modules as opposed to objects of derived categories, our treatment is closer in spirit to \cite{Positselski}, \cite{Rezk} and \cite{YekutieliFlatness}. Given a ring $A$ and a finitely generated ideal $I=(f_1, f_2, \dots f_n)$, the derived $I$--completion\footnote{That is, this is derived $I$--completion of $M$ as a module. This will be sufficient to consider for our purposes.} of an $A$--module $M$ is 
\begin{equation}\label{eqn:completion}
\widehat{M}=M[[X_1, X_2, \dots X_n]]/(X_1-f_1, X_2-f_2, \dots, X_n-f_n)M[[X_1, X_2, \dots X_n]].
\end{equation}
$M$ is said to be \emph{derived $I$--complete} if the natural map $M \rightarrow \widehat{M}$ is an isomorphism. This is equivalent to the vanishing of $\mathrm{Ext}^i_A(A_f, M)$ for $i=0, 1$ and all $f \in I$ (equivalently, for $f=f_j$ for all $j$), and as a consequence, it can be shown that the category of derived $I$--complete modules forms a full abelian subcategory of the category of all $A$--modules with exact inclusion functor (and the derived $I$--completion is its left adjoint; in particular, derived $I$--completion is right exact as a functor on $A$--modules). Another consequence is that derived $I$--completeness is equivalent to derived $J$--completeness when $I, J$ are two finitely generated ideals and $\sqrt{I}=\sqrt{J}$. There is always a natural surjection $\widehat{M}\rightarrow {\widehat{M}}^{\mathrm{cl}}$ where $\widehat{(-)}^{\mathrm{cl}}$ stands for $I$--adic completion, which will be reffered to as classical $I$--completion for the rest of the paper. Just like for classsicaly $I$--complete modules, if $M$ is derived $I$--complete, then $M/IM=0$ implies $M=0$ (this is referred to as \emph{derived Nakayama lemma}).

A convenient consequence of the completion formula (\ref{eqn:completion}) is that in the case when $M=R$ is a derived $I$--complete $A$--algebra, the isomorphism $R \rightarrow R[[X_1, \dots X_n]]/(X_1-f_1, \dots,..., X_n-f_n)$ picks a preferred representative in $R$ for the power series symbol $\sum_{j_1, \dots, j_n}a_{j_1, \dots, j_n}f_1^{j_1}\dots f_m^{j_n}$ as the preimage of the class represented by $\sum_{j_1, \dots, j_n}a_{j_1, \dots, j_n}X_1^{j_1}\dots X_n^{j_n}$. This gives an algebraically well--behaved notion of power series summation despite the fact that $R$ is not necessarily $I$--adically separated\footnote{This operation further leads to the notion of contramodules, discussed e.g. in \cite{Positselski}.}. 

An $A$--module $M$ is said to be \emph{$I$--completely (faithfully) flat} if $\mathrm{Tor}_i^A(M, A/I)=0$ for all $i>0$ and $M/IM$ is a (faithfully) flat $A/I$--module. Just like for derived completeness, $I$--complete flatness is equivalent to $J$--complete flatness when $J$ is another finitely generated ideal with $\sqrt{I}=\sqrt{J}$ \footnote{However, note that while (derived) $I$--completeness more generally implies (derived) $I'$--completeness when $I'$ is a finitely generated ideal contained in $\sqrt{I}$, the ``opposite'' works for flatness, i.e. $I$--complete flatness implies $I''$--complete flatness when when $I''$ is a finitely generated ideal with $I\subseteq \sqrt{I''}$.}.  

Let us start by a brief discussion of regular sequences on derived complete modules in general. For that purpose, given an $A$--module $M$ and $\underline{f}=f_1, \dots, f_n \in A$, denote by $\mathrm{Kos}(M; \underline{f})$ the usual Koszul complex and let $H_m(M; \underline{f})$ denote the $m$-th Koszul homology of $M$ with respect to $f_1, f_2, \dots, f_n$. 

The first lemma is a straightforward generalization of standard facts about Koszul homology (e.g. \cite[Theorem~16.5]{Matsumura}) and regularity on finitely generated modules. 

\begin{lem}\label{regKoszul}
Let $A$ be a ring, $I \subseteq A$ a finitely generated ideal and let $M$ be a nonzero derived $I$-complete module. Let $\underline{f}=f_1, f_2, \dots, f_n \in I$. Then
\begin{enumerate}
\item{$\underline{f}$ forms a regular sequence on $M$ if and only if $H_m(M; \underline{f})=0$ for all $m \geq 1$ if and only if $H_1(M; \underline{f})=0$.}
\item{In this situation, any permutation of $f_1, f_2, \dots, f_n$ is also a regular sequence on $M$.}
\end{enumerate}
\end{lem}

\begin{proof}
As Koszul homology is insensitive to the order of the elements $f_1, f_2, \dots, f_n$, part (2) follows immediately from (1).

To prove (1), the forward implications are standard and hold in full generality (see e.g. \cite[Theorem~16.5]{Matsumura}). It remains to prove that the sequence $f_1, f_2, \dots f_n$ is regular on $M$ if $H_1(M; f_1, f_2, \dots, f_n)=0$. We proceed by induction on $n$. The case $n=1$ is clear ($H_1(M; x)=M[x]$ by definition, and $M/xM\neq 0$ follows by derived Nakayama). Let $n \geq 2$, and denote $\underline{f}'$ the truncated sequence $f_1, f_2, \dots, f_{n-1}$. Then we have $\mathrm{Kos}(M; \underline{f})\simeq \mathrm{Kos}(M; \underline{f'})\otimes\mathrm{Kos}(A; f_n),$ which produces a short exact sequence
$$0 \longrightarrow \mathrm{Kos}(M; \underline{f'})\longrightarrow \mathrm{Kos}(M; \underline{f})\longrightarrow \mathrm{Kos}(M; \underline{f'})[-1]\longrightarrow 0$$ of chain complexes. Taking homologies results in a long exact sequence
$$\cdots \rightarrow H_1(M; \underline{f}') \stackrel{\pm f_n}{\longrightarrow} H_1(M; \underline{f}')\longrightarrow  H_1(M; \underline{f})\longrightarrow M/(\underline{f}')M \stackrel{\pm f_n}{\rightarrow} M/(\underline{f}')M\longrightarrow M/(\underline{f})M\rightarrow 0$$
(as in \cite[Theorem~7.4]{Matsumura}).
By assumption, $H_1(M; \underline{f})=0$ and thus, $f_n H_1(M; \underline{f}')=H_1(M; \underline{f}')$ where $f_n \in I$. Upon observing that $H_1(M; \underline{f}')$ is obtained from finite direct sum of copies of $M$ by repeatedly taking kernels and cokernels, it is derived $I$--complete. Thus, derived Nakayama implies that $H_1(M; \underline{f}')=0$ as well, and by induction hypothesis, $\underline{f}'$ is a regular sequence on $M$. Finally, the above exact sequence also implies that $f_n$ is injective on $M/(\underline{f}')M,$ and $M/(\underline{f})M\neq 0$ is satisfied thanks to derived Nakayama again. This finishes the proof. 
\end{proof}

\begin{cor}\label{FlatReg}
Let $A$ be a derived $I$--complete ring for an ideal $I=(\underline{f})$ where $\underline{f}=f_1, f_2, \dots, f_n$ is a regular sequence on $A$, and let $F$ be a nonzero derived $I$--complete $A$--module that is $I$--completely flat. Then $\underline{f}$ is a regular sequence on $F$ and consequently, each $f_i$ is a non--zero divisor on $F$.
\end{cor}

\begin{proof}
By Lemma~\ref{regKoszul} (1), $H_m(A; \underline{f})=0$ for all $m \geq 1$, hence $\mathrm{Kos}(A; \underline{f})$ is a free resolution of $A/I$. Thus, on one hand, the complex $F\otimes_A\mathrm{Kos}(A; \underline{f})$ computes $\mathrm{Tor}^A_*(F, A/I)$, hence is acyclic in positive degrees by $I$--complete flatness; on the other hand, this complex is by definition $\mathrm{Kos}(F; \underline{f})$. We may thus conclude that $H_i(F; \underline{f})=0$ for all $i \geq 1$. By Lemma~\ref{regKoszul}, $\underline{f}$ is a regular sequence on $F$, and it remains regular on $F$ after arbitrary permutation. This proves the claim. 
\end{proof}

Now we specialize to the case at hand, that is, $A= \ainf$. Recall that this is a domain and so is $\ainf/p=\oh_{\mathbb{C}_K^\flat}$ (which is a rank $1$ valuation ring).

\begin{lem}\label{disjointness}
For every element $x \in\ainf\setminus(\ainf^{\times} \cup p\ainf), $ $p, x$ forms a regular sequence, and for all $k, l,$ we have the equality 
$p^k\ainf \cap x^l\ainf =p^kx^l\ainf,$. Furthermore, the ideal $\sqrt{(p, x)}$ is equal to $(p,W(\mathfrak{m}_{\mathbb{C}_K^{\flat}}))$, the unique maximal ideal of $\ainf$. In particular, given two choices $x, x'$ as above, we have $\sqrt{(p, x)}=\sqrt{(p, x')}$.
\end{lem}

In particular, the equalities ``$\sqrt{(p, x)}=\sqrt{(p, x')}$'' imply that all the $(p, x)$--adic topologies (for $x$ as above) are equivalent to each other; this is the so--called weak topology on $\ainf$ (usually defined as $(p, u)$--adic topology in our notation), and it is standard that $\ainf$ is complete with respect to this topology.

\begin{proof}
By assumption, the image $\overline{x}$ of $x$ in $\ainf/p=\mathcal{O}_{\mathbb{C}_K^{\flat}}$ is non--zero and non--unit in $\ainf/p$ (non--unit since $x \notin \ainf^{\times}$ and $p \in \mathrm{rad}(\ainf)$). Thus, $x^l$ is a non--zero divisor both on $\ainf$ and on $\ainf/p$, hence the claim that $p\ainf \cap x^l\ainf=px^l\ainf$ follows for every $l$. The element  $p$ is itself non--zero divisor on $\ainf$ and thus, $p, x$ is a regular sequence. 

To obtain $p^k\ainf \cap x^l\ainf=p^kx^l\ainf$ for general $k$, one can e.g. use induction on $k$ using the fact that $p$ is a non-zero divisor on $\ainf$ (or simply note that one may replace elements in regular sequences by arbitrary positive powers).

To prove the second assertion, note that $\sqrt{(\overline{x})}=\mathfrak{m}_{\mathbb{C}_K^{\flat}}$ since $\ainf/p=\mathcal{O}_{\mathbb{C}_K^{\flat}}$ is a rank $1$ valuation ring. It follows that $(p,W(\mathfrak{m}_{\mathbb{C}_K^{\flat}}))$ is the unique maximal ideal of $\ainf$ above $(p)$, hence the unique maximal ideal since $p \in \mathrm{rad}(\ainf)$, and that $\sqrt{(p, x)}$ is equal to this ideal.
\end{proof} 
 
We are ready to prove the claim mentioned at the beginning of the section.

\begin{cor}\label{FlatTorFree}
Let $F$ be a derived $(p, E(u))$--complete and $(p, E(u))$--completely flat $\ainf$--module, and let $x \in \ainf \setminus (\ainf^\times \cup p\ainf)$. Then $p, x$ is a regular sequence on $F$. In particular, for each $k, l >0$, we have $p^kF \cap x^l F=p^kx^l F$. Consequently, $F$ is a torsion--free $\ainf$--module. 
\end{cor}

\begin{proof}
By Lemma~\ref{disjointness}, $\ainf$ and $F$ are derived $(p, x)$--complete and $F$ is $(p, x)$--completely flat over $\ainf$, and $p, x$ is a regular sequence on $\ainf$. Corollary~\ref{FlatReg} then proves the claim about regular sequence. The sequence $p^k, x^l$ is then also regular on $F$, and the claim $p^kF \cap x^lF=p^kx^lF$ follows. 
To prove the ``consequently'' part, let $y$ be a non--zero and non--unit element of $\ainf$. Since $\ainf$ is classically $p$--complete, we have $\bigcap_n p^n \ainf = 0$, and so there exist $n$ such that $y=p^nx$ with $x \notin p\ainf$. If $x$ is a unit, then $y$ is a non--zero divisor on $F$ since so is $p^n$. Otherwise $x \in \ainf \setminus (\ainf^\times \cup p\ainf)$, so $p, x$ is a regular sequence on $F$, and so is $x, p$ (e.g. by Lemma~\ref{regKoszul}). In particular $p, x$ are both non--zero divisors on $F$, and hence so is $y=p^nx$.
\end{proof}

Finally, we record the following consequence on flatness of $(p, E(u))$--completely flat modules modulo powers of $p$ that seems interesting on its own.

\begin{cor}\label{cor:FlatModp}
Let $x \in \ainf \setminus (\ainf^{\times} \cup p\ainf)$, and let $F$ be a derived $(p, x)$--complete and $(p, x)$--completely (faithfully) flat $\ainf$--module. Then $F$ is derived $p$--complete and $p$--completely (faithfully) flat. In particular, $F/p^nF$ is a flat $\ainf/p^n$--module for every $n>0$.
\end{cor}

\begin{proof}
The fact that $F$ is derived $p$--complete is clear since it is derived $(p, x)$--complete. We need to show that $F/pF$ is a flat $\ainf/p$--module and that $\mathrm{Tor}_i^{\ainf}(F, \ainf/p)=0$ for all $i>0$. The second claim is a consequence of the fact that $p$ is a non--zero divisor on both $\ainf$ and $F$ by Corollary~\ref{FlatTorFree}. For the first claim, note that $\ainf/p=\oh_{\mathbb{C}_K^\flat}$ is a valuation ring and therefore it is enough to show that $F/pF$ is a torsion--free $\oh_{\mathbb{C}_K^\flat}$--module. This follows again by Corollary~\ref{FlatTorFree}.

For the `faithful' version, note that both the statements that $F/pF$ is faithfully flat over $\ainf/p$ and that $F/(p, x)F$ is faithfully flat over $\ainf/(p, x)$ are now equivalent to the statement $F/\mathfrak{m}F \neq 0$ where $\mathfrak{m}=(p, W(\mathfrak{m}_{\mathbb{C}_K^\flat}))$ is the unique maximal ideal of $\ainf$.  
\end{proof}

\subsection{\v{C}ech--Alexander complex}\label{sec:CAComplex}

Next, we discuss the construction of \v{C}ech--Alexander complexes for computing prismatic cohomology, introduced in \cite{BhattScholze} in the affine case, in a global situation. Throughout this section, let $(A, I)$ be a fixed bounded base prism, and let $\mathscr{X}$ be a smooth separated $p$--adic formal scheme over $A/I.$ Recall that $(\mathscr{X}/A)_{\Prism}$ denotes the site whose underlying category is the opposite of the category of bounded prisms $(B, IB)$ over $(A, I)$ together with a map of formal schemes $\spf(B/IB)\rightarrow \mathscr{X}$ over $A/I$. Covers in $(\mathscr{X}/A)_{\Prism}$ are given by the opposites of faithfully flat maps $(B, IB)\rightarrow (C, IC)$ of prisms, meaning that $C$ is $(p, I)$--completely flat over $(B, IB)$. The prismatic cohomology $\R\Gamma_{\Prism}(\mathscr{X}, A)$ is then defined as the sheaf cohomology $\R\Gamma((\mathscr{X}/A)_{\Prism}, \mathcal{O})$($=\R\Gamma((*, \mathcal{O})$ where $*$ is the terminal sheaf) for the sheaf $\mathcal{O}=\mathcal{O}_{\Prism}$ on $(\mathscr{X}/A)_{\Prism}$ defined by $(B, IB)\mapsto B$.

Additionally, let us denote by $\Prism$ the site of all bounded prisms, i.e the opposite of the category of all bounded prisms and their maps, with topology given by faithfully flat maps of prisms.

In order to discuss the  \v{C}ech--Alexander complex in a non-affine situation, a slight modification of the topology on $ (\mathscr{X}/A)_{\Prism}$ is convenient. The following proposition motivates the change.

\begin{prop}\label{prop:DisjointUnions}
Let $(A, I)$ be a bounded prism. 
\begin{enumerate}
\item{Given a collection of maps of (bounded) prisms $(A, I)\rightarrow (B_i, IB_i),$  $i=1, 2, \dots, n,$ the canonical map $(A, I)\rightarrow (C, IC)=\left(\prod_iB_i, I\prod_iB_i\right)$
is a map of (bounded) prisms.}
\item{$(C, IC)$ is flat over $(A, I)$ if and only if each $(B_i, IB_i)$ is flat over $(A, I)$. In that situation, $(C,IC)$ is faithfully flat prism over $(A, I)$ if and only if the family of maps of formal spectra $\spf(B_i/IB_i)\rightarrow \spf(A/I)$ is jointly surjective.}
\item{Let $f \in A$ be an element. Then $(\widehat{A_{f}}, I\widehat{A_f})$, where $\widehat{(-)}$ stands for the derived (equivalently, classical) $(p, I)$--completion, is a bounded prism\footnote{We do consider the zero ring with its zero ideal a prism, hence allow the possibility of $\widehat{A_f}=0$, which occurs e.g. when $f \in (p, I).$ Whether the zero ring satisfies Definition~3.2 of \cite{BhattScholze} depends on whether the inclusion of the empty scheme to itself is considered an effective Cartier divisor; following the usual definitions pedantically, it indeed seems to be the case. Also some related claims, such as \cite[Lemma~3.7 (3)]{BhattScholze} or \cite[Lecture 5, Corollary 5.2]{BhattNotes}, suggest that the zero ring is allowed as a prism.}, and the map $(A, I) \rightarrow (\widehat{A_{f}}, I\widehat{A_f})$ is a flat map of prisms.}
\item{Let $f_1, \dots, f_n \in A$ be a collection of elements generating the unit ideal. Then the canonical map $(A, I)\rightarrow \left(\prod_i\widehat{A_{f_i}}, I\prod_i\widehat{A_{f_i}}\right)$
is a faithfully flat map of (bounded) prisms.}
\end{enumerate}
\end{prop}

\begin{proof}
The proof of (1) is more or less formal. The ring $C=\prod_i B_i$ has a unique $A$--$\delta$--algebra structure since the forgetful functor from $\delta$--rings to rings preserves limits, and $C$ is as product of $(p, I)$--complete rings $(p, I)$--complete. Clearly $IC=\prod_i (IB_i)$ is an invertible ideal since each $IB_i$ is. In particular, $C[I]=0$, hence a prism by \cite[Lemma~3.5]{BhattScholze}. Assuming that all $(B_i, IB_i)$ are bounded, from $C/IC = \prod_i B_i/IB_i$ we have $C/IC[p^\infty]=C/IC[p^k]$ for $k$ big enough so that $B_i/IB_i[p^{\infty}]=B_i/IB_i[p^{k}]$ for all $i$, showing that $(C, IC)$ is bounded.

The ($(p, I)$--complete) flatness part of (2) is clear. For the faithful flatness statement, note that $C/(p, I)C=\prod_i B_i/(p,I)B_i$, hence $A/(p,I)\rightarrow C/(p, I)C$ is faithfully flat if and only if the map of spectra $\coprod_i \spec({B_i/(p, I)B_i})=\spec({C/(p, I)C})\rightarrow \spec({A/(p,I)})$ is surjective.

Let us prove (3). Since $\widehat{A_f}$ has $p\in \mathrm{rad}(\widehat{A_f}),$ the equality $\varphi^n(f^k)=f^{kp^n}+p(\dots)$ shows that $\varphi^n(f^k)$ for each $n, k \geq 0$ is a unit in $\widehat{A_f}$. Consequently, as in \cite[Remark~2.16]{BhattScholze}, $\widehat{A_f}=\widehat{S^{-1}A}$ for $S=\{\varphi^n(f^k)\;|\; n, k \geq 0 \}$, and the latter has a unique $\delta$--structure extending that of $A$ by \cite[Lemmas~2.15 and 2.17]{BhattScholze}. In particular, $\widehat{A_f}$ is a $(p, I)$--completely flat $A$--$\delta$--algebra, hence $(\widehat{A_f}, I\widehat{A_f})$ is flat prism over $(A, I)$ by \cite[Lemma~3.7 (3)]{BhattScholze}.

Part (4) follows formally from parts (1)--(3).
\end{proof}

\begin{constr} Denote by $(\mathscr{X}/A)_{\Prism}^\amalg$ the site whose underlying category is $(\mathscr{X}/A)_{\Prism}$. 
The covers on $(\mathscr{X}/A)_{\Prism}^\amalg$ are given by the opposites of finite families $\{(B, IB) \rightarrow (C_i, IC_i)\}_{i}$ of flat maps of prisms such that the associated maps $\{\spf(C_i/IC_i)\rightarrow \spf(B/IB)\}$ are jointly surjective. Let us call these ``faithfully flat families'' for short. The covers of the initial object $\varnothing$ \footnote{That is, $\varnothing$ corresponds to the zero ring, which we consider to be a prism as per the previous footnote.} are the empty cover and the identity. We similarly extend $\Prism$ to $\Prism^{\amalg}$, that is, we proclaim the identity cover and the empty cover to be covers of $\varnothing$, and generally proclaim (finite) faithfully flat families to be covers.

Clearly isomorphisms as well as composition of covers are covers in both cases. To check that $(\mathscr{X}/A)_{\Prism}^\amalg$ and $\Prism^\amalg$ are sites, it thus remains to check the base change axiom. This is trivial for situations involving $\varnothing,$ so it remains to check that given a faithfully flat family $\{(B, IB)\rightarrow (C_i, IC_i)\}_i$ and a map of prisms $(B, IB) \rightarrow (D, ID)$, the fibre products\footnote{Here we mean fibre products in the variance of the site, i.e. ``pushouts of prisms''. We use the symbol $\boxtimes$ to denote this operation.} $(C_i, IC_i)\boxtimes_{(B, IB)}(D, ID)$ in $\Prism^{\amalg}$ exist and the collection $\{(D, ID)\rightarrow (C_i, IC_i)\boxtimes_{(B, IB)}(D, ID)\}_i$ is a faithfully flat family; the existence and $(p, I)$--complete flatness follows by the same proof as in \cite[Corollary~3.12]{BhattScholze}, only with ``$(p, I)$--completely faithfully flat'' replaced by ``$(p, I)$--completely flat'' throughout, and the fact that the family is faithfully flat follows as well, since $\left(\prod_i(C_i, IC_i)\right)\boxtimes_{(B, IB)}(D, ID)=\prod_i \left( (C_i, IC_i)\boxtimes_{(B, IB)}(D, ID)\right)$ (and using Remark~\ref{rem:CompareSites} (1) below).
\end{constr}
\begin{rem}\label{rem:CompareSites}
\begin{enumerate}[(1)]
\item{Note that for a finite family of objects $(C_i, IC_i)$ in $(\mathscr{X}/A)_{\Prism},$  the structure map of the product $(A, I)\rightarrow \prod_i(C_i, IC_i)$ together with the map of formal spectra (induced from the maps for individual $i$'s) $$\spf(\prod_i C_i/IC_i)=\coprod_i \spf(C_i/IC_i)\rightarrow \mathscr{X}$$ makes $(\prod_i C_i, I\prod_i C_i)$ into an object of $(\mathscr{X}/A)_{\Prism}$ that is easily seen to be the coproduct of $(C_i, IC_i)$'s. In view of Proposition~\ref{prop:DisjointUnions} (2), one thus arrives at the equivalent formulation
$$
\{Y_i \rightarrow Z\}_{i}\text{ is a }(\mathscr{X}/A)_{\Prism}^\amalg\text{--cover }\Leftrightarrow 
\coprod_iY_i \rightarrow Z\text{ is a }(\mathscr{X}/A)_{\Prism}\text{--cover.}
$$
That is, $(\mathscr{X}/A)_{\Prism}^{\amalg}$ is the (finitely) superextensive site having covers of $(\mathscr{X}/A)_{\Prism}$ as singleton covers. (Similar considerations apply to $\Prism$ and $\Prism^{\amalg}$.)}
\item{The two sites are honestly different in that they define different categories of sheaves. Namely, for every finite coproduct $Y=\coprod_i Y_i$, the collection of canonical maps $\{Y_i \rightarrow \coprod_i Y_i\}_i$ forms a $(\mathscr{X}/A)_{\Prism}^\amalg$--cover, and the sheaf axiom forces upon $\mathcal{F}\in \shv((\mathscr{X}/A)_{\Prism}^\amalg)$ the identity $\mathcal{F}\left(\coprod_i Y_i\right)=\prod_i \mathcal{F}(Y_i),$ which is not automatic\footnote{For example, every constant presheaf is a sheaf for a topology given by singleton covers only, which is not the case for $(\mathscr{X}/A)_{\Prism}^{\amalg}.$}. In fact, $ \shv((\mathscr{X}/A)_{\Prism}^\amalg)$ can be identified with the full category of  $\shv((\mathscr{X}/A)_{\Prism})$ consisting of all sheaves compatible with finite disjoint unions in the sense above. In particular, the structure sheaf $\mathcal{O}=\mathcal{O}_{\Prism}: (B, IB)\mapsto B$ is a sheaf for the $(\mathscr{X}/A)_{\Prism}^\amalg$--topology. (Again, the same is true for $\Prism$ and $\Prism^{\amalg}$, including the fact that $\mathcal{O}: (B, IB)\mapsto B$ is a sheaf.)}
\end{enumerate}\end{rem}

 Despite the above fine distinction, for the purposes of prismatic cohomology, the two topologies are interchangeable. This is a consequence of the following lemma.

\begin{lem}\label{VanishingObjects}
Given an object $(B, IB) \in (\mathscr{X}/A)_{\Prism}^\amalg,$ one has $\H^i((B, IB), \mathcal{O})=0$ for $i>0$.
\end{lem}

\begin{proof}
The sheaf $\mathcal{O}: (B, I) \mapsto B$ on $\Prism^\amalg$ has vanishing positive \v{C}ech cohomology essentially by the proof of \cite[Corollary~3.12]{BhattScholze}: one needs to show acyclicity of the \v{C}ech complex for any $\Prism^\amalg$--cover $\{(B, I)\rightarrow (C_i, IC_i)\}_i,$ but the resulting \v{C}ech complex is identical to that for the $\Prism$--cover $(B, I)\rightarrow \prod_i(C_i, IC_i)$, for which the acyclicity is proved in \cite[Corollary~3.12]{BhattScholze}. By a general result (e.g. \cite[03F9]{stacks}), this implies vanishing of $\H^i_{\Prism^\amalg}((B, I), \mathcal{O})$ for all bounded prisms $(B, I)$ and all $i>0$. 

Now we make use of the fact that cohomology of an object can be computed as the cohomology of the corresponding slice site, \cite[03F3]{stacks}. Let $(B, IB)\in (\mathscr{X}/A)_{\Prism}^\amalg.$ After forgetting structure, we may view $(B, IB)$ as an object of $\Prism^\amalg$ as well, and then \cite[03F3]{stacks} implies that for every $i,$ we have the isomorphisms
\begin{align*}
\H^i_{(\mathscr{X}/A)_{\Prism}^\amalg}((B, IB), \mathcal{O}) & \simeq \H^i((\mathscr{X}/A)_{\Prism}^\amalg/(B, IB), \mathcal{O}|_{(B, IB)}), \\ 
\H^i_{\Prism^\amalg}((B, IB), \mathcal{O}) & \simeq \H^i((\Prism^\amalg/(B, IB), \mathcal{O}|_{(B, IB)})
\end{align*}
(where $\mathcal{C}/c$ for a site $\mathcal{C}$ and $c \in \mathcal{C}$ denotes the slice site). Upon noting that the slice sites $(\mathscr{X}/A)_{\Prism}^\amalg/(B, IB),$ $\Prism^\amalg/(B, IB)$ are equivalent sites (in a manner that identifies the two versions of the sheaf $\mathcal{O}|_{(B, IB)}$), the claim follows.
\end{proof}

\begin{cor}\label{cor:CohomologySame}
One has $$\R\Gamma((\mathscr{X}/A)_{\Prism}, \mathcal{O}) = \R\Gamma((\mathscr{X}/A)_{\Prism}^\amalg, \mathcal{O}).$$
\end{cor}

\begin{proof}
The coverings of $(\mathscr{X}/A)_{\Prism}^\amalg$ contain the coverings of $(\mathscr{X}/A)_{\Prism},$ so we are in the situation of \cite[0EWK]{stacks}, namely, there is a morphism of sites $\varepsilon:(\mathscr{X}/A)_{\Prism}^\amalg \rightarrow (\mathscr{X}/A)_{\Prism}$ given by the identity functor of the underlying categories, the pushforward functor $\varepsilon_*: \shv((\mathscr{X}/A)_{\Prism}^\amalg) \rightarrow \shv((\mathscr{X}/A)_{\Prism})$ being the natural inclusion and the (exact) inverse image functor $\varepsilon^{-1}: \shv((\mathscr{X}/A)_{\Prism}) \rightarrow \shv((\mathscr{X}/A)_{\Prism}^\amalg)$ is the sheafification with respect to the ``$^\amalg$''-topology. One has $$\Gamma((\mathscr{X}/A)^\amalg,-)=\Gamma((\mathscr{X}/A),-)\circ \varepsilon_*$$
(where $\varepsilon_*$ denotes the inclusion of abelian sheaves in this context), hence
$$\R\Gamma((\mathscr{X}/A)^\amalg,\mathcal{O})=\R\Gamma((\mathscr{X}/A),\R\varepsilon_*\mathcal{O}),$$ and to conclude it is enough to show that $\R^i\varepsilon_* \mathcal{O}=0$ $\forall i>0$. But $\mathsf{R}^i\varepsilon_* \mathcal{O}$ is the sheafification of the presheaf given by $(B, IB) \mapsto \H^i((B, IB), \mathcal{O})$ (\cite[072W]{stacks}), which is $0$ by Lemma~\ref{VanishingObjects}. Thus, $\mathsf{R}^i\varepsilon_* \mathcal{O}=0$, which proves the claim. 
\end{proof}

For an open $p$--adic formal subscheme $\mathscr{V} \subseteq \mathscr{X}$, denote by $h_{\mathscr{V}}$ the functor sending $(B, IB) \in (\mathscr{X}/A)_{\Prism}$ to the set of factorizations of the implicit map $\spf(B/IB) \rightarrow \mathscr{X}$ through $\mathscr{V} \hookrightarrow \mathscr{X};$ that is, 
$$h_\mathscr{V}((B, IB))=\begin{cases}*\;\; \text{ if the image of }\spf(B/IB) \rightarrow \mathscr{X} \text{ is contained in }\mathscr{V},\\ \emptyset\;\; \text{ otherwise.}\end{cases}$$
Let $(B, IB)\rightarrow (C, IC)$ correspond to a morphism in $(\mathscr{X}/A)_{\Prism}$. If $\spf(B/IB)\rightarrow \mathscr{X}$ factors through $\mathscr{V},$ then so does $\spf(C/IC)\rightarrow \spf(B/IB)\rightarrow \mathscr{X}$. It follows that $h_\mathscr{V}$ forms a presheaf on $(\mathscr{X}/A)_{\Prism}$ (with transition maps $h_\mathscr{V}((B, IB))\rightarrow h_\mathscr{V}((C, IC))$ given by $* \mapsto *$ when $h_\mathscr{V}((B, IB)) \neq \emptyset$, and the empty map otherwise). Note that $h_{\mathscr{X}}$ is the terminal sheaf.

\begin{prop}
$h_\mathscr{V}$ is a sheaf on $(\mathscr{X}/A)_{\Prism}^\amalg$.
\end{prop}

\begin{proof}
Consider a cover in $(\mathscr{X}/A)_{\Prism}^{\amalg},$ which is given by a faithully flat family $\{(B, IB)\rightarrow (C_i, IC_i)\}_i$. One needs to check that the sequence
$$h_\mathscr{V}((B, IB)))\rightarrow \prod_i h_\mathscr{V}((C_i, IC_i)) \rightrightarrows \prod_{i,j} h_\mathscr{V}((C_i, IC_i)\boxtimes_{(B, IB)}(C_j, IC_j))$$
is an equalizer sequence. All the terms have at most one element; consequently, there are just two cases to consider, depending on whether the middle term is empty or not. In both cases, the pair of maps on the right necessarily agree, and so one needs to see that the map on the left is an isomorphism. This is clear in the case when the middle term is empty (since the only map into an empty set is an isomorphism). It remains to consider the case when the middle term is nonempty, which means that $h_\mathscr{V}((C_i, IC_i))=*$ for all $i$. In this case we need to show that $h_{\mathscr{V}}((B, IB)) =*$. Since the maps $\spf(C_i/IC_i)\rightarrow \spf(B/IB) $ are jointly surjective and each $\spf(C_i/IC_i) \rightarrow \mathscr{X}$ lands in $\mathscr{V}$, it follows that so does the map $\spf(B/IB) \rightarrow \mathscr{X}$. Thus, $h_\mathscr{V}((B, IB))=*$, which finishes the proof.
\end{proof}

When $\mathscr{V}$ is affine, one can cover the sheaf $h_{\mathscr{V}}$ by a representable sheaf. Note that the construction of the representing object is essentially equivalent to Construction~4.17 of \cite{BhattScholze}.

\begin{constr}[\v{C}ech--Alexander cover of $\mathscr{V}$]\label{constACcover}
Let us additionally assume that $\mathscr{V}=\spf(R)$ is affine. Choose a surjection $P_{\mathscr{V}} \rightarrow R$ where $P_{\mathscr{V}}=\widehat{A[\underline{X}]}$ is a $(p, I)$--completed free $A$--algebra. Denote by $J_{\mathscr{V}}$ the kernel of the surjection. Then there is a commutative diagram with exact rows
\begin{center}
\begin{tikzcd}
0 \ar[r] & J_{\mathscr{V}} \ar[r] \ar[d] & P_{\mathscr{V}} \ar[d] \ar[r] & R\ar[d] \ar[r] & 0 \\
& \widehat{J_{\mathscr{V}}P_{\mathscr{V}}^{\delta}} \ar[r] & \widehat{P_{\mathscr{V}}^{\delta}} \ar[r] & \widehat{R\otimes_{P_{\mathscr{V}}}P_{\mathscr{V}}^{\delta}} \ar[r] & 0, 
\end{tikzcd}
\end{center}
where 
$\widehat{(-)}$ stands for derived $(p, I)$--completion. Here for an $A$--algebra $S$, $S^{\delta}$ denotes the ``$\delta$--envelope'' of $S$, that is, the $S$--algebra initial among $S$--algebras endowed with an $A$--$\delta$--algebra structure. Note that $\widehat{P_{\mathscr{V}}^{\delta}}=\widehat{(P^0_{\mathscr{V}})^{\delta}}$, where $P^{0}_{\mathscr{V}}=A[\underline{X}]$ is the polynomial algebra before completion; in particular, since $(P^{0}_{\mathscr{V}})^{\delta}$ is a flat $P^0_{\mathscr{V}}$--algebra (essentially by \cite[Lemma~2.11]{BhattScholze}), it follows that $\widehat{P_{\mathscr{V}}^{\delta}}$ is $(p, I)$--completely flat $P_{\mathscr{V}}$--algebra. Consequently, the completions in the lower row of the diagram can be equivalently taken as classical $(p, I)$-completions (cf. \cite[Lemma~3.7]{BhattScholze}). 

Denote by $J_{\mathscr{V}}^{\delta, \wedge} \subseteq \widehat{P_{\mathscr{V}}^{\delta}}$ the image of the map $\widehat{J_{\mathscr{V}}P_{\mathscr{V}}^{\delta}} \rightarrow \widehat{P_{\mathscr{V}}^{\delta}},$ i.e. the $(p, I)$--complete ideal of $\widehat{P_{\mathscr{V}}^{\delta}}$ topologically generated by $J_{\mathscr{V}}$. Then we have a short exact sequence
\begin{center}
\begin{tikzcd}
0 \ar[r] & J_{\mathscr{V}}^{\delta, \wedge} \ar[r]  & \widehat{P_{\mathscr{V}}^{\delta}} \ar[r] & \widehat{R\otimes_{P_{\mathscr{V}}}P_{\mathscr{V}}^{\delta}} \ar[r] & 0. 
\end{tikzcd}
\end{center}

Let $(\check{C}_\mathscr{V}, I\check{C}_\mathscr{V})$ be the prismatic envelope of $(\widehat{P^{\delta}_\mathscr{V}}, J_\mathscr{V}^{\delta, \wedge})$. It follows from \cite[Proposition 3.13, Example 3.14]{BhattScholze} that $(\check{C}_{\mathscr{V}}, I\check{C}_{\mathscr{V}})$ exists and is given by a flat prism over $(A, I)$. The map
$$R\rightarrow \widehat{R\otimes_{P_{\mathscr{V}}}P^{\delta}_\mathscr{V}}=\widehat{P^{\delta}_\mathscr{V}}/J_\mathscr{V}^{\delta, \wedge} \rightarrow \check{C}_{\mathscr{V}}/I\check{C}_{\mathscr{V}}$$ 
of $p$--complete rings corresponds to the map of formal schemes $\spf(\check{C}_{\mathscr{V}}/I\check{C}_{\mathscr{V}}) \rightarrow \mathscr{V}\hookrightarrow \mathscr{X}$. This defines an object of $(\mathscr{X}/A)_{\Prism}^\amalg,$ which we call a \emph{\v{C}ech--Alexander cover of $\mathscr{V}$}. 
\end{constr}

\begin{rems}\label{rem:NonCompletedEnvelopes}\mbox{}\\[-0.9cm]
\begin{enumerate}[(1)]
\item{Note that
 $(\check{C}_{\mathscr{V}}, I\check{C}_{\mathscr{V}})$ is equivalently the prismatic envelope of $(\widehat{P^{\delta}_\mathscr{V}}, J_\mathscr{V}\widehat{P^{\delta}_\mathscr{V}})$. Moreover, when the ideal $J_{\mathscr{V}}$ is finitely generated, one has the equality $J_{\mathscr{V}}^{\delta, \wedge}=J_{\mathscr{V}}\widehat{P_{\mathscr{V}}^{\delta}}.$}
\item{Since the ring $R$ in Construction~\ref{constACcover} is a $p$-completely smooth $A/I$--algebra, it is in particular a $p$--completion of a finitely presented $A/I$--algebra. It follows that the map $P_{\mathscr{V}}\rightarrow R$ may be chosen so that $P_{\mathscr{V}}$ is the (derived) $(p, I)$--completion of a polynomial $A$--algebra of finite type, with the kernel $J_{\mathscr{V}}$ finitely generated. While such a choice may be preferable, we formulate the construction without imposing it, as it may be convenient to allow non--finite--type free algebras in the construction e.g. for the reasons of functoriality (see the remark at the end of \cite[Construction~4.17]{BhattScholze}).}
\end{enumerate}
\end{rems}

\begin{prop}\label{ACechCovers}
Denote by $h_{\check{C}_{\mathscr{V}}}$ the sheaf represented by the object $(\check{C}_\mathscr{V}, I\check{C}_\mathscr{V})\in (\mathscr{X}/A)_{\Prism}^\amalg$. 
There exists a unique map of sheaves $h_{\check{C}_{\mathscr{V}}} \rightarrow h_\mathscr{V}$, and it is an epimorphism.
\end{prop}

\begin{proof}
If $(B, IB) \in (\mathscr{X}/A)_{\Prism}$ with $h_{\check{C}_{\mathscr{V}}}((B, IB))\neq \emptyset,$ this means that $\spf(B/IB)\rightarrow \mathscr{X}$ factors through $\mathscr{V}$ since it factors through $\spf(\check{C}_{\mathscr{V}}/I\check{C}_{\mathscr{V}}).$ Thus, we also have $h_{\mathscr{V}}((B, IB))=*$, and so the (necessarily unique) map $h_{\check{C}_{\mathscr{V}}}((B, IB))\rightarrow h_{\mathscr{V}}((B, IB))$ is defined. When $h_{\check{C}_{\mathscr{V}}}((B, IB))$ is empty, the map $h_{\check{C}_{\mathscr{V}}}((B, IB))\rightarrow h_{\mathscr{V}}((B, IB))$ is still defined and unique, namely given by the empty map. Thus, the claimed morphism of sheaves exists and is unique.

We show that this map is an epimorphism. Let $(B, IB)\in (\mathscr{X}/A)_{\Prism}$ such that $h_{\mathscr{V}}((B, IB)) = *$, i.e. $\spf(B/IB) \rightarrow \mathscr{X}$ factors through $\mathscr{V}$, and consider the map $R \rightarrow B/IB$ associated to the map $\spf(B/IB) \rightarrow \mathscr{V}$. Since $P_{\mathscr{V}}$ is a $p$--completed free $A$--algebra surjecting onto $R$ and $B$ is $(p, I)$--complete, the map $R \rightarrow B/IB$ admits a lift $P_{\mathscr{V}} \rightarrow B$. This induces an $A$--$\delta$--algebra map $\widehat{P_{\mathscr{V}}^{\delta}}\rightarrow B$ which gives a morphism of $\delta$--pairs $(\widehat{P_{\mathscr{V}}^{\delta}}, J_{\mathscr{V}}\widehat{P_{\mathscr{V}}^{\delta}})\rightarrow (B, IB)$, and further the map of prisms $(\check{C}_\mathscr{V}, I\check{C}_\mathscr{V}) \rightarrow (B, IB)$ using the universal properties of objects involved. It is easy to see that this is indeed (the opposite of) a morphism in $(\mathscr{X}/A)_{\Prism}$. This shows that $h_{\check{C}_{\mathscr{V}}}((B, IB))$ is nonempty whenever $h_{\mathscr{V}}((B, IB))$ is. Thus, the map is an epimorphism. 
\end{proof}

Let $\mathfrak{V}=\{\mathscr{V}_j\}_{j \in J}$ be an affine open cover of $\mathscr{X}$. For $n \geq 1$ and a multi--index $(j_1, j_2, \dots, j_n) \in J^n,$ denote by $\mathscr{V}_{j_1, \dots j_n}$ the intersection $\mathscr{V}_{j_1}\cap \dots \cap \mathscr{V}_{j_n}$. As $\mathscr{X}$ is assumed to be separated, each $\mathscr{V}_{j_1, \dots j_n}$ is affine and we write $\mathscr{V}_{j_1, \dots j_n}=\mathrm{Spf}(R_{j_1, \dots, j_n})$. 

\begin{rem}[Binary products in $(\mathscr{X}/A)_{\Prism}$] \label{products}
For $(B,IB), (C, IC)\in (\mathscr{X}/A)_{\Prism}$, let us denote their binary product by $(B, IB)\boxtimes (C, IC)$. Let us describe it explicitly at least under the additional assumptions that 
\begin{enumerate}[(1)]
\item{$(B,IB), (C, IC)$ are flat prisms over $(A, I),$}
\item{there are affine opens $\mathscr{U}, \mathscr{V} \subseteq \mathscr{X}$ such that $h_{\mathscr{U}}((B, IB))=*=h_{\mathscr{V}}((C, IC))$.}
\end{enumerate}
Set $\mathscr{W}=\mathscr{U} \cap \mathscr{V}$ and denote the rings corresponding to the affine open sets $\mathscr{U}, \mathscr{V}$ and $\mathscr{W}$ by $R, S$ and $ T,$ resp. Then any object $(D, ID)\in (\mathscr{X}/A)_{\Prism}$ with maps both to $(B, IB)$ and $(C, IC)$ lives over $\mathscr{W},$ i.e. satisfies $h_{\mathscr{W}}((D, ID))=*$. This justifies the following construction. Consider the following commutative diagram, where $\urcorner$ denotes the pushout of $p$--complete commutative rings, i.e. taking the classically $p$--completed tensor product $\widehat{\otimes}$ (and $B\widehat{\otimes}_AC$ is the derived, but equivalently classical, $(p, I)$--completion of $B\otimes_AC$):

\begin{center}
\begin{tikzcd}[row sep = small]
&& B \widehat{\otimes}_A C \ar[dd] && \\
B \ar[urr]\ar[dd]&& 
&& \ar[ull] C\ar[dd]\\
&& (B/IB\widehat{\otimes}_{R}T)\widehat{\otimes}_{T}(C/IC\widehat{\otimes}_{S}T) \ar[d, phantom, "\cornerdown", near start] && 
\\
B/IB \ar[r] 
& B/IB\widehat{\otimes}_{R}T\ar[ur]  & {\color{white} A} &\ar[ul] C/IC\widehat{\otimes}_{S}T & \ar[l]
C/IC \\
R\ar[u]\ar[r] \ar[ur, phantom, "\llcorner", very near end] & T\ar[u] \ar[rr, equal] & & T \ar[u]& \ar[l] \ar[ul, phantom, "\lrcorner", very near end] S \ar[u]
\end{tikzcd}
\end{center}
Let $J \subseteq B \widehat{\otimes}_A C$ be the kernel of the map
$B \widehat{\otimes}_A C \rightarrow (B/IB\widehat{\otimes}_{R} T) \widehat{\otimes}_{T} (C/IC\widehat{\otimes}_{S} T).$
Then the product $(B, IB)\boxtimes (C, IC)$ is given by the prismatic envelope of the $\delta$--pair $(B \widehat{\otimes}_A C, J)$.
\end{rem}

\begin{prop}\label{prop:ProductsInPrismaticSite}
The \v{C}ech--Alexander covers can be chosen so that for all $j_1, \dots, j_n $ we have 
$$(\check{C}_{\mathscr{V}_{j_1, \dots j_n}}, I\check{C}_{\mathscr{V}_{j_1, \dots j_n}})=(\check{C}_{\mathscr{V}_{j_1}}, I\check{C}_{\mathscr{V}_{j_1}})\boxtimes (\check{C}_{\mathscr{V}_{j_2}}, I\check{C}_{\mathscr{V}_{j_2}}) \boxtimes \dots \boxtimes (\check{C}_{\mathscr{V}_{j_n}}, I\check{C}_{\mathscr{V}_{j_n}}).$$ 
\end{prop}

\begin{proof}
Clearly it is enough to show the statement for binary products. More precisely,  given  two affine opens $\mathscr{V}_1, \mathscr{V}_2 \subseteq \mathscr{X}$ and an arbitrary initial choice of $(\check{C}_{\mathscr{V}_{1}}, I\check{C}_{\mathscr{V}_{1}})$ and $(\check{C}_{\mathscr{V}_{2}}, I\check{C}_{\mathscr{V}_{2}}),$ we show that $P_{\mathscr{V}_{12}} \rightarrow R_{12}$ can be chosen so that the resulting \v{C}ech--Alexander cover $(\check{C}_{\mathscr{V}_{12}}, I\check{C}_{\mathscr{V}_{12}})$ of $\mathscr{V}_{12}$ is equal to $(\check{C}_{\mathscr{V}_{1}}, I\check{C}_{\mathscr{V}_{1}}) \boxtimes (\check{C}_{\mathscr{V}_{2}}, I\check{C}_{\mathscr{V}_{2}})$. For the purposes of this proof, let us refer to a prismatic envelope of a $\delta$--pair $(S, J)$ also as ``the prismatic envelope of the arrow $S \rightarrow S/J$''.

Consider $\alpha_i:P_{\mathscr{V}_i}\twoheadrightarrow R_i,\; i=1, 2$ as in Construction~\ref{constACcover}, and set $P_{\mathscr{V}_{12}}=P_{\mathscr{V}_1}\widehat{\otimes}_A P_{\mathscr{V}_2}$. Then one has the induced surjection $\alpha_1 \otimes \alpha_2:P_{\mathscr{V}_{12}} \rightarrow R_1\widehat{\otimes}_{A/I} R_2$, which can be followed by the induced map $R_1\widehat{\otimes}_{A/I} R_2\rightarrow R_{12}$. This latter map is surjective as well since $\mathscr{X}$ is separated, and therefore the composition of these two maps $\alpha_{12}:P_{\mathscr{V}_{12}}\rightarrow R_{12}$ is surjective, with the kernel $J_{\mathscr{V}_{12}}$ that contains $(J_{\mathscr{V}_1}, J_{\mathscr{V}_2})P_{\mathscr{V}_{12}}$. We may construct a diagram analogous to the one from Remark~\ref{products}, which becomes the diagram
\begin{center}
\begin{tikzcd}[row sep = small]
&& \widehat{P_{\mathscr{V}_{12}}^{\delta}} \ar[dd] && \\
\widehat{P_{\mathscr{V}_1}^{\delta}} \ar[urr]\ar[dd]&& 
&& \ar[ull] 
\widehat{P_{\mathscr{V}_2}^{\delta}}\ar[dd]\\
&& R_{12}\widehat{\otimes}_{P_{\mathscr{V}_{12}}}(\widehat{P_{\mathscr{V}_{12}}^{\delta}}) \ar[d, phantom, "\cornerdown", near start] && \\
R_1\widehat{\otimes}_{P_{\mathscr{V}_1}}\widehat{P_{\mathscr{V}_1}^{\delta}}  
\ar[r]& R_{12}\widehat{\otimes}_{P_{\mathscr{V}_1}}\widehat{P_{\mathscr{V}_1}^{\delta}}\ar[ur]  & {\color{white} A} &\ar[ul] R_{12}\widehat{\otimes}_{P_{\mathscr{V}_2}}\widehat{P_{\mathscr{V}_2}^{\delta}} & \ar[l] 
R_2\widehat{\otimes}_{P_{\mathscr{V}_2}}\widehat{P_{\mathscr{V}_2}^{\delta}} \\
R_1\ar[u]\ar[r] \ar[ur, phantom, "\llcorner", very near end] & R_{12}\ar[u] \ar[rr, equal] & & R_{12} \ar[u]& \ar[l] \ar[ul, phantom, "\lrcorner", very near end] R_2, \ar[u]
\end{tikzcd}
\end{center}
where the expected arrow in the central column is replaced by an isomorphic one, namely the map obtained from the surjection $P_{\mathscr{V}_{12}} \rightarrow R_{12}$ by the procedure as in Construction~\ref{constACcover}. Now $(\check{C}_{\mathscr{V}_{12}}, I\check{C}_{\mathscr{V}_{12}})$ is obtained as the prismatic envelope of this composed central arrow, while $(\check{C}_{\mathscr{V}_{1}}, I\check{C}_{\mathscr{V}_{1}})\boxtimes (\check{C}_{\mathscr{V}_{2}}, I\check{C}_{\mathscr{V}_{2}})$ is obtained the same way, but only after replacing the downward arrows on the left and right by their prismatic envelopes. Comparing universal properties, one easily sees that the resulting central prismatic envelope remains unchanged, proving the claim.
\end{proof}

\begin{rem}\label{rem:ComplFinPresented}
Suppose that for each $j$, the initial choice of the map $P_{\mathscr{V}_j} \rightarrow R_j$ has been made as in Remark~\ref{rem:NonCompletedEnvelopes} (2), that is, $P_{\mathscr{V}_j}$ is the $(p, I)$--completion of a finite type free $A$--algebra and the ideal $J_{\mathscr{V}_j}$ is finitely generated. If now $P_{\mathscr{V}_{j_1, j_2, \dots j_n}}$ is the $(p, I)$--completed free $A$--algebra for $\mathscr{V}_{j_1, j_2, \dots j_n}$ obtained by iterating the procedure in the proof of Proposition~\ref{prop:ProductsInPrismaticSite}, it is easy to see that in this case, the algebra $P_{\mathscr{V}_{j_1, j_2, \dots j_n}}$ is still the $(p, I)$--completion of a finite type free $A$--algebra, and it can be shown that the corresponding ideal $J_{\mathscr{V}_{j_1, j_2, \dots j_n}}$ is finitely generated.

In more detail, given a ring $B$ and a finitely generated ideal $J \subseteq B$, Let us call a $B$--algebra $C$ \textit{$J$--completely finitely presented} if $C$ is derived $J$--complete and there exists a map $\alpha:B[\underline{X}]\rightarrow C$ from the polynomial ring in finitely many variables $\underline{X}=\{X_1, \dots, X_n\}$ such that the derived $J$--completed map $\widehat{\alpha}:\widehat{B[\underline{X}]}\rightarrow C$ is surjective and with a finitely generated kernel. Then the algebra $R_{j_1, j_2, \dots j_n}$ corresponding to $\mathscr{V}_{j_1, j_2, \dots j_n}$ is $(p, I)$--completely finitely presented by Remark~\ref{rem:NonCompletedEnvelopes} (2), and since $P_{\mathscr{V}_{j_1, j_2, \dots j_n}}$ is the $(p, I)$--completion of a finite type polynomial $A$--algebra, the following lemma shows that $J_{\mathscr{V}_{j_1, j_2, \dots j_n}}$ is finitely generated.
\end{rem} 

\begin{lem}
Let $C$ be a $J$--completely finitely presented $B$--algebra, and consider a $B$--algebra map $\beta:B[\underline{Y}] \rightarrow C$ from a polynomial algebra in finitely many variables $\underline{Y}=\{Y_1, \dots, Y_m\}$ such that $\widehat{\beta}$ is surjective. Then the kernel of $\widehat{\beta}$ is finitely generated. 
\end{lem}

\begin{proof}
The proof is an adaptation of the proof of \cite[00R2]{stacks}, which is a similar assertion about finitely presented algebras. Consider $\alpha$ as in Remark~\ref{rem:ComplFinPresented}, and additionally let us fix a generating set $(f_1, f_2, \dots, f_k)\subseteq \widehat{B[\underline{X}]}$ of $\mathrm{Ker}\,\widehat{\alpha}$. 

For $i=1, \dots, m,$ let us choose $g_i \in \widehat{B[\underline{X}]}$ such that $\widehat{\alpha}(g_i)=\beta(Y_i)$. Then one can define a surjective map
$$\theta_0:\widehat{B[\underline{X}]}[\underline{Y}]\rightarrow C, \;\;\; \theta_0\mid_{\widehat{B[\underline{X}]}}=\widehat{\alpha}, \;\;\theta_0(Y_i)=\beta(Y_i),$$
and it is easy to see that $\mathrm{Ker}\,\theta_0=(f_1, \dots, f_k, Y_1-g_1, \dots, Y_m-g_m).$
That is, we have an exact sequence 
$$(\widehat{B[\underline{X}]}[\underline{Y}])^{\oplus k+m}\rightarrow \widehat{B[\underline{X}]}[\underline{Y}] \stackrel{\theta_0}\rightarrow C\rightarrow 0,$$
where the map on the left is a module map determined by the finite set of generators of $\mathrm{Ker}\,\theta_0$. After taking the derived $J$--completion, the sequence becomes the exact sequence
$$\widehat{B[\underline{X},\underline{Y}]}^{\oplus k+m}\rightarrow \widehat{B[\underline{X},\underline{Y}]} \stackrel{\theta}\rightarrow C\rightarrow 0.$$
That is, we have a surjective map $\theta:\widehat{B[\underline{X},\underline{Y}]} \rightarrow C$, which is determined on topological generators by $\theta(X_j)=\alpha(X_j), \theta(Y_i)=\beta(Y_i),$ and the kernel of $\theta$ is $(f_1, \dots, f_k, Y_1-g_1, \dots, Y_m-g_m)$.

Next, we choose elements $h_j \in \widehat{B[\underline{Y}]}$ such that $\widehat{\beta}(h_j)=\alpha(X_j)$ for each $j$. Then we have a surjective map 
$\psi: \widehat{B[\underline{X}, \underline{Y}]}\rightarrow \widehat{B[\underline{Y}]}$ given by $X_j \mapsto h_j$ and $Y_i \mapsto Y_i$,
which has the property that $ \widehat{\beta} \circ \psi=\theta$. That is, $$\mathrm{Ker}\,\theta=\mathrm{Ker}\,(\widehat{\beta} \circ \psi)=\psi^{-1}(\mathrm{Ker}\,(\widehat{\beta})),$$
and therefore $\psi(\mathrm{Ker}\,\theta)=\mathrm{Ker}\,\widehat{\beta}$ since $\psi$ is surjective. But $\mathrm{Ker}\,\theta$ is finitely generated by the previous, and hence so is  $\mathrm{Ker}\,\widehat{\beta}$.\end{proof}

\begin{prop}\label{prop:Cover}
The map $\coprod_j h_{\mathscr{V}_j} \rightarrow h_\mathscr{X}=*$ (where $\coprod$ denotes the coproduct in $\shv((\mathscr{X}/A)_{\Prism}^\amalg)$) to the final object is an epimorphism, hence so is the map $\coprod_j h_{\check{C}_{\mathscr{V}_j}} \rightarrow *$. 
\end{prop}

\begin{proof}

It is enough to show that for a given object $(B, IB)\in (\mathscr{X}/A)_{\Prism}^\amalg,$ there is a faithfully flat family $(B, IB) \rightarrow (C_i, IC_i)$ in $ (\mathscr{X}/A)_{\Prism}^{\amalg, \mathrm{op}}$ such that $\coprod^{\mathrm{pre}}_j h_{\mathscr{V}_j}((C_i, IC_i)) \neq \emptyset$ for all $i$ where $\coprod^{\mathrm{pre}}$ denotes the coproduct of presheaves.

With that aim, let us first consider the preimages $\mathscr{W}_j \subseteq \spf(B/IB)$ of each $\mathscr{V}_j $ under the map $\spf(B/IB)\rightarrow \mathscr{X}$. This is an open cover of $\spf(B/IB)$ that corresponds to an open cover of $\spec B/(p, I)B$. One can then choose $f_1, f_2, \dots, f_m$ such that $\{\spec (B/(p, I)B)_{f_i}\}_i$ refines this cover, i.e. every $\spec (B/(p, I)B)_{f_i}$ corresponds to an open subset of $\mathscr{W}_{j(i)}$ for some index $j(i)$. 

The elements $f_1, \dots, f_m$ generate the unit ideal of $B$ since they do so modulo $(p, I)$ which is contained in $\mathrm{rad}(B).$ Thus, the family 
$$(B, IB) \rightarrow (C_i, IC_i):=(\widehat{B_{f_i}}, I \widehat{B_{f_i}})\;\; i=1, 2, \dots, m$$
is easily seen to give the desired faithfully flat family, with each $\coprod^{\mathrm{pre}}_j h_{\mathscr{V}_j}((C_i, IC_i))$ nonempty, since each $\spf(C_i/IC_i) \rightarrow \mathscr{X}$ factors through $\mathscr{V}_{j(i)}$ by construction.
\end{proof}

\begin{rem}
The proof of Proposition~\ref{prop:Cover} is the one step where we used the relaxation of the topology, namely the fact that the faithfully flat cover $(B, IB)\rightarrow \prod_i(C_i, IC_i)$ can be replaced by the family $\{(B, IB)\rightarrow (C_i, IC_i)\}_i$.
\end{rem}

Finally, we obtain the \v{C}ech--Alexander complexes in the global case. 

\begin{prop} \label{prop:CechComplex}
With the notation for $\mathscr{V}_{j_1, j_2, \dots, j_n}$ as above and the choice of \v{C}ech--Alexander covers $\check{C}_{\mathscr{V}_{j_1, j_2, \dots, j_n}}$ as in Proposition~\ref{prop:ProductsInPrismaticSite}, $\R\Gamma((\mathscr{X}/A)_{\Prism}, \mathcal{O})$ is modelled by the \v{C}ech--Alexander complex 

\begin{equation}\tag{$\check{C}^\bullet_{\mathfrak{V}}$}\label{eqn:CechComplex}
0 \longrightarrow \prod_j \check{C}_{\mathscr{V}_j}\longrightarrow \prod_{j_1, j_2} \check{C}_{\mathscr{V}_{j_1, j_2}} \longrightarrow \prod_{j_1, j_2, j_3} \check{C}_{\mathscr{V}_{j_1, j_2, j_3}} \longrightarrow\dots
\end{equation}
(that is, the complex associated to the cosimplicial ring $(\prod_{j_1, \dots, j_n}\check{C}_{\mathscr{V}_{j_1, \dots, j_n}})_n$).
\end{prop}

\begin{proof}
By \cite[079Z]{stacks}, the epimorphism of sheaves $\coprod_j h_{\check{C}_{\mathscr{V}_j}}\rightarrow *$ from Proposition~\ref{prop:Cover} implies that there is a spectral sequence with $E_1$-page
$$E_1^{p, q}=H^q\Big(\big(\coprod_j h_{\check{C}_{\mathscr{V}_j}}\big)^{\times p}, \mathcal{O}\Big)=H^q\Big(\coprod_{j_1, j_2, \dots, j_p} h_{\check{C}_{\mathscr{V}_{j_1,\dots, j_p}}}, \mathcal{O}\Big)=\prod_{j_i, \dots, j_p}H^q((\check{C}_{\mathscr{V}_{j_1, \dots, j_p}}, I\check{C}_{\mathscr{V}_{j_1, \dots, j_p}}), \mathcal{O})$$
converging to $H^{p+q}(*, \mathcal{O})=H^{p+q}((\mathscr{X}/A)_{\Prism}^\amalg, \mathcal{O})=H^{p+q}((\mathscr{X}/A)_{\Prism}, \mathcal{O}),$ where we implicitly used Corollary~\ref{cor:CohomologySame} and the fact that $h_{\check{C}_{\mathscr{V}_{j_1}}}\times h_{\check{C}_{\mathscr{V}_{j_2}}}=h_{\check{C}_{\mathscr{V}_{j_1}}\boxtimes\check{C}_{\mathscr{V}_{j_2}}}=h_{\check{C}_{\mathscr{V}_{j_1, j_2}}}$ as in Proposition~\ref{prop:ProductsInPrismaticSite}, and similarly for higher multi--indices.

By Lemma~\ref{VanishingObjects}, $H^q((\check{C}_{\mathscr{V}_{j_1, \dots, j_n}}, I\check{C}_{\mathscr{V}_{j_1, \dots, j_n}}), \mathcal{O})=0$ for every $q>0$ and every multi--index $j_1, \dots, j_n$. The first page is therefore concentrated in a single row of the form $\check{C}^\bullet_{\mathfrak{V}}$ and thus, the spectral sequence collapses on the second page. This proves that the cohomologies of $\R\Gamma((\mathscr{X}/A)_{\Prism}, \mathcal{O})$ are computed as cohomologies of $\check{C}^\bullet_{\mathfrak{V}}$, but in fact, this yields a quasi--isomorphism of the complexes themselves. (For example, analyzing the proof of \cite[079Z]{stacks} via \cite[03OW]{stacks}, the double complex $E_0^{\bullet \bullet}$ of the above spectral sequence comes with a natural map $\alpha:\check{C}^\bullet_{\mathfrak{V}}\rightarrow \mathrm{Tot}(E_0^{\bullet\bullet}),$ and a natural quasi--isomorphism $\beta: \R\Gamma((\mathscr{X}/A)_{\Prism}, \mathcal{O}) \rightarrow \mathrm{Tot}(E_0^{\bullet\bullet});$ when the spectral sequence collapses as above, $\alpha$ is also a quasi--isomorphism).
\end{proof}

\begin{rems}\label{rem:CechBaseChange}
\begin{enumerate}
\item{Just as in the affine case, the formation of \v{C}ech--Alexander complexes is compatible with ``termwise flat base--change'' on the base prism essentially by \cite[Proposition~3.13]{BhattScholze}. That is, if $(\check{C}^{m}, \partial)_m$ is a \v{C}ech--Alexander complex modelling $\R\Gamma_{\Prism}(\mathscr{X}/A)$ and $(A, I) \rightarrow (B, IB)$ is a flat map of prisms, then the complex $(\check{C}^{m}\widehat{\otimes}_A B, \partial\otimes 1)_m$ is a \v{C}ech--Alexander complex that computes $\R\Gamma_{\Prism}(\mathscr{X}_B/B)$.}
\item{Let now $(A, I)$ be the prism $(\ainf, \mathrm{Ker}\,\theta)$ and let $\mathscr{X}$ be of the form $\mathscr{X}=\mathscr{X}^0\times_{\oh_K}\oh_{\mathbb{C}_K}$ where $\mathscr{X}^0$ is a smooth separated formal $\oh_K$--scheme. A convenient way to describe the $G_K$--action on $\R\Gamma_{\Prism}(\mathscr{X}/\ainf)$ is via base--change: given $g \in G_K,$ the action of $g$ on $\ainf$ gives a map of prisms $g:(\ainf, (E(u)))\rightarrow (\ainf, (E(u)))$, and $g^*\mathscr{X}=\mathscr{X}$ since $\mathscr{X}$ comes from $\oh_K$. Base--change theorem for prismatic cohomology \cite[Theorem~1.8 (5)]{BhattScholze} then gives an $\ainf$--linear map $g^*\R\Gamma_{\Prism}(\mathscr{X}/\ainf)\rightarrow \R\Gamma_{\Prism}(\mathscr{X}/\ainf);$ untwisting by $g$ on the left, this gives an $\ainf$--$g$--semilinear action map $g: \R\Gamma_{\Prism}(\mathscr{X}/\ainf)\rightarrow \R\Gamma_{\Prism}(\mathscr{X}/\ainf)$. The exact same procedure defines the $G_K$--action on the \v{C}ech--Alexander complexes modelling the cohomology theories since they are base--change compatible in the sense above.}
\end{enumerate}
\end{rems}

\section{The conditions \Crs}\label{sec:crs}

\subsection{Definition and basic properties}

In order to describe the conditions \Crs, we need to fix more notation. For a natural number $s$, denote by $K_s$ the field $K(\pi_s)$ (where $(\pi_n)_n$ is the compatible chain of $p^n$--th roots of $\pi$ chosen before,  i.e. so that $u=[(\pi_n)_n]$ in $\ainf$), and set $K_\infty=\bigcup_s K_s$. Further set $K_{p^{\infty}}=\bigcup_m K(\zeta_{p^m})$ and for $s \in \mathbb{N}\cup \{\infty\}$, set $K_{p^{\infty},s}=K_{p^{\infty}}K_{s}$. Note that the field $K_{p^{\infty}, \infty}$ is the Galois closure of $K_{\infty}$. Denote by $\widehat{G}$ the Galois group $\mathrm{Gal}(K_{p^\infty, \infty}/K)$ and by $G_s$ the group $\mathrm{Gal}(\overline{K}/K_s),$ for $s \in \mathbb{N}\cup\{\infty\}$.

The group $\widehat{G}$ is generated by its two subgroups $\mathrm{Gal}(K_{p^\infty, \infty}/K_{p^{\infty}})$ and $\mathrm{Gal}(K_{p^\infty, \infty}/K_{\infty})$ (by \cite[Lemma~5.1.2]{LiuBreuilConjecture}). The subgroup $\mathrm{Gal}(K_{p^\infty, \infty}/K_{p^{\infty}})$ is normal, and its element $g$ is uniquely determined by its action on the elements $(\pi_s)_s$, which takes the form $g(\pi_s)=\zeta_{p^s}^{a_s} \pi_s$, with the integers $a_s$ unique modulo $p^s$ and compatible with each other as $s$ increases.  It follows that $\mathrm{Gal}(K_{p^{\infty},\infty}/K_{p^{\infty}})\simeq \mathbb{Z}_p$, with a topological generator $\tau$ given by $\tau (\pi_n)=\zeta_{p^n}\pi_n$ (where, again, $\zeta_{p^n}$'s are chosen as before, so that $v=[(\zeta_{p^n})_n]-1$). 

Similarly, the image of $G_s$ in $\widehat{G}$ is the subgroup $\widehat{G}_s=\mathrm{Gal}(K_{p^\infty, \infty}/K_s)$. Clearly $\widehat{G}_s$ contains $\mathrm{Gal}(K_{p^\infty, \infty}/K_{\infty})$ and the intersection of $\widehat{G}_s$ with $\mathrm{Gal}(K_{p^\infty, \infty}/K_{p^\infty})$ is $\mathrm{Gal}(K_{p^\infty, \infty}/K_{p^\infty, s}).$ Just as in the $s=0$ case, $\widehat{G}_s$ is generated by these two subgroups, with the subgroup $\mathrm{Gal}(K_{p^\infty, \infty}/K_{p^\infty, s})$ normal and topologically generated by the element $\tau^{p^s}$.

There is a natural $G_K$--action on $\ainf=W(\oh_{\mathbb{C}_K}^{\flat}),$ extended functorially from the natural action on $\oh_{\mathbb{C}_K}^{\flat}$. This action makes the map $\theta: \ainf \rightarrow \oh_{\mathbb{C}_K}$ $G_K$--equivariant, in particular, the kernel $E(u)\ainf$ is $G_K$--stable. The $G_K$--action on the $G_K$--closure of $\Es$ in $\ainf$ factors through $\widehat{G}$. Note that the subgroup $\mathrm{Gal}(K_{p^{\infty},\infty}/K_{{\infty}})$ of $\widehat{G}$ acts trivially on elements of $\Es$, and the action of the subgroup $\mathrm{Gal}(K_{p^{\infty},\infty}/K_{p^{\infty}})$ is determined by the equality $\tau (u)=(v+1)u$.

For an integer $s \geq 0$ and $i$ between $0$ and $s$,  denote by $\xi_{s, i}$ the element
$$\xi_{s, i}=\frac{\varphi^{s}(v)}{\omega \varphi(\omega) \dots \varphi^i(\omega)}=\varphi^{-1}(v)\varphi^{i+1}(\omega)\varphi^{i+2}(\omega)\dots \varphi^{s}(\omega)$$ (recall that $\omega=v/\varphi^{-1}(v)$), and set 
$$I_s=\left(\xi_{s, 0}u, \xi_{s, 1}u^p, \dots, \xi_{s, s}u^{p^s}\right).$$ 
For convenience of notation, we further set $I_{\infty}=0$ and $\varphi^{\infty}(v)u=0$.

We are concerned with the following conditions.

\begin{deff}\label{Def:Crys}
Let $M_{\inf}$ be an $\ainf$--module endowed with a $G_K$--$\ainf$--semilinear action, let $M_{\BK}$ be an $\Es$--module and let $M_{\BK}\rightarrow M_{\inf}$ be an $\Es$--linear map. Let $s \geq 0$ be an integer or $\infty$.
\begin{enumerate}[(1)]
\item{An element $x \in M_{\inf}$ is called a \emph{{\Crs}--element} if for every $g \in G_s$,  $$g(x)-x\in I_sM_{\inf}.$$}
\item{We say that the pair $M_{\BK} \rightarrow M_{\ainf}$ \emph{satisfies the condition \Crs} if for every element $x \in M_{\BK}$, the image of $x$ in $M_{\inf}$ is \Crs.} 
\item{An element $x \in M_{\inf}$ is called a \emph{{\Crrs}--element} if for every $g \in G_s$, there is an element $y \in M_{\inf}$ such that $$g(x)-x=\varphi^{s}(v)uy.$$}
\item{We say that the pair $M_{\BK} \rightarrow M_{\ainf}$ \emph{satisfies the condition \Crrs} if for every element $x \in M_{\BK}$, the image of $x$ in $M_{\inf}$ is \Crrs.} 
\item{Aditionally, we call \Cr{0}--elements \emph{crystalline elements} and we call the condition \Cr{0} \emph{the crystalline condition}.}
\end{enumerate}
\end{deff}

\begin{rems} 
\begin{enumerate}[(1)]
\item{Since $I_0=\varphi^{-1}(v)u\ainf,$ the crystalline condition equivalently states that for all $g \in G_K$ and all $x$ in the image of $M_{\BK},$ $$g(x)-x \in \varphi^{-1}(v)uM_{\inf}.$$ The reason for the extra terminology in the case $s=0$ is that the condition is connected with a criterion for certain representations to be crystalline, as discussed in \S\ref{subsec:BKF}. The higher conditions \Crs will on the other hand find application in computing bounds on ramification of $p^n$--torsion \'{e}tale cohomology. The conditions \Crrs serve an auxillary purpose. Clearly \Crrs implies \Crs. The conditions \Cr{\infty}, \Crr{\infty} are clearly both equivalent to the condition $f(M_{\BK}) \subseteq M_{\inf}^{G_\infty}$.
}
\item{Strictly speaking, one should talk about the crystalline condition (or \Crs) for the map $f$, but we choose to talk about the the crystalline condition (or \Crs) for the pair $(M_{\BK}, M_{\inf})$ instead, leaving the datum of the map $f$ implicit. This is because typically we consider the situation that $M_{\BK}$ is an $\Es$--submodule of $M_{\inf}^{G_{\infty}}$ and  $M_{\BK}\otimes_\Es \ainf \simeq M_{\ainf}$ via the natural map (or the derived $(p, E(u))$--completed variant, $M_{\BK}\widehat{\otimes}_\Es \ainf \simeq M_{\ainf}$). Also note that $f:M_{\BK}\rightarrow M_{\inf}$ satisfies the condition \Crs if and only if $f(M_{\BK})\subseteq M_{\inf}$ does.}
\end{enumerate}
\end{rems}

\begin{lem}\label{lem:GkStableIdeals}
For any integer $s$, the ideals $\varphi^s(v)u\ainf$ and $I_s$ are $G_K$--stable.
\end{lem}

\begin{proof}
It is enough to prove that the ideals $u\ainf$ and $v\ainf$ are $G_K$--stable. Note that the $G_K$--stability of $v\ainf$ implies $G_K$--stability of $\varphi^s(v)\ainf$ for any $s \in \mathbb{Z}$ since $\varphi$ is a $G_K$--equivariant automorphism of $\ainf$. Once we know this, we know that $g \varphi^s(v)$ equals to $\varphi^s(v)$ times a unit for every $g$ and $s$, the same is then true of $\varphi^{i}(\omega)=\varphi^i(v)/\varphi^{i-1}(v)$, hence also of all the elements $\xi_{i,s}$ and it follows that $I_s$  is $G_K$--stable.

Given $g \in G_K$, $g(\pi_n)=\zeta_{p^n}^{a_n}\pi_n$ for an ineger $a_n$ unique modulo $p^n$ and such that $a_{n+1}\equiv a_n \pmod{p^n}$. It follows that  $g(u)=[\underline{\varepsilon}]^a u$ for a $p$--adic integer $a$($=\lim_n a_n$). (The $\mathbb{Z}_p$--exponentiation used here is defined by $[\underline{\varepsilon}]^a =\lim_n [\underline{\varepsilon}]^{a_n} $ and the considered limit is with respect to the weak topology.) Thus, $u\ainf$ is $G_K$--stable.

Similarly, we have $g(\zeta_{p^n})=\zeta_{p^n}^{b_n},$ for integers $b_n$ coprime to $p$, unique modulo $p^n$ and compatible with each other as $n$ grows. It follows that $g([\underline{\varepsilon}])=[\underline{\varepsilon}]^b$ for $b=\lim_n b_n$, and so $g(v)=(v+1)^b-1=\lim_n((v+1)^{b_n}-1).$ The resulting expression is still divisible by $v$. To see that, fix the integers $b_n$ to have all positive representatives. Then the claim follows from the formula $$(v+1)^{b_n}-1=v((v+1)^{b_n-1}+(v+1)^{b_n-2}+\dots +1),$$ upon noting that the sequence of elements $((v+1)^{b_n-1}+(v+1)^{b_n-2}+\dots +1)=((v+1)^{b_n}-1)/v$ is still $(p, v)$--adically (i.e. weakly) convergent thanks to Lemma~\ref{disjointness}.
\end{proof}

Let $f: M_{\BK} \rightarrow M_{\inf}$ be as in Definition~\ref{Def:Crys}. Lemma~\ref{lem:GkStableIdeals} shows that the modules $M_{\inf}/{I_s} M_{\inf}$ and $M_{\inf}/\varphi^s(v)uM_{\inf}$ have a well--defined $G_K$--action. Consequently, we get the following restatement of the conditions \Crs, \Crrs.

\begin{lem}\label{lem:RestatementCrs}
Given $f: M_{\BK} \rightarrow M_{\inf}$ as in Definition~\ref{Def:Crys}, the pair $(M_{\BK}, M_{\inf})$ satisfies the condition \Crs (\Crrs, resp.) if and only if the image of  $M_{\BK}$ in $\overline{M_{\inf}}:=M_{\inf}/I_sM_{\inf}$ ($\overline{M_{\inf}}:=M_{\inf}/\varphi^s(v)uM_{\inf}$, resp.) lands in $\overline{M_{\inf}}^{G_s}$. 
\end{lem}

In the case of the above--mentioned condition $f(M_{\BK})\subseteq M_{\inf}^{G_{\infty}},$ the $G_K$--closure of $f(M_{\BK})$ in $M_{\inf}$ is contained in the $G_K$--submodule $M^{G_{K_{p^\infty,\infty}}}$, and thus, the $G_K$--action on it factors through $\widehat{G}$. Under mild assumtions on $M_{\inf}$, the $G_s$--action on the elements of $f(M_{\BK})$ is ultimately determined by $\tau^{p^s}$, the topological generator of $\mathrm{Gal}(K_{p^\infty, \infty}/ K_{p^\infty, s})$. Consequently, the conditions \Crrs are also determined by the action of this single element: 

\begin{lem}\label{lem:TauIsEnough}
Let $f: M_{\BK} \rightarrow M_{\inf}$ be as in Definition~\ref{Def:Crys}. Additionally, assume that $M_{\inf}$ is classically $(p, E(u))$--complete and $(p, E(u))$--completely flat, and that the pair $(M_{\BK}, M_{\inf})$ satisfies \Cr{\infty}. Then the action of $\widehat{G}$ on elements of $f(M_{\BK})$ makes sense, and $(M_{\BK}, M_{\inf})$ satisfies \Crrs if and only if
$$\forall x \in f(M_{\BK}): \tau^{p^s}(x)-x\in \varphi^s(v)uM_{\inf}.$$
\end{lem}

\begin{proof}
Clearly the stated condition is necessary. To prove sufficiency, assume the above condition for $\tau^{p^s}$. By the fixed--point interpretation of the condition \Crrs as in Lemma~\ref{lem:RestatementCrs}, it is clear that the analogous condition holds for every element $g \in \langle \tau^{p^s}\rangle $. 

Next, assume that $g$ is an element of $\mathrm{Gal}(K_{p^\infty, \infty}/K_{p^\infty, s})$, the $p$--adic closure of $\langle \tau^{p^s} \rangle$. Then $g$ can be written as $\lim_n \tau^{p^s a_n}$, with the sequence of integers $(a_n)$ $p$--adically convergent. For $x \in f(M_{\BK}),$ by continuity we have $g(x)-x=\lim_n (\tau^{p^s a_n}(x)-x)$, which is equal to $\lim_n \varphi^s(v)u y_n$ with $y_n \in M_{\inf}$. Since the sequence $(y_n)$ is still convergent (using the fact that the $(p, E(u))$--adic topology is the $(p, \varphi^s(v)u)$--adic topology, and that $p, \varphi^s(v)u$ is a regular sequence on $M_{\inf}$), we have that $g(x)-x=\varphi^s(v)u y$ where $y=\lim_n y_n$. 

To conclude, note that a general element of $\widehat{G}_s$ is of the form $g_1g_2$ where $g_1 \in \mathrm{Gal}(K_{p^\infty, \infty}/ K_{p^\infty, s})$ and $g_2 \in \mathrm{Gal}(K_{p^\infty, \infty}/ K_{\infty}).$ Then for $x \in f(M_{\BK})$, by the assumption $f(M_{\BK})\subseteq M_{\inf}^{G_{\infty}}$ we have $g_1g_2(x)-x=g_1(x)-x$, and so the condition \Crrs is proved by the previous part.
\end{proof}

Let us now discuss some basic algebraic properties of the conditions \Crs and \Crrs. The basic situation when they are satisfied is the inclusion $\Es \hookrightarrow \ainf$ itself.

\begin{lem}\label{coeff}
The pair $\Es\hookrightarrow \ainf$ satisfies the conditions \Crrs (hence also \Crs) for all $s \geq 0$.
\end{lem}

\begin{proof}
Note that $\Es\hookrightarrow \ainf$ satisfies the assumptions of Lemma~\ref{lem:TauIsEnough}, so it is enough to consider the action of the element $\tau^{p^s} \in \widehat{G}_s$.  For an element $f = \sum_i a_i u^i \in \Es$ we have
$$\tau^{p^s}(f)-f=\sum_{i\geq 0} a_i ((v+1)^{p^s}u)^{i}-\sum_{i\geq 0} a_iu^i=\sum_{i\geq 1}a_i((v+1)^{p^si}-1)u^i,$$
and thus,
$$\frac{\tau^{p^s}(f)-f}{\varphi^s(v)u}=\sum_{i \geq 1}a_i\frac{(v+1)^{p^si}-1}{\varphi^s(v)}u^{i-1}=\sum_{i \geq 1}a_i\frac{(v+1)^{p^si}-1}{(v+1)^{p^s}-1}u^{i-1}$$
Since $\varphi^s(v)=(v+1)^{p^s}-1$ divides $(v+1)^{p^si}-1$ for each $i$, the obtained series has coefficients in $\ainf$, showing that $\tau^{p^s}(f)-f \in \varphi^s(v)u\ainf$ as desired.
\end{proof}

The following lemma shows that in various contexts, it is often sufficient to verify the conditions \Crs, \Crrs on generators.

\begin{lem}\label{generators}
Fix an integer $s \geq 0$. Let \textnormal{(C)} be either the condition \Crs or \Crrs.
\begin{enumerate}
\item{Let $M_{\inf}$ be an $\ainf$--module with a $G_K$--$\ainf$--semilinear action. The set of all \textnormal{(C)}--elements forms an $\Es$--submodule of $M_{\inf}$.}
\item{Let $C_{\inf}$ be an $\ainf$--algebra endowed with a $G_K$--semilinear action. The set of \textnormal{(C)}--elements of $C_{\inf}$ forms an $\Es$--subalgebra of $C_{\inf}$.}
\item{If the algebra $C_{\inf}$ from (2) is additionally $\ainf$--$\delta$--algebra such that $G_K$ acts by $\delta$--maps (i.e. $\delta g=g \delta$ for all $g \in G_K$) then the set of all \textnormal{(C)}--elements forms a $\Es$--$\delta$--subalgebra of $C_{\inf}$.} 
\item{If the algebra $C_{\inf}$ as in (2) is additionally derived $(p, E(u))$--complete 
and $C_{\BK}\rightarrow C_{\inf}$ is a map of $\Es$--algebras  that satisfies the condition \textnormal{(C)}, then so does $\widehat{C_{\BK}} \rightarrow C_{\inf},$ where $\widehat{C_{\BK}}$ is the derived $(p, E(u))$--completion of $C_{\BK}$.  In particular, the set of all \textnormal{(C)}--elements in $C_{\inf}$ forms a derived $(p, E(u))$--complete $\Es$--subalgebra of $C_{\inf}$.} 
\end{enumerate}
\end{lem}

\begin{proof}
Let $J$ be the ideal $I_s$ if (C)=\Crs and the ideal $\varphi^s(v)u\ainf$ if (C)=\Crrs.
In view of Lemma~\ref{lem:RestatementCrs}, the sets described in (1),(2) are obtained as the preimages of $\left(M_{\inf}/J M_{\inf}\right)^{G_s}$ (ring $\left(C_{\inf}/J C_{\inf}\right)^{G_s}$, resp.) under the canonical projection $M_{\inf} \rightarrow M_{\inf}/J M_{\inf}$ ($C_{\inf} \rightarrow C_{\inf}/J C_{\inf},$ resp.). As these $G_s$--fixed points form an $\Es$--module ($\Es$--algebra, resp.) by Lemma~\ref{coeff}, this proves (1) and (2).

Similarly, to prove (3) we need to prove only that the ideal $J C_{\inf}$ is a $\delta$--ideal and therefore the canonical projection $C_{\inf} \rightarrow C_{\inf}/J C_{\inf}$ is a map of $\delta$--rings. 

Let us argue first in the case \Crrs. As $\delta(u)=0$, we have

$$\delta(\varphi^{s}(v)u)=\delta(\varphi^{s}(v))u^{p}=\frac{\varphi(\varphi^{s}(v))-(\varphi^{s}(v))^p}{p}u^{p}=\frac{\varphi^{s+1}(v)-(\varphi^{s}(v))^p}{p}u^{p}.$$
Recall that $\varphi^s(v)=[\underline{\varepsilon}]^{p^s}-1$ divides $\varphi^{s+1}(v)=([\underline{\varepsilon}]^{p^s})^p-1$.
The numerator of the last fraction is thus divisible by $\varphi^{s}(v)$ and since $\varphi^{s}(v)\ainf \cap p\ainf=\varphi^{s}(v)p\ainf$ by Lemma~\ref{disjointness}, $\varphi^{s}(v)$ divides the whole fraction ${(\varphi^{s+1}(v)-(\varphi^{s}(v))^p)/p}$ in $\ainf$. (We note that this is true for \textit{every} integer $s$, in particular $s=-1$, as well.)

Let us now prove that the ideal $J=I_s$ (hence also $I_sC_{\inf}$) is a $\delta$--ideal. For any $i$ between $0$ and $s-1$, we have
$$\delta\left(\xi_{s, i}\right)=\delta(\varphi^{-1}(v)\varphi^{i+1}(\omega)\dots \varphi^s(\omega))=\frac{\varphi^{-1}(v)\omega \varphi^{i+2}(\omega)\dots \varphi^{s+1}(\omega)-\varphi^{-1}(v)^p\varphi^{i+1}(\omega)^p\dots \varphi^s(\omega)^p}{p}.$$
The numerator is divisible by $\xi_{s, i+1}$, and so is the whole fraction thanks to Lemma~\ref{disjointness}. Thus, we have that $\delta(\xi_{s, i}u^{p^i})=\delta(\xi_{s, i})u^{p^{i+1}}$ is a multiple of $\xi_{s, i+1}u^{p^{i+1}}$. Finally, when $i=s$, we have $\xi_{s, s}=\varphi^{-1}(v)$, and $\delta(\xi_{s, s})$ is thus a multiple of $\xi_{s, s}$ by the previous. Consequently, $\delta(\xi_{s,s}u^{p^s})=\delta(\xi_{s,s})u^{p^{s+1}}$ is divisible by $\xi_{s,s}u^{p^s}$. This shows that $I_s$ is a $\delta$--ideal.

Finally, let us prove (4). Note that $E(u)\equiv u^e \pmod{p\Es}$, hence $\sqrt{(p, E(u))}=\sqrt{(p, u^e)}=\sqrt{(p, u)}$ even as ideals of $\Es$; consequently, the derived $(p, E(u))$--completion agrees with the derived $(p, u)$--completion both for $\Es$-- and $\ainf$--modules. We may therefore replace $(p, E(u))$--completions with $(p, u)$--completions throughout.

Since $C_{\inf}$ is derived $(p, u)$--complete, any power series of the form 
$$f=\sum_{i,j}c_{i, j}p^iu^j$$
with $c_{i, j}\in C_{\inf}$ defines a unique\footnote{Here we are using the preferred representatives of powers series as mentioned at the beginning of \S\ref{subsec:regularity}.} element $f \in C_{\inf}$, and $f$ comes from $\widehat{C_{\BK}}$ if and only if the coefficients $c_{i, j}$ may be chosen in the image of the map $C_{\BK}\rightarrow C_{\inf}$. Assuming this, for $g \in G_s$ we have
$$g(f)-f=\sum_{i, j}g(c_{i, j})p^i(\gamma u)^j-\sum_{i, j}c_{i, j}p^iu^j=$$
$$=\sum_{i, j}\left(g(c_{i, j})\gamma^j-g(c_{i, j})+g(c_{i, j})-c_{i, j}\right)p^iu^j,$$
where $\gamma$ is the $\ainf$--unit such that $g(u)=\gamma u$. Thus, it is clearly enough to show, upon assuming the condition \textnormal{(C)} for $(C_{\BK}, C_{\inf})$, that the terms $\left(g(c_{i, j})\gamma^{j}-g(c_{i, j})\right)p^iu^j$ and $\left(g(c_{i, j})-c_{i, j}\right)p^iu^j$ are in $J C_{\inf}$ when $g \in G_{{s}}$. (Note that an element $d=\sum_{i, j} d_{i,j}p^i u^j$ with $d_{i, j}\in JC_{\inf}$ is itself in $JC_{\inf}$, since $J$ is finitely generated.)

We have $g(c_{i, j})-c_{i, j} \in J C_{\inf}$ by assumption, so it remains to treat the term $g(c_{i, j})(\gamma^j-1)$. Since $(\gamma^{j}-1)$ is divisible by $\gamma-1$, it is also divisible by $\varphi^s(v)$ by Lemma~\ref{coeff}. Thus, the terms $g(c_{i, j})(\gamma^j-1)p^i u^j$ are divisible by $\varphi^s(v)u$ when $j\geq 1$, and are $0$ when $j=0$; in either case, they are members of $JC_{\inf}$.

To prove the second assertion of (4), let now $C_{\BK} \subseteq C_{\inf}$ be the $\Es$--subalgebra of all crystalline elements. By the previous, the map $\widehat{C_{\BK}}\rightarrow C_{\inf}$  satisfies \textnormal{(C)}, and hence the image $C_{\BK}^+$ of this map consists of \textnormal{(C)}--elements. Thus, we have $C_{\BK} \subseteq C_{\BK}^+\subseteq C_{\BK},$ and hence, $C_{\BK}$ is derived $(p, E(u))$--complete since so is $C_{\BK}^+$.
\end{proof}

\begin{rem}
One consequence of Lemma~\ref{generators} is that the $\Es$--subalgebra $\mathfrak{C}$ of $\ainf$ formed by all crystalline elements (or even \Crr{0}--elements) forms a prism, with the distinguished invertible ideal $I=E(u)\mathfrak{C}$. As Lemma~\ref{coeff} works for any choice of Breuil--Kisin prism associated to $K/K'$ in $\ainf$, $\mathfrak{C}$ contains all of these (in particular, it contains all $G_K$--translates of $\Es$). 
\end{rem}

For future use in applications to $p^n$--torsion modules, we consider the following approximation of the ideals $I_s$ appearing in the conditions \Crs.

\begin{lem}\label{lem:IsModPn}
Consider a pair of integers $n, s$ with $s\geq 0, n \geq 1$. Set $t=\mathrm{max}\left\{0, s+1-n\right\}$. Then the image of the ideal $I_s$ in the ring $W_n(\oh_{\mathbb{C}_K^\flat})=\ainf/p^n$ is contained in the ideal $\varphi^{-1}(v)u^{p^{t}}W_n(\oh_{\mathbb{C}_K^\flat})$. That is, we have $I_s+p^n\ainf \subseteq \varphi^{-1}(v)u^{p^t}\ainf+p^n\ainf.$
\end{lem}

\begin{proof}
When $t=0$ there is nothing to prove, therefore we may assume that $t=s+1-n>0$. In the definition of $I_s$, we may replace the elements $$\xi_{s, i}=\varphi^{-1}(v)\varphi^{i+1}(\omega)\varphi^{i+2}(\omega)\dots \varphi^{s}(\omega)$$
by the elements
$$\xi'_{s, i}=\varphi^{-1}(v)\varphi^{i+1}(E(u))\varphi^{i+2}(E(u))\dots \varphi^{s}(E(u)),$$
since the quotients $\xi_{s, i}/\xi'_{s, i}$ are $\ainf$--units.

It is thus enough to show that for every $i$ with $0 \leq i \leq s,$ the element 
$$\vartheta_{s, i}=\frac{\xi'_{s, i}u^{p^i}}{\varphi^{-1}(v)}=\varphi^{i+1}(E(u))\varphi^{i+2}(E(u))\dots \varphi^{s}(E(u))u^{p^i}$$ taken modulo $p^n$ is divisible by $u^{p^{s+1-n}}$.  

This is clear when $i\geq s+1-n$, and so it remains to discuss the cases when $i \leq s-n.$ Write $\varphi^{j}(E(u))=(u^e)^{p^j}+px_j$ (with $x_j \in \Es$). Then it is enough to show that
\begin{equation}\tag{$*$}\label{eqn:ProductPhiEu}\frac{\vartheta_{s, i}}{u^{p^i}}=((u^e)^{p^{i+1}}+px_{i+1})((u^e)^{p^{i+2}}+px_{i+2})\dots((u^e)^{p^s}+px_s)
\end{equation}
taken modulo $p^n$ is divisible by $$u^{p^{s+1-n}-p^{i}}=u^{p^i(p-1)(1+p+\dots +p^{s-n-i})}.$$ 
Since we are interested in the product (\ref{eqn:ProductPhiEu}) only modulo $p^n$, in expanding the brackets we may ignore the terms that use the expressions of the form $px_j$ at least $n$ times. Each of the remaining terms contains the product of at least $s-i-n+1$ distinct terms from the following list:
$$(u^e)^{p^{i+1}}, (u^e)^{p^{i+2}}, \dots, (u^e)^{p^{s}}.$$
Thus, each of the remaining terms is divisible by (at least)
$$(u^e)^{p^{i+1}+p^{i+2}+\dots+p^{s-n+1}}=(u^e)^{p^i\cdot(p)\cdot(1+p+\dots +p^{s-n-i})},$$
which is more than needed. This finishes the proof. 
\end{proof}

\subsection{Crystalline condition for Breuil--Kisin--Fargues $G_K$--modules}\label{subsec:BKF}

The situation of central interest regarding the crystalline condition is the inclusion $M_{\BK} \rightarrow M_{\inf}^{G_{\infty}}$ such that $\ainf\otimes_{\Es}M_{\BK} \rightarrow M_{\inf}$ is an isomorphism, where $M_{\BK}$ is a Breuil--Kisin module and $M_{\inf}$ is a Breuil--Kisin--Fargues $G_K$--module. The version of these notions used in this paper is tailored to the context of prismatic cohomology. Namely, we have:

\begin{deff}
\begin{enumerate}[(1)]
\item{A \emph{Breuil--Kisin module} is a finitely generated $\Es$--module $M$ together with a $\Es[1/E(u)]$--linear isomorphism $$\varphi=\varphi_{M[1/E]}:(\varphi_{\Es}^*M)[1/E(u)]\stackrel{\sim}{\rightarrow} M[1/E(u)].$$ For a positive integer $i$, the Breuil--Kisin module $M$ is said to be \textit{of height $\leq i$} if $\varphi_{M[1/E]}$  is induced (by linearization and localization) by a $\varphi_{\Es}$--semilinear map $\varphi_M: M \rightarrow M$ such that, denoting $\varphi_{\lin}: \varphi^*M \rightarrow M$ its linearization, there exists an $\Es$--linear map $\psi: M\rightarrow \varphi^*M$ such that both the compositions $\psi \circ \varphi_{\lin}$ and $ \varphi_{\lin}\circ \psi $ are multiplication by $E(u)^i$. A Breuil--Kisin module is \emph{of finite height} if it is of height $\leq i$ for some $i$.} 
\item{A \emph{Breuil--Kisin--Fargues module} is a finitely presented $\ainf$--module $M$ such that $M[1/p]$ is a free $\ainf[1/p]$--module, together with an $\ainf[1/E(u)]$--linear isomorphism $$\varphi=\varphi_{M[1/E]}:(\varphi_{\ainf}^*M)[1/E(u)]\stackrel{\sim}{\rightarrow} M[1/E(u)].$$ Similarly, the Breuil--Kisin--Fargues module is called \textit{of height $\leq i$} if $\varphi_{M[1/E]}$ comes from a $\varphi_{\ainf}$--semilinear map $\varphi_M:M \rightarrow M$ admitting an $\ainf$--linear map $\psi: M \rightarrow \varphi^*M$ such that $\psi \circ \varphi_{\lin}$ and $ \varphi_{\lin}\circ \psi $ are multiplication maps by $E(u)^i$, where $\varphi_{\lin}$ is the inearization of $\varphi_M$. A Breuil--Kisin--Fargues module is \emph{of finite height} if it is of height $\leq i$ for some $i$.}
\item{A \emph{Breuil--Kisin--Fargues $G_K$--module} (of height $\leq i$, of finite height, resp.) is a Breuil--Kisin--Fargues module (of height $\leq i$, of finite height, resp.) that is additionally endowed with an $\ainf$--semilinear $G_K$--action that makes $\varphi_{M[1/E]}$ $G_K$--equivariant (that makes also $\varphi_M$ $G_K$--equivariant in the finite height cases).}

\end{enumerate}
\end{deff}

That is, the definition of a Breuil--Kisin module agrees with the one in \cite{BMS1}, and $M_{\inf}$ is a Breuil--Kisin--Fargues module in the sense of the above definition if and only if $\varphi_{\ainf}^*M_{\inf}$ is a Breuil--Kisin--Fargues module in the sense of \cite{BMS1}\footnote{This is to account for the fact that while Breuil--Kisin--Fargues modules in the sense of \cite{BMS1} appear as $\ainf$--cohomology groups of smooth proper formal schemes, Breuil--Kisin--Fargues modules in the above sense appear as \textit{prismatic} $\ainf$--cohomology groups of smooth proper formal schemes.}. The notion of Breuil--Kisin module of height $\leq i$ agrees with what is called ``(generalized) Kisin modules of height $i$'' in \cite{LiLiu}. The above notion of finite height Breuil--Kisin--Fargues modules agrees with the one from \cite[Appendix~F]{EmertonGee2} except that the modules are not assumed to be free. Also note that under these definitions, for a Breuil--Kisin module $M_{\BK}$ (of height $\leq i,$ resp.), the $\ainf$--module $M_{\inf}=\ainf\otimes_{\Es}M_{\BK}$ is a Breuil--Kisin--Fargues module (of height $\leq i,$ resp.), without the need to twist the embedding $\Es \rightarrow \ainf$ by $\varphi$.

The connection between Breuil--Kisin--, Breuil--Kisin--Fargues $G_K$--modules and the crystalline condition (justifying its name) is the following theorem.

\begin{thm}[{\cite[Appendix F]{EmertonGee2}}, \cite{GaoBKGK}]\label{BKBKFCrystallineThm}
Let  $M_{\inf}$ be a free Breuil--Kisin--Fargues $G_K$--module which admits as an $\Es$--submodule a free Breuil--Kisin module $M_{\BK}\subseteq M_{\inf}^{G_{\infty}}$ of finite height, such that $\ainf\otimes_{\Es}M_{\BK} \stackrel{\sim}\rightarrow M_{\inf}$ (as Breuil--Kisin--Fargues modules) via the natural map, and such that the pair $(M_{\BK}, M_{\inf})$ satisfies the crystalline condition. Then the \'{e}tale realization of $M_{\inf}$,
$$V(M_{\inf})=\left(W(\mathbb{C}_K^\flat)\otimes_{\ainf} M_{\inf}\right)^{\varphi=1}\left[\frac{1}{p}\right],$$
is a crystalline representation.
\end{thm}

\begin{rems}\label{rem:CrystConditionProof}\begin{enumerate}[(1)]
\item{Theorem~\ref{BKBKFCrystallineThm} is actually an equivalence: If $V(M_{\inf})$ is crystalline, it can be shown that the pair $(M_{\BK}, M_{\inf})$ satisfies the crystalline condition. We state the theorem in the one direction since this is the one that we use. However, the converse direction motivates why it is resonable to expect the crystalline condition for prismatic cohomology groups that is discussed in Section~\ref{sec:CrsCohomology}.}
\item{Strictly speaking, in \cite[Appendix~F]{EmertonGee2} one assumes extra conditions on the pair $M_{\inf}$ (``satisfying all descents''); however, these extra assumptions are used only for a semistable version of the statement. Theorem~\ref{BKBKFCrystallineThm} in its equivalence form is therefore only implicit in the proof of \cite[Theorem~F.11]{EmertonGee2}. (See also \cite[Theorem~3.8]{Ozeki} for a closely related result.)}
\item{On the other hand, Theorem~\ref{BKBKFCrystallineThm} in the one--sided form as above is a consequence of \cite[Proposition~7.11]{GaoBKGK} that essentially states that $V(M)$ is crystalline if and only if the much weaker condition 
$$\forall g \in G_K:\;\; (g-1)M_{\BK} \subseteq \varphi^{-1}(v)W(\mathfrak{m}_{\oh_{\mathbb{C}_K^\flat}})M_{\inf}$$
is satisfied. We note a related result of \textit{loc. cit.}: $V(M)$ is semistable if and only if 
$$\forall g \in G_K:\;\; (g-1)M_{\BK} \subseteq W(\mathfrak{m}_{\oh_{\mathbb{C}_K^\flat}})M_{\inf}.$$
This semistable criterion above might be a good starting point in generalizing the results of Sections~\ref{sec:CrsCohomology} and \ref{sec:bounds} of the present paper to the case of semistable reduction, using the log--prismatic cohomology developed in \cite{Koshikawa}. Thus, a natural question to ask is: Similarly to how the crystalline condition is a stronger version of the crystallinity criterion from \cite{GaoBKGK}, what is an analogous stronger (while still generally valid) version of the semistability criterion from \cite{GaoBKGK}?}
\end{enumerate}
\end{rems}

It will be convenient later to have version of Theorem~\ref{BKBKFCrystallineThm} that applies to not necessarily free Breuil--Kisin and Breuil--Kisin--Fargues modules. Recall that, by \cite[Propostition 4.3]{BMS1}, any Breuil--Kisin module $M_{\BK}$ is related to a free Breuil--Kisin module $M_{\BK,\mathrm{free}}$ by a functorial exact sequence 
\begin{center}
\begin{tikzcd}
0 \ar[r] & M_{\BK,\mathrm{tor}} \ar[r] & M_{\BK} \ar[r] & M_{\BK, \mathrm{free}} \ar[r] & \overline{M_{\BK}}  \ar[r] & 0
\end{tikzcd}
\end{center}
where $M_{\BK,\mathrm{tor}}$ is a $p^n$-torsion module for some $n$ and $\overline{M_{\BK}}$ is supported at the maximal ideal $(p, u)$. Taking the base--change to $\ainf,$ one obtains an analogous exact sequence 
\begin{center}
\begin{tikzcd}
0 \ar[r] & M_{\inf,\mathrm{tor}} \ar[r] & M_{\inf} \ar[r] & M_{\inf, \mathrm{free}} \ar[r] &\overline{M_{\inf}}\ar[r] & 0
\end{tikzcd}
\end{center}
(also described by \cite[Proposition 4.13]{BMS1}) where $M_{\inf, \mathrm{free}}$ is a free Breuil--Kisin--Fargues module. Clearly the maps $M_{\BK}\rightarrow M_{\mathrm{free}}$ and $ M_{\inf}\rightarrow M_{\inf, \mathrm{free}}$ become isomorphisms after inverting $p$.

Assume that $M_{\inf}$ is endowed with a $G_K$--action that makes it a Breuil--Kisin--Fargues $G_K$--module. The functoriality of the latter exact sequence implies that the $G_K$--action on $M_{\inf}$ induces a $G_K$--action on $M_{\inf, \mathrm{free}}$, endowing it with the structure of a free Breuil--Kisin--Faruges $G_K$--module. In more detail, given $\sigma \in G_K$, the semilinear action map $\sigma: M_{\inf}\rightarrow M_{\inf}$ induces an $\ainf$--linear map $\sigma_{\mathrm{lin}}:\sigma^*M_{\inf} \rightarrow M_{\inf}$ where $\sigma^*M=\ainf \otimes_{\sigma, \ainf} M$. As $\sigma$ is an isomorphism fixing $p$, $E(u)$ up to unit and the ideal $(p, u)\ainf,$  it is easy to see that $\sigma^*M_{\inf}$ is itself a Breuil--Kisin--Fargues module, and the exact sequence from \cite[Proposition 4.13]{BMS1} for $\sigma^*M_{\inf}$ can be identified with the upper row of the diagram
\begin{center}
\begin{tikzcd}
0 \ar[r] & \sigma^*M_{\inf,\mathrm{tor}} \ar[r] \ar[d, "\sigma_{\mathrm{lin}}"] & \sigma^*M_{\inf} \ar[d, "\sigma_{\mathrm{lin}}"] \ar[r] & \sigma^*M_{\inf, \mathrm{free}} \ar[d, "\sigma_{\mathrm{lin}}"] \ar[r] & \sigma^*\overline{M_{\inf}} \ar[d, "\sigma_{\mathrm{lin}}"] \ar[r] & 0 \\
0 \ar[r] & M_{\inf,\mathrm{tor}} \ar[r] & M_{\inf} \ar[r] & M_{\inf, \mathrm{free}} \ar[r] &\overline{M_{\inf}}\ar[r] & 0,
\end{tikzcd}
\end{center}

where the second vertical map is the linearization of $\sigma$ and the rest is induced by functoriality of the sequence. Finally, untwisting $\sigma^*M_{\inf, \mathrm{free}},$ the third vertical map $\sigma_{\mathrm{lin}}$ induces a semilinear map $\sigma: M_{\inf, \mathrm{free}} \rightarrow M_{\inf, \mathrm{free}}$. Note that the module $M_{\inf}[1/p]\simeq M_{\inf, \mathrm{free}}[1/p]$ inherits the $G_K$--action from $M_{\inf}$; it is easy to see that the $G_K$-action on $M_{\inf, \mathrm{free}}$ agrees with the one on $M_{\inf}[1/p]$ when viewing $M_{\inf, \mathrm{free}}$ as its submodule. 

\begin{prop}\label{CrystallineFree}
Assume that the pair $M_{\BK} \hookrightarrow M_{\inf}$ satisfies the crystalline condition. Then so does the pair $M_{\BK,\mathrm{free}} \hookrightarrow M_{\inf, \mathrm{free}}.$
\end{prop}

\begin{proof}
Notice that the crystalline condition is satisfied for $M_{BK}[1/p]\rightarrow M_{\inf}[1/p]$ and by \cite[Propositions 4.3, 4.13]{BMS1}, this map can be identified with $M_{\BK,\mathrm{free}}[1/p] \hookrightarrow M_{\inf, \mathrm{free}}[1/p]$. Thus, the following lemma finishes the proof.
\end{proof}

\begin{lem}
Let $F_{\inf}$ be a free $\ainf$--module endowed with $\ainf$--semilinear $G_K$--action and let $F_{BK} \subseteq F_{\inf}$ be a free $\Es$--submodule such that $F_{BK}[1/p]\hookrightarrow F_{\inf}[1/p]$ satisfies the crystalline condition. Then the pair $F_{BK}\hookrightarrow F_{\inf}$ satisfies the crystalline condition.
\end{lem}

\begin{proof}
Fix an element $a\in F_{BK}$ and $g \in G_K$. The crystalline condition holds after inverting $p$, and so we have the equality 
$$b:=(g-1)a=\varphi^{-1}(v)u\frac{c}{p^k}$$
with $c \in F_{\inf}$. In other words (using that $p^k$ is a non-zero divisor on $F_{\inf}$), we have 
$$p^kb=\varphi^{-1}(v)uc\in p^k F_{\inf}\cap \varphi^{-1}(v)uF_{\inf}=p^k\varphi^{-1}(v)uF_{\inf},$$
where the last equality follows by Lemma~\ref{disjointness} since $F_{\inf}$ is a free module. In particular, we have $$p^kb=p^k\varphi^{-1}(v)ud$$ for yet another element $d \in F_{\inf}$. As $p^k$ is a non--zero divisor on $\ainf$, hence on $F_{\inf},$ we may cancel out to conclude 
$$(g-1)a=b=\varphi^{-1}(v)ud\in \varphi^{-1}(v)u F_{\inf},$$
as desired.
\end{proof}

Proposition~\ref{CrystallineFree} leads to the following strenghtening of Theorem~\ref{BKBKFCrystallineThm}.

\begin{thm}\label{BKBKFCrystallineGeneralThm}
The ``free'' assumption in Theorem~\ref{BKBKFCrystallineThm} is superfluous. That is, given a Breuil--Kisin--Fargues $G_K$--module $M_{\inf}$ together with its Breuil--Kisin--$\Es$--submodule $M_{\BK} \subseteq M_{\inf}^{G_{\infty}}$ of finite height such that $\ainf\otimes_{\Es}M_{\BK} \stackrel{\sim}\rightarrow M_{\inf}$ and such that the pair $(M_{\BK}, M_{\inf})$ satisfies the crystalline condition, the representation
$$V(M_{\inf})=\left(W(\mathbb{C}_K^\flat)\otimes_{\ainf} M_{\inf}\right)^{\varphi=1}\left[\frac{1}{p}\right]$$
is crystalline.
\end{thm}

\begin{proof}
With the notation as above, upon realizing that $V(M_{\inf})$ and $V(M_{\inf, \mathrm{free}})$ agree, the result is a direct consequence of Proposition~\ref{CrystallineFree}.
\end{proof}

\section{Conditions \Crs for cohomology} \label{sec:CrsCohomology}

\subsection{\Crs for \v{C}ech--Alexander complexes}

Let $\mathscr{X}$ be a smooth separated $p$--adic formal scheme over $\oh_K$. Denote by $\v{C}_{\BK}^{\bullet}$ a \v{C}ech--Alexander complex that models $\R\Gamma_{\Prism}(\mathscr{X}/\Es)$ and set $\check{C}_{\inf}^{\bullet}=\check{C}_{\BK}^{\bullet}\widehat{\otimes}_{\Es}\ainf$, computed termwise -- by Remark~\ref{rem:CechBaseChange}, this is a \v{C}ech--Alexander complex modelling $\R\Gamma_{\Prism}(\mathscr{X}_{\ainf}/\ainf)$. We aim to prove the following.

\begin{thm}\label{thm:CrsForCechComplex}
For every $m\geq 0$ and $s \in \mathbb{N}\cup \{\infty\}$, the pair $\v{C}_{\BK}^m\rightarrow \v{C}_{\inf}^m$ satisfies the condition \Crs.
\end{thm}

Let $\spf(R)=\mathscr{V} \subseteq \mathscr{X}$ be an affine open formal subscheme. Then it is enough to prove the content of Theorem~\ref{thm:CrsForCechComplex} for $\check{C}_{\BK}\rightarrow \check{C}_{\inf}$ where $\check{C}_{\BK}$  and $\check{C}_{\inf}=\check{C}_{\BK}\widehat{\otimes}_{\Es} \ainf$ are the \v{C}ech--Alexander covers of $\mathscr{V}$  and $\mathscr{V}'=\mathscr{V}\times_{\Es}\ainf$ with respect to the base prism $\Es$ and $\ainf$, respectively, since the \v{C}ech--Alexander complexes termwise consist of products of such covers. Let $R'=R \widehat{\otimes}_{\oh_K}\oh_{\mathbb{C}_K}(=R \widehat{\otimes}_{\Es}\ainf)$. 

Let us fix a choice of the free $\Es$--algebra $P_0=\Es[\{X_i\}_{i \in I}]$ whose $(p, E(u))$--completion is the algebra $P=P_{\mathscr{V}}$ as in Construction~\ref{constACcover}, with $J$ being the kernel of the surjection $P \rightarrow R$. Then the corresponding choices at the $\ainf$--level are $P_0'=P_0\otimes_{\Es}\ainf$ and $P'=P\widehat{\otimes}_{\Es}\ainf,$ and the associated $(p, E(u))$--completed ``$\delta$--envelopes'' are also related by the completed base change; that is, we have a diagram with exact rows
\begin{equation}\label{CechDiagramGalois}
\begin{tikzcd}
\widehat{JP^{\delta}} \ar[r] \ar[d] & \widehat{P^{\delta}} \ar[r] \ar[d] & \widehat{R\otimes_{P}P^{\delta}} \ar[r] \ar[d] & 0 \ar[d, phantom, "-\widehat{\otimes}_{\Es}\ainf"]\\
\widehat{J(P')^{\delta}} \ar[r] & \widehat{(P')^{\delta}} \ar[r, "\alpha"] & \widehat{R'\otimes_{P'}(P')^{\delta}} \ar[r] & 0.
\end{tikzcd}
\end{equation}  

By Remark~\ref{rem:ComplFinPresented}, we may and do assume that the set of variables $\{X_i\}_i=\{X_1, \dots, X_m\}$ is finite, and that the ideal $J$ is finitely generated. Consequently, after replacing the maps on the left by their respective images (and invoking Remark~\ref{rem:NonCompletedEnvelopes} (1)), diagram~(\ref{CechDiagramGalois}) becomes

\begin{equation}\label{DiagramNonCompleted}
\begin{tikzcd}
0 \ar[r] & J\widehat{P^{\delta}} \ar[r] \ar[d] & \widehat{P^{\delta}} \ar[r] \ar[d, "-\widehat{\otimes}_{\Es}\ainf"] & \widehat{R\otimes_{P}P^{\delta}} \ar[r] \ar[d, "-\widehat{\otimes}_{\Es}\ainf"] & 0  \\
0 \ar[r] & J\widehat{(P')^{\delta}} \ar[r] & \widehat{(P')^{\delta}} \ar[r, "\alpha"] & \widehat{R'\otimes_{P'}(P')^{\delta}} \ar[r] & 0
\end{tikzcd}
\end{equation}  

\noindent where the rows are exact. The prescription $g(X_i)=X_i$ determines uniquely a continuous, $\ainf$--semilinear Galois action by $\delta$--maps on $\widehat{(P')^{\delta}} \simeq \widehat{P_0^{\delta}\otimes_{\Es}\ainf}$ (and, in particular, this action satisfies $g(\delta^j(X_i))=\delta^j(X_i)$ for all $g \in G_K$ and all $i, j$).
Similarly, the term $\widehat{R'\otimes_{P'}(P')^{\delta}}\simeq \widehat{R\otimes_{P}P^{\delta}\otimes_{\oh_K}\oh_{\mathbb{C}_K}}$ is given the (linear) $G_K$--action prescribed by $g(x \otimes a)=x \otimes g(a)$ for every $g \in G_K$, $a \in \oh_{\mathbb{C}_K^{\flat}}$ and $x$ coming from the first row. This makes the map $\alpha$ $G_K$--equivariant, and therefore the kernel $\widehat{J(P')^{\delta}}$ $G_K$--stable. As a consequence, the action extends to the prismatic envelope $(\check{C}_{\inf}, I\check{C}_{\inf})$ where the action obtained this way agrees with the one indicated in Remark~\ref{rem:CechBaseChange}. Upon taking the prismatic envelope $(\check{C}_{\BK}, I\check{C}_{\BK})$ of the $\delta$--pair $(\widehat{P_0^{\delta}}, J_0\widehat{P_0^{\delta}})$, we arrive at the situation $\check{C}_{\BK}\hookrightarrow \check{C}_{\inf}=\check{C}_{\BK}\widehat{\otimes}_{\Es}\ainf$ for which we wish to verify the conditions \Crs.

With the goal of understanding the $G_K$--action on $\check{C}_{\inf}$ even more explicitly, in similar spirit to the proof of \cite[Proposition~3.13]{BhattScholze} we employ the following approximation of the prismatic envelope.

\begin{deff}
Let $B$ be a $\delta$--ring, $J \subseteq B$ an ideal with a fixed generating set $\underline{x}=\{x_i\}_{i \in \Lambda},$ and let $b \in J$ be an element. Denote by $\mathfrak{b}_0$ be the kernel of the $B$--algebra map
\begin{align*}
B[\underline{T}]=B[\{T_i\}_{i \in \Lambda}] &\longrightarrow B \left[\frac{1}{b}\right]\\
T_i &\longmapsto \frac{x_i}{b},
\end{align*}
and let $\mathfrak{b}$ be the $\delta$-ideal in $B\{\underline{T}\}$ generated by $\mathfrak{b}_0$.
Then we denote by $B\{\frac{\underline{x}}{b}\}$ the $\delta$--ring $B\{\underline{T}\}/\mathfrak{b}$, and call it \emph{weak $\delta$--blowup algebra of $\underline{x}$ and $b$}.
\end{deff}

That is, the above construction adjoins (in $\delta$--sense) the fractions $x_i/b$ to $B$ together with all relations among them that exist in $B[1/b]$, making it possible to naturally compute with fractions. 

Note that if $B \rightarrow C$ is a map of $B$--$\delta$--algebras such that $JC=bC$ and this ideal is invertible, the fact that the localization map $C\rightarrow C[\frac{1}{b}]$ is injective shows that there is a unique map of $B$--$\delta$--algebras $B\{\frac{\underline{x}}{b}\} \rightarrow C$. (In fact, if $b$ happens to be a non--zero divisor on $B\{\frac{\underline{x}}{b}\}$, then $B\{\frac{\underline{x}}{b}\}$ is initial among all such $B$--$\delta$--algebras; this justifies the name 'weak $\delta$--blowup algebra'.)

The purpose of the construction is the following.

\begin{prop}\label{ApproxEnvelopes}
Let $(A, I)$ be a bounded orientable prism with an orientation $d \in I$. Consider a map of $\delta$--pairs $(A, I) \rightarrow (B, J)$ and assume that $(C, IC)$ is a prismatic envelope for $(B, J)$ that is classically $(p, I)$--complete. Let $\underline{x}=\{x_i\}_{i \in \Lambda}$ be a system of generators of $J$. Then there is a surjective map of $\delta$--rings $\widehat{B\{\frac{\underline{x}}{d}\}}^{\mathrm{cl}} \rightarrow C$, where $\widehat{(-)}^{\mathrm{cl}}$ denotes the classical $(p, I)$--completion.
\end{prop}

Note that the assumptions apply to a \v{C}ech--Alexander cover in place of $(C, IC)$ since it is $(p, I)$--completely flat over the base prism, hence classically $(p, I)$--complete by \cite[Proposition~3.7]{BhattScholze}.

\begin{proof}
Since $JC=dC$ and and $d$ is a non--zero divisor on $C$, there is an induced map $B\{\frac{\underline{x}}{d}\}\rightarrow C$ and hence a map of $\delta$--rings $\widehat{B\{\frac{\underline{x}}{d}\}}^{\mathrm{cl}} \rightarrow C$ (using \cite[Lemma~2.17]{BhattScholze}).

To see that this map is surjective, let $C'$ denote its image in $C$, and denote by $\iota$ the inclusion of $C'$ into $C$. Then $C'$ is (derived, and, consequently, clasically) $(p, I)$--complete $A$--$\delta$--algebra with $C'[d]=0$. It follows that $(C', IC')=(C', (d))$ is a prism by \cite[Lemma~3.5]{BhattScholze} and thus, by the universal property of $C$, there is a map of $B$--$\delta$--algebras $r: C \rightarrow C'$ which is easily seen to be right inverse to $\iota$. Hence, $\iota$ is surjective, proving the claim. 
\end{proof}

Finally, we are ready to prove the following proposition which, as noted above, proves Theorem~\ref{thm:CrsForCechComplex}.

\begin{prop}\label{thm:CrsForACCoover}
The pair $\check{C}_{\BK} \rightarrow \check{C}_{\inf}$ satisfies the conditions \Crs for every $s \in \mathbb{N}\cup \{\infty\}$. 
\end{prop}

\begin{proof}
Fix a generating set $y_1, y_2, \dots, y_n$ of $J$, and set $P_1=\widehat{P^{\delta}}, P_1'=\widehat{(P')^{\delta}}$. We obtain a commutative diagram 

\begin{equation}\label{postcompletion}
\begin{tikzcd}
\widehat{P_1\{\frac{\underline{y}}{E(u)}\}} \ar[r] \ar[d] & \widehat{P_1'\{\frac{\underline{y}}{E(u)}\}} \ar[d] \\
\check{C}_{\BK} \ar[r] & \check{C}_{\inf},
\end{tikzcd}
\end{equation}
where the vertical maps are the surjective maps from Proposition~\ref{ApproxEnvelopes}, and the horizontal maps come from the $(p, E(u))$--completed base change $-\widehat{\otimes}_{\Es}\ainf$.

The $G_K$--action on $P_1'$ naturally extends to $P_1'\{\frac{\underline{y}}{E(u)}\}$ by the rule on generators 
$$g\left(\frac{y_j}{E(u)}\right)=\frac{g(y_j)}{g(E(u))}=\gamma^{-1}\frac{g(y_j)}{E(u)}$$ where $\gamma$ is the $\ainf$--unit such that $g(E(u))=\gamma E(u)$ (note that the fraction on the right--hand side makes sense as $g(y_j)\in JP_1'$). The action can be again extended continuously to the $(p, E(u))$--adic completion, and this action makes the right vertical map $G_K$--equivariant. 

It is therefore enough to prove the validity of the conditions \Crs for the pair $(\widehat{P_1\{\frac{\underline{y}}{E(u)}\}}, \widehat{P_1'\{\frac{\underline{y}}{E(u)}\}})$. By Lemma~\ref{generators} (3),(4), it is enough to check the conditions for the topological generators of $\widehat{P_1\{\frac{\underline{y}}{E(u)}\}}$ as an $\Es$--$\delta$--algebra, which are $X_1, X_2, \dots, X_m$ and $y_1/E(u), y_2/E(u), \dots, y_n/E(u).$ 

Fix an integer $s \in \mathbb{N}\cup \{\infty\}$. Since the elements $X_1, X_2, \dots, X_m$ satisfy $g(X_i)-X_i=0$ for every $g \in G_s$, by Lemma~\ref{generators} the pair $P \rightarrow P'$ satisfies the stronger condition \Crrs. In particular, \Crrs holds not only for the variables $X_i$, but also for $y_1, y_2, \dots, y_n$ since they come from $P$. 

Let us now fix an index $1\leq j \leq n$ and an element $g \in G_s$. We may write 
$$g(y_j)-y_j=\varphi^{s}(v)u z_j =\xi_{s, 0}uE(u)\tilde{z}_j$$ 
for some $z_j, \tilde{z_j} \in P'$ (that are equal up to a multiplication by an $\ainf$--unit). Similarly, we have 
$$g^{-1}(E(u))-E(u)=(\gamma^{-1}-1)E(u)=\varphi^{s}(v)u a=\xi_{s, 0}uE(u)\tilde{a}$$ 
with $a, \tilde{a} \in \ainf$ (again equal up to a unit). 

Thus, regarding the generator $y_j/E(u),$ we have that
$$g\left(\frac{y_j}{E(u)}\right)-\frac{y_j}{E(u)}=\frac{\gamma^{-1} g(y_j)-y_j}{E(u)}=\frac{\gamma^{-1} g(y_j)-\gamma^{-1}y_j+\gamma^{-1}y_j-y_j}{E(u)}=$$ $$=\gamma^{-1} \frac{g(y_j)-y_j}{E(u)}+(\gamma^{-1}-1)\frac{y_j}{E(u)}=\gamma^{-1}\xi_{s, 0}u\tilde{z}_j+\xi_{s, 0}u\tilde{a}y_j \in I_{s}.$$
This shows that each of the generators $y_j/E(u)$ is a $\Crs$--element, which finishes the proof.
\end{proof}

\subsection{Consequences for cohomology groups}

Let us now use Theorem~\ref{thm:CrsForCechComplex} to draw some conclusions for individual cohomology groups. The first is the crystalline condition for the prismatic cohomology groups and its consequence for $p$--adic \'{e}tale cohomology. As before, let $\mathscr{X}$ be a separated smooth $p$--adic formal scheme over $\oh_K$. Denote by $\mathscr{X}_{\ainf}$ the base change $\mathscr{X}\times_{\oh_K}\oh_{\mathbb{C}_K}=\mathscr{X}\times_{\Es}\ainf$, and by $\mathscr{X}_{\overline{\eta}}$ the geometric adic generic fiber.

\begin{cor}\label{cor:CrystallineForCohomologyGrps} For every $i \geq 0,$ the pair $H_{\Prism}^i(\mathscr{X}/\Es) \rightarrow H_{\Prism}^i(\mathscr{X}_{\ainf}/\ainf)$ satisfies the conditions \Cr{0} and \Cr{\infty}.
\end{cor}

\begin{proof}
By the results of Section~\ref{sec:CAComplex}, we may and do model the cohomology theories by the \v{C}ech--Alexander complexes $$\check{C}_{\BK}^{\bullet} \rightarrow \check{C}_{\inf}^{\bullet}=\check{C}_{\BK}^{\bullet}\widehat{\otimes}_{\Es}\ainf,$$ and by Theorem~\ref{thm:CrsForCechComplex} the conditions \Cr{0} and \Cr{\infty} termwise hold for this pair. The condition \Cr{\infty} for $H^i_{\Prism}(\mathscr{X}/\Es) \subseteq H^i_{\Prism}(\mathscr{X}/\ainf)^{G_{\infty}}$ thus follows immediately, and it remains to verify the crystalline condition.

Each of the terms $\check{C}_{\inf}^{i}$ is $(p, E(u))$--completely flat over $\ainf$, which means in particular that the terms $\check{C}_{\inf}^{i}$ are torsion--free by Corollary~\ref{FlatTorFree}. Denote the differentials on $\check{C}_{\BK}^{\bullet}, \check{C}_{\inf}^{\bullet}$ by $\partial$ and $\partial'$, resp.

To prove the crystalline condition for cohomology groups, it is clearly enough to verify the condition at the level of cocycles. Given $x \in Z^i(\check{C}_{\BK}^{\bullet}),$ denote by $x'$ its image in $Z^i(\check{C}_{\inf}^{\bullet})$. For $g \in G_K$ we have $g(x')-x'=\varphi^{-1}(v)uy'$ for some $y' \in \check{C}_{\inf}^{i}$. As $g(x')-x' \in Z^i(\check{C}_{\inf}^{\bullet}),$ we have
$$\varphi^{-1}(v)u \partial'(y')=\partial'(\varphi^{-1}(v)u y')=0,$$
and the torsion--freeness of $\check{C}_{\inf}^{i+1}$ implies that $\partial'(y')=0$. Thus, $y' \in Z^i(\check{C}_{\inf})$ as well, showing that $g(x')-x' \in \varphi^{-1}(v)uZ^i(\check{C}_{\inf}^{\bullet}),$ as desired.
\end{proof}

When $\mathscr{X}$ is proper over $\oh_K$, we use Corollary~\ref{cor:CrystallineForCohomologyGrps} to reprove the result from \cite{BMS1} that the \'{e}tale cohomology groups $\H^{i}_{\et}(\mathscr{X}_{\overline{\eta}}, \mathbb{Q}_p)$ are in this case crystalline representations.

\begin{cor}\label{cor:EtaleCohomologyCrystalline}
Assume that $\mathscr{X}$ is additionally proper over $\oh_K.$  Then for any $i \geq 0,$ the $p$--adic \'{e}tale cohomology $\H_{\et}^i(\mathscr{X}_{\overline{\eta}}, \mathbb{Q}_p)$ is a crystalline representation.
\end{cor}

\begin{proof}
It follows from \cite[Theorem~1.8]{BhattScholze} (and faithful flatness of $\ainf/\Es$)  that $M_{\BK}=\H_{\Prism}^i(\mathscr{X}/ \Es)$ and $M_{\inf}=\H_{\Prism}^i(\mathscr{X}_{\ainf}/ \ainf)$ are Breuil--Kisin and Breuil--Kisin--Fargues modules, resp., such that $M_{\inf}=M_{\BK}\otimes_{\Es}\ainf$. Moreover, $M_{\inf}$ has the structure of a Breuil--Kisin--Fargues $G_K$--module with 
$$V(M_{\inf}):=\left(W(\mathbb{C}_K^{\flat})\otimes_{\ainf}  M_{\inf}\right)^{\varphi=1}\left[\frac{1}{p}\right]\simeq \H_{\et}^i(\mathscr{X}_{\overline{\eta}}, \mathbb{Q}_p)$$ 
as $G_K$--representations. By Corollary~\ref{cor:CrystallineForCohomologyGrps}, the pair $(M_{\BK}, M_{\inf})$ satisfies all the assumptions of Theorem~\ref{BKBKFCrystallineGeneralThm}. The claim thus follows.
\end{proof}

For the purposes of obtaining a bound on ramification of $p$--torsion \'{e}tale cohomology in \S\ref{sec:bounds}, let us recall the notion of torsion prismatic cohomology as defined in \cite{LiLiu}, and discuss the consequences of the conditions \Crs in this context.

\begin{deff}
Given a bounded prism $(A, I)$ and a smooth $p$--adic formal scheme $\mathscr{X}$ over $A/I$, the \textit{$p^n$--torsion prismatic cohomology} of $\mathscr{X}$ is defined as $$\R\Gamma_{\Prism, n}(\mathscr{X}/ A)=\R\Gamma_\Prism (\mathscr{X}/ A)\stackrel{{\mathsf{L}}}{\otimes}_{\mathbb{Z}}\mathbb{Z}/p^n\mathbb{Z}.$$
We denote the cohomology groups of $\R\Gamma_{\Prism, n}(\mathscr{X}/ A)$ by $\H^i_{\Prism, n}(\mathscr{X}/A)$ (and refer to them as $p^n$--torsion prismatic cohomology groups).
\end{deff}

\begin{prop}\label{prop:CrystallineCohMopPN}
Let $s, n$ be a pair of integers satisfying $s\geq 0, n \geq 1$. Set $t=\mathrm{max}\left\{0, s+1-n\right\}.$ Then the torsion prismatic cohomology groups $\H^i_{\Prism, n}(\mathscr{X}/\Es)\rightarrow \H^i_{\Prism, n}(\mathscr{X}_{\ainf}/\ainf)$ satisfy the following condition:
$$\forall g \in G_s: \;\;\; (g-1)\H^i_{\Prism, n}(\mathscr{X}/\Es) \subseteq \varphi^{-1}(v)u^{p^{t}}\H^i_{\Prism, n}(\mathscr{X}_{\ainf}/\ainf).$$  
\end{prop}

\begin{proof} The proof is a slightly refined variant of the proof of Corollary~\ref{cor:CrystallineForCohomologyGrps}. Consider again the associated \v{C}ech--Alexander complexes over $\Es$ and $\ainf$, 
$$\check{C}_{\BK}^{\bullet} \rightarrow \check{C}_{\inf}^{\bullet}=\check{C}_{\BK}^{\bullet}\widehat{\otimes}_{\Es}\ainf.$$
Both of these complexes are given by torsion--free, hence $\mathbb{Z}$--flat, modules by Corollary~\ref{FlatTorFree}. Consequently, $\R\Gamma_{\Prism, n}(\mathscr{X}/\Es)$ is modelled by $\check{C}_{\BK, n}^{\bullet}:=\check{C}_{\BK}^{\bullet}/p^n\check{C}_{\BK}^{\bullet}$, and similarly for $\R\Gamma_{\Prism, n}(\mathscr{X}_{\ainf}/\ainf)$ and $\check{C}_{\inf, n}^{\bullet}=\check{C}_{\inf}^{\bullet}/p^n\check{C}_{\inf}^{\bullet}$. That is, the considered maps between cohomology groups are obtained as the maps on cohomologies for the base--change map of chain complexes
$$\check{C}_{\BK, n}^{\bullet}\rightarrow \check{C}_{\inf, n}^{\bullet}=\check{C}_{\BK, n}^{\bullet}\widehat{\otimes}_{\Es}\ainf,$$
and as in the proof of Corolary~\ref{cor:CrystallineForCohomologyGrps}, it is enough to establish the desired condition for the respective groups of cocycles.

Set $\alpha=\varphi^{-1}(v)u^{p^{t}}$. Note that by Lemma~\ref{lem:IsModPn}, the condition \Crs for the pair of complexes $\check{C}_{\BK, n}^{\bullet}\rightarrow \check{C}_{\inf, n}^{\bullet}$ implies the condition
$$\forall g \in G_{s}:\;\;\; (g-1)\check{C}_{\BK, n}^{\bullet} \subseteq \alpha\check{C}_{\inf, n}^{\bullet}$$
(meant termwise as usual), and since the terms of the complex $\check{C}_{\inf}^{\bullet}$ are $(p, E(u))$--complete and $(p, E(u))$--completely flat, $\alpha$ is a non--zero divisor on the terms of $\check{C}_{\inf, n}^{\bullet}$ by Corollary~\ref{FlatTorFree}. 

So pick any element $x \in Z^{i}(\check{C}_{\BK, n}^{\bullet})$. The image $x'$ of $x$ in $\check{C}_{\inf, n}^i$ lies in $Z^i(\check{C}_{\inf, n}^\bullet)$ and for any chosen $g \in G_s $ we have $g(x')-x'=\alpha y'$ for some $y'\in \check{C}_{\inf, n}^i$. Now $g(x')-x'$ lies in $Z^i(\check{C}_{\inf, n}^{\bullet}),$ so $\alpha y'=g(x')-x' $ satisfies
$$0=\partial'(\alpha y')=\alpha\partial'(y').$$
Since $\alpha$ is a non--zero divisor on $\check{C}_{\inf, n}^{i+1}$, it follows that $\partial'(y')=0$, that is, $y'$ lies in $Z^i(\check{C}_{\inf, n}^{\bullet}).$ We thus infer that $g(x')-x'=\alpha y'\in\alpha Z^i(\check{C}_{\inf, n}^{\bullet}),$ as desired.
\end{proof}

\section{Ramification bounds for mod $p$ \'{e}tale cohomology}\label{sec:bounds}

\subsection{Ramification bounds}
We are ready to discuss the implications to the question of ramification bounds for mod $p$ \'{e}tale cohomology groups $\H_{\et}^i(\mathscr{X}_{\overline{\eta}}, \mathbb{Z}/p\mathbb{Z})$ when $\mathscr{X}$ is smooth and proper $p$--adic formal scheme over $\oh_K$.

We define an additive valuation $v^{\flat}$ on $\oh_{\mathbb{C}_K}^{\flat}$ by $v^\flat(x)=v(x^{\sharp})$ where $v$ is the valuation on $\oh_{\mathbb{C}_K}$ normalized so that $v(\pi)=1$, and $(-)^{\sharp}:\oh_{\mathbb{C}_K}^{\flat}\rightarrow \oh_{\mathbb{C}_K}$ is the multiplicative lift of $\mathrm{pr}_0: \oh_{\mathbb{C}_K}^{\flat} \rightarrow \oh_{\mathbb{C}_K}/p$. This way, we have $v^{\flat}(\underline{\pi})=1$ and $v^\flat(\underline{\varepsilon}-1)=pe/(p-1)$. For a real number $c\geq 0$, denote by $\mathfrak{a}^{>c}$ ($\mathfrak{a}^{\geq c},$ resp.) the ideal of $\oh_{\mathbb{C}_K}^{\flat}$ formed by all elements $x$ with $v^{\flat}(x)>c$ ($v^{\flat}(x)\geq c$, resp.). 

Similarly, we fix an additive valuation $v_K$ of $K$ normalized by $v_K(\pi)=1$. Then for an algebraic extension $L/K$ and a real number $c \geq 0$, we denote by $\mathfrak{a}_L^{> c}$ the ideal consisting of all elements $x \in \oh_L$ with $v_K(x)>c$ (and similarly, for '$\geq$' as well).  

For a finite extensions $M/F/K$ and a real number $m \geq 0$, let us recall (a version of\footnote{Fontaine's original condition uses the ideals $\mathfrak{a}_E^{\geq m}$ instead. Up to changing some inequalities from `$<$' to `$\leq$' and vice versa, the conditions are fairly equivalent.}) Fontaine's property $(P_m^{M/F})$:
$$\begin{array}{cc}
(P_m^{M/F}) & \begin{array}{l}\text{For any algebraic extension }E/F,\text{ the existence of an }\oh_F\text{--algebra map}\\ \oh_M\rightarrow \oh_E/\mathfrak{a}_E^{>m}\text{ implies the existence of an }F\text{--injection of fields }M\hookrightarrow E.\end{array}
\end{array}$$

We also recall the upper ramification numbering in the convention used in \cite{Fontaine, CarusoLiu}. For $G=\mathrm{Gal}(M/F)$ and a non--negative real number $\lambda,$ set $$G_{(\lambda)}=\{ g \in G \;|\; v_M(g(x)-x)\geq \lambda \;\;\forall x \in \oh_M\},$$
where $v_M$ is again the additive valuation of $M$ normalized by $v_M(M^{\times})=\mathbb{Z}$.

For $t\geq 0,$ set 
$$\phi_{M/F}(t)=\int_0^t \frac{\mathrm{d}t}{[G_{(1)}:G_{(t)}]}$$
(which makes sense as $G_{(t)}\subseteq G_{(1)}$ for all $t>1$). Then $\phi_{M/F}$ is a piecewise--linear increasing continuous concave function. Denote by $\psi_{M/F}$ its inverse, and set $G^{(\mu)}=G_{(\psi_{M/F}(\mu))}.$ 

Denote by $\lambda_{M/F}$ the infimum of all $\lambda \geq 0$ such that $G_{(\lambda)}=\{\mathrm{id}\},$ and by $\mu_{M/F}$ the infimum of all $\mu \geq 0$ such that $G^{(\mu)}=\{\mathrm{id}\}.$ Clearly one has $\mu_{M/F}=\phi_{M/F}(\lambda_{M/F}).$

\begin{rem}
Let us compare the indexing conventions with \cite{SerreLocalFields} and \cite{Fontaine}, as the results therein are (implicitly or explicitly) used. If $G^{\text{S-}(\mu)}, G^{\text{F-}(\mu)}$ are the upper--index ramification groups in \cite{SerreLocalFields} and \cite{Fontaine}, resp., and similarly we denote $G_{\text{S-}(\lambda)}$ and $G_{\text{F-}(\lambda)}$ the lower--index ramification groups, then we have
$$G^{(\mu)}=G^{\text{S-}(\mu-1)}=G^{\text{F-}(\mu)}, \;\;\;\; G_{(\lambda)}=G_{\text{S-}(\lambda-1)}=G_{\text{F-}(\lambda/\tilde{e})},$$
where $\tilde{e}=e_{M/F}$ is the ramification index of $M/F$.
\end{rem}

In particular, since the enumeration differs from the one in \cite{SerreLocalFields} only by a shift by one, the upper indexing is still compatible with passing to quotients, and it make sense to set 
$$G_{F}^{(\mu)}=\varprojlim_{M'/F}\mathrm{Gal}(M'/F)^{(\mu)}$$
where $M'/F$ varies over finite Galois extensions $M'/F$ contained in a fixed algebraic closure $\overline{K}$ of $K$ (and $G_F=\varprojlim_{M'/F}\mathrm{Gal}(M'/F)$ is the absolute Galois group).

Regarding $\mu$, the following transitivity formula is useful.

\begin{lem}[{\cite[Lemma 4.3.1]{CarusoLiu}}]\label{lem:transitivity}
Let $N/M/F$ be a pair of finite extensions with both $N/F$ and $M/F$ Galois. Then we have $\mu_{N/F}=\mathrm{max}(\mu_{M/F}, \phi_{M/F}(\mu_{N/M})).$
\end{lem}

The property $(P^{M/F}_m)$ is connected with the ramification of the field extension $M/F$ as follows.

\begin{prop}\label{prop:RamificationEngine}
Let $M/F/K$ be finite extensions of fields with $M/F$ Galois and let $m>0$ be a real number. If the property $(P^{M/F}_m)$ holds, then:
\begin{enumerate}[(1)]
\item{{\normalfont(\cite[Proposition~3.3]{Yoshida})} $\mu_{M/F}\leq e_{F/K}m.$ In fact, $\mu_{M/F}/e_{F/K}$ is the infimum of all $m>0$ such that $(P^{M/F}_m)$  is valid.}
\item{{\normalfont(\cite[Corollary~4.2.2]{CarusoLiu})} $v_K(\mathcal{D}_{M/F})<m,$ where $\mathcal{D}_{M/F}$ denotes the different of the field extension $M/F$.}
\end{enumerate}
\end{prop}

\begin{cor}\label{cor:FontainePropertyWlogTotRamified}
Both the assumptions and the conclusions of Proposition~\ref{prop:RamificationEngine} are insensitive to replacing $F$ by any unramified extension of $F$ contained in $M$.
\end{cor}

\begin{proof}
Let $F'/F$ be an unramified extension such that $F' \subseteq M$. The fact that $(P^{M/F}_m)$ is equivalent to $(P^{M/F'}_m)$ is proved in  \cite[Proposition~2.2]{Yoshida}. To show that also the conclusions are the same for $F$ and $F'$, it is enough to observe that $e_{F'/K}=e_{F/K}, e_{M/F'}=e_{M/F},$ $v_{K}(\mathcal{D}_{M/F'})=v_{K}(\mathcal{D}_{M/F})$ and $\mu_{M/F'}=\mu_{M/F}$. The first two equalities are clear since $F'/F$ is unramified. The third equality follows from $\mathcal{D}_{M/F}=\mathcal{D}_{M/F'}\mathcal{D}_{F'/F}$ upon noting that $\mathcal{D}_{F'/F}$ is the unit ideal. Finally, by Lemma~\ref{lem:transitivity}, we have $\mu_{M/F}=\mathrm{max}(\mu_{F'/F}, \phi_{F'/F}(\mu_{M/F'})).$ As $F'/F$ is unramified, we have $\mu_{F'/F}=0$ and $\phi_{F'/F}(t)=t$ for all $t \geq 0$. The fourth equality thus follows as well. \end{proof}

Let $\mathscr{X}$ be a proper and smooth $p$--adic formal scheme over $\oh_K$. Fix the integer $i$, and denote by $T'$ the Galois module $\H^{i}_{\et}(\mathscr{X}_{\overline{\eta}}, \mathbb{Z}/p\mathbb{Z})$. Let $L$ be the splitting field of $T'$, i.e. $L=\overline{K}^{\mathrm{Ker}\,\rho}$ where $\rho: G_K\rightarrow \mathrm{Aut}_{\mathbb{F}_p}(T')$ is the associated representation. The goal is to provide an upper bound on $v_K(\mathcal{D}_{L/K})$, and a constant $\mu_0=\mu_0(e, i, p)$ such that $G_K^{(\mu)}$ acts trivially on $T'$ for all $\mu>\mu_0$. 

To achieve this, we follow rather closely the strategy of \cite{CarusoLiu}. The main difference is that the input of $(\varphi, \widehat{G})$--modules attached to the discussed $G_K$--respresentations in \cite{CarusoLiu} is in our situation replaced by a $p$--torsion Breuil--Kisin module and a Breuil--Kisin--Fargues $G_K$--module that arise as the $p$--torsion prismatic $\Es$-- and $\ainf$--cohomology, resp. Let us therefore lay out the strategy, referring to proofs in \cite{CarusoLiu} whenever possible, and describe the needed modifications where necessary. To facilitate this approach further, the notation used will usually reflect the notation of \cite{CarusoLiu}, except for mostly omitting the index $n$ throughout (which in our situation is always equal to $1$).

The relation of the above--mentioned $p$--torsion prismatic cohomologies to the $p$--torsion \'{e}tale cohomology is as follows.

\begin{prop}[{\cite[Proposition~7.2, Corollary~7.4, Remark~7.5]{LiLiu}}] \label{prop:TorsionSetup}
Let $\mathscr{X}$ be a smooth and proper $p$--adic formal scheme over $\oh_K$. Then  

\begin{enumerate}[(1)]
\item{$M_{\BK}=\H^i_{\Prism, n}(\mathscr{X}/\Es)$ is a $p^n$--torsion Breuil--Kisin module of height $\leq i$, and we have $$\H_{\et}^i(\mathscr{X}_{\overline{\eta}}, \mathbb{Z}/p^n\mathbb{Z})\simeq T_n(M_{\BK}):=\left(M_{\BK}\otimes_{W_n(k)[[u]]} W_n(\mathbb{C}_K^{\flat})\right)^{\varphi=1}$$
as $\mathbb{Z}/p^n\mathbb{Z}[G_\infty]$--modules.}
\item{$M_{\inf}=\H^i_{\Prism, n}(\mathscr{X}_{\ainf}/\ainf)$ is a $p^n$--torsion Breuil--Kisin--Fargues $G_K$--module of height $\leq i$, and we have $$\H_{\et}^i(\mathscr{X}_{\overline{\eta}}, \mathbb{Z}/p^n\mathbb{Z})\simeq T_n(M_{\inf}):=\left(M_{\inf}\otimes_{W_n(\oh_{\mathbb{C}_K^{\flat}})} W_n(\mathbb{C}_K^{\flat})\right)^{\varphi=1}$$
as $\mathbb{Z}/p^n\mathbb{Z}[G_K]$--modules.}
\item{We have $M_{\BK}\otimes_{\Es}\ainf=M_{\BK}\otimes_{W_n(k)[[u]]}W_n(\mathcal{O}_{\mathbb{C}_K^{\flat}})\simeq M_{\inf},$ and the natural map $M_{\BK} \hookrightarrow M_{\inf}$ has the image contained in $M_{\inf}^{G_{\infty}}$.}
\end{enumerate}
\end{prop}

So let $M^0_{\BK}=\H^i_{\Prism, 1}(\mathscr{X}/\Es)$ and $M_{\inf}^0=\H^i_{\Prism, 1}(\mathscr{X}_{\ainf}/\ainf)$, so that $T_1(M_{\BK}^0)=T_1^{\inf}(M_{\inf}^0)=\H_{\et}^i(\mathscr{X}_{\overline{\eta}}, \mathbb{Z}/p\mathbb{Z}).$ Observe further that, since $u$ is a unit of $W_1(\mathbb{C}_{K}^{\flat})=\mathbb{C}_{K}^{\flat},$ we have $T_1(M_{\BK}^0)=T_1(M_{\BK})$ and $T_1^{\inf}(M_{\inf}^0)=T_1^{\inf}(M_{\inf}),$ where $M_{\BK}=M_{\BK}^0/M_{\BK}^0[u^{\infty}]$ and $M_{\inf}=M_{\inf}^0/M_{\inf}^0[u^{\infty}]$ are again a Breuil--Kisin module and a Breuil--Kisin--Fargues $G_K$--module, resp., of height $\leq i$. Since $\Es\hookrightarrow \ainf$ is faithfully flat, it is easy to see that the isomorphism $M_{\inf}\simeq M_{\BK}\otimes_{\Es}\ainf$ remains true. Furthermore, the pair $(M_{\BK}, M_{\inf})$ satisfies the conditions
\begin{equation}\label{eqn:CrsModP}
\forall g \in G_s\;\;\forall x\in M_{\BK}:\;\; g(x)-x \in \varphi^{-1}(v)u^{p^s}M_{\inf}
\end{equation}
for all $s \geq 0$, since the pair $(M_{\BK}^0, M_{\inf}^0)$ satisfies the analogous conditions by Proposition~\ref{prop:CrystallineCohMopPN}. Finally, the module $M_{\BK}$ is finitely generated and $u$--torsion--free $k[[u]]$--module, hence a finite free $k[[u]]$--module (and, consequently, $M_{\inf}$ is a finite free $\oh_{\mathbb{C}_K^{\flat}}$--module). 

Instead of using $T'=T_1(M_{\inf})=\H_{\et}^i(\mathscr{X}_{\overline{\eta}}, \mathbb{Z}/p^n\mathbb{Z})$ directly, we work with the dual module $$T:=T^{*, \inf}_1(M_{\inf})=\mathrm{Hom}_{\ainf, \varphi}(M_{\inf}, \oh_{\mathbb{C}_K^{\flat}})\simeq \H_{\et}^i(\mathscr{X}_{\overline{\eta}}, \mathbb{Z}/p\mathbb{Z})^{\vee}$$ instead; this is equivalent, as the splitting field of $T$ is still $L$. Note that $$T\simeq T^*_1(M_{\BK})=\mathrm{Hom}_{\Es, \varphi}(M_{\BK}, \oh_{\mathbb{C}_K^{\flat}})$$ as a $\mathbb{Z}/p\mathbb{Z}[G_\infty]$--module. 

\begin{rem}[Ramification bounds of \cite{Caruso}]\label{rem:CarusoBound}
Similarly to the discussion above we may take, for any $n \geq 1,$  $M^0_{\BK}=\H^i_{\Prism, n}(\mathscr{X}/\Es), $ and $M_{\BK}=M^0_{\BK}/M^0_{\BK}[u^\infty]$. Then the $G_\infty$--module 
$$T:=T^*_n(M_\BK)=\mathrm{Hom}_{\Es, \varphi}(M_{\BK}, W_n(\oh_{\mathbb{C}_K^{\flat}}))$$
 is the restriction of $\H^i_{\et}(\mathscr{X}_{\overline{\eta}}, \mathbb{Z}/p^n\mathbb{Z})^{\vee}$ to $G_\infty$. Denoting by $\oh_{\mathcal{E}}$ the $p$--adic completion of $\Es[1/u]$, $M_{\mathcal{E}}:=M_{\BK}\otimes_{\Es}\oh_{\mathcal{E}}$ then becomes an \'{e}tale $\varphi$--module over $\oh_{\mathcal{E}}$ in the sense of \cite[\S A]{Fontaine3}, with the natural map $M_\BK \rightarrow M_{\mathcal{E}}$ injective; thus, in terminology of \cite{Caruso}, $M_{\BK}$ serves as a $\varphi$--lattice of height dividing $E(u)^i$. Upon observing that $T$ is the $G_\infty$--respresentation associated with $M_{\mathcal{E}}$ (see e.g. \cite[\S 2.1.3]{Caruso}), Theorem~2 of \cite{Caruso} implies the ramification bound
$$\mu_{L/K}\leq 1+c_0(K)+e\left(s_0(K)+\mathrm{log}_p(ip)\right)+\frac{e}{p-1}.$$
Here $c_0(K), s_0(K)$ are constants that depend on the field $K$ and that generally grow with increasing $e$. (Their precise meaning is described in \S~\ref{subsec:Comparisons}.)
\end{rem}

We employ the following approximations of the functors $T_1^{*}$ and $T_1^{*, \inf}$.

\begin{nott}
For a real number $c\geq 0$, we define 
$$J_c(M_{\BK})=\mathrm{Hom}_{\Es, \varphi}(M_{\BK}, \oh_{\mathbb{C}_K^{\flat}}/\mathfrak{a}^{>c}),$$
$$J^{\inf}_c(M_{\inf})=\mathrm{Hom}_{\ainf, \varphi}(M_{\inf}, \oh_{\mathbb{C}_K^{\flat}}/\mathfrak{a}^{>c}).$$
We further set $J_\infty(M_{\BK})=T_1^{*}(M_{\BK})$ and $J^{\inf}_\infty(M_{\inf})=T_1^{*, \inf}(M_{\inf})$.
Given $c, d \in \mathbb{R}^{\geq 0}\cup \{\infty\}$ with $c \geq d,$ we denote by $\rho_{c, d}: J_c(M_{\BK})\rightarrow J_d(M_{\BK})$ ($\rho^{\inf}_{c, d}: J^{\inf}_c(M_{\inf})\rightarrow J^{\inf}_d(M_{\inf}),$ resp.) the map induced by the quotient map $\oh_{\mathbb{C}_K^{\flat}}/\mathfrak{a}^{>c}\rightarrow \oh_{\mathbb{C}_K^{\flat}}/\mathfrak{a}^{>d}$.
\end{nott}

Since $M_{\inf}\simeq M_{\BK}\otimes_{\Es}\ainf$ as $\varphi$--modules, it is easy to see that for every $c \in \mathbb{R}^{\geq 0}\cup \{\infty\},$ we have a natural isomorphism $\theta_c:J_c(M_{\BK}) \stackrel{\simeq}{\rightarrow} J_c^{\inf}(M_{\inf})$ of abelian groups; the biggest point of distinction between the two is that $J_c^{\inf}(M_{\inf})$ naturally attains the action of $G_K$ from the one on $M_{\inf}$, by the usual rule $$g(f)(x):=g(f(g^{-1}(x))),\;\; g \in G_K,\; f \in J_c^{\inf}(M_{\inf}),\;x \in M_{\inf}.$$

As for $J_c(M_{\BK}),$ there is a natural action given similarly by the formula 
$g(f)(x):=g(f(x))$ where $f \in J_c(M_{\BK})$ and $x \in M_{\BK}$. However, in order for this action to make sense, one needs that each $g(f)$ defined this way is still an $\Es$--linear map, which comes down to requiring that $g(u)=u$ (that is, $g(\underline{\pi})=\underline{\pi}$) in the ring $ \oh_{\mathbb{C}_K^{\flat}}/\mathfrak{a}^{>c}$. This condition holds for $g \in G_{s}$ for $s$ depending on $c$ as follows.

\begin{prop}[{\cite[Proposition~2.5.3]{CarusoLiu}}]\label{prop:GsActionOnJc}
Let $s$ be a non--negative integer with $s>\mathrm{log}_p(\frac{c(p-1)}{ep})$. Then the natural action of $G_s$ on $\oh_{\mathbb{C}_K^{\flat}}/\mathfrak{a}^{>c}$ induces an action of $G_s$ on $J_c(M_{\BK}).$ Furthermore, when $ d \leq c$, the map $\rho_{c, d}:J_c(M_{\BK}) \rightarrow J_d(M_{\BK})$ is $G_s$--equivariant, and when $s' \geq s$, the $G_{s'}$--action on $J_c(M_{\BK})$ defined in this manner is the restriction of the $G_s$--action to $G_{s'}$.
\end{prop}

The crucial connection between the actions on $J_c(M_{\BK})$ and $J_c^{\inf}(M_{\inf})$ is established using (the consequences of) the conditions \Crs.

\begin{prop}\label{prop:GsEquivarianceOfJc}
For $$s > \mathrm{max}\left\{\mathrm{log}_p\left(\frac{c(p-1)}{ep}\right),\mathrm{log}_p\left(c-\frac{e}{p-1}\right)\right\},$$ the natural isomorphism $\theta_c:J_c(M_{\BK})\stackrel{\simeq}\rightarrow J_c^{\inf}(M_{\inf})$ is $G_s$--equivariant.
\end{prop}

\begin{proof}
Identifying $M_{\inf}$ with $M_{\BK}\otimes_{\Es}\ainf,$ $\theta_c$ takes the form $f \mapsto \widetilde{f}$ where $\widetilde{f}(x \otimes a):=af(x)$ for $x \in M_{\BK}$ and $a \in \ainf$. Note that we have $\varphi^{-1}(v)u^{p^s}\oh_{\mathbb{C}_K^{\flat}}=\mathfrak{a}^{\geq p^s + e/(p-1)}$. The condition (\ref{eqn:CrsModP}) then states that for all $x \in M_{\BK}$ and all $g \in G_s, $ $g(x\otimes 1)-x\otimes 1$ lies in $\mathfrak{a}^{\geq p^s + e/(p-1)}M_{\inf}$ and therefore in $\mathfrak{a}^{>c}M_{\inf}$ thanks to the assumption on  $s$. It then follows that for every $\widetilde{f}\in J_c^{\inf}(M_{\inf}),$ $\widetilde{f}(g(x\otimes 1))=\widetilde{f}(x\otimes 1),$ and hence $$g(\widetilde{f})(x \otimes a)=g\left(\widetilde{f}(g^{-1}(x \otimes a))\right)=g\left(g^{-1}(a)\widetilde{f}(g^{-1}(x \otimes 1))\right)=ag\left(\widetilde{f}(x \otimes 1)\right)=ag(f(x))$$ for every $g \in G_s,$ $x \in M_{\BK}$ and $a \in \ainf$. Thus, we have that $g(\widetilde{f})=\widetilde{g(f)}$ for every $g \in G_s$ and $f \in J_c(M_{\BK}),$ proving the equivariance of $\theta_c$.
\end{proof}

From now on, set $b:=ie/(p-1)$ and $a:=iep/(p-1)$. Then $T$ is determined by $J_a(M), J_b(M)$ in the following sense. 

\begin{prop}\label{prop:ActionProlongationJc}
\begin{enumerate}[(1)]
\item{The map $\rho_{\infty, b}: T^{*}_1(M_{\BK})\rightarrow J_b(M_{\BK})$ 
is injective, with the image equal to $\rho_{a, b}(J_{a}(M_{\BK}))\subseteq J_b(M_{\BK})$.}
\item{Similarly, the map $\rho_{\infty, b}^{\inf}: T^{*, \inf}_1(M_{\inf})\rightarrow J^{\inf}_b(M_{\inf})$ is injective with $\rho^{\inf}_{\infty, b}(M_{\inf})=\rho^{\inf}_{a, b}(J^{\inf}_{a}(M_{\inf}))$.} 
\item{For $s>\mathrm{log}_p(i) $, $T^{*}_1(M_{\BK})$ has a natural action of $G_s$ that extends the usual $G_\infty$--action.}
\item{For $s> \mathrm{max}\left( \mathrm{log}_p(i),\mathrm{log}_p((i-1)e/(p-1)) \right)$, the action from (3) agrees with $T|_{G_s}$.}
\end{enumerate}  
\end{prop}

\begin{proof}
Part (1) is proved in \cite[Proposition~2.3.3]{CarusoLiu}. Then $T^{*}_1(M_{\BK})$ attains the action of $G_s$ with $s>\mathrm{log}_p(i)$ by identification with $\rho_{a, b}(J_{a}(M_{\BK}))$ and using Proposition~\ref{prop:GsActionOnJc} (see also \cite[Theorem~2.5.5]{CarusoLiu}), which proves (3). Finally, the proof of (2),(4) is analogous to \cite[Corollary~3.3.3]{CarusoLiu} and \cite[Theorem~3.3.4]{CarusoLiu}. Explicitly, consider the commutative diagram  

\begin{center}
\begin{tikzcd}
T_1^{*}(M_{\BK}) \ar[r, "\rho_{\infty,a}"] \ar[d, "\sim"', "\theta_{\infty}"] & J_a(M_{\BK}) \ar[d, "\sim"', "\theta_{a}"] \ar[r, "\rho_{a, b}"] & J_b(M_{\BK}) \ar[d, "\sim"', "\theta_b"] \\
T_1^{*, \inf}(M_{\inf}) \ar[r, "\rho^{\inf}_{\infty,a}"]  & J^{\inf}_a(M_{\inf})  \ar[r, "\rho^{\inf}_{a, b}"] & J^{\inf}_b(M_{\inf}) ,
\end{tikzcd}
\end{center}
where the composition of the rows gives $\rho_{\infty, b}$ and $\rho^{\inf}_{\infty, b},$ resp. This immediately proves (2) using (1). Finally, the map $\rho^{\inf}_{\infty, b}$ is $G_K$--equivariant, the map $\rho_{\infty, b}$ is (tautologically) $G_s$--equivariant for $s > \mathrm{log}_p(i)$ by the proof of (3), and both maps are injective. Since $\theta_b$ is $G_s$--equivariant when $s > \mathrm{log}_p((i-1)e/(p-1))$ by Proposition~\ref{prop:GsEquivarianceOfJc}, it follows that so is $\theta_{\infty}$, which proves (4).
\end{proof}

We employ further approximations of $J_c(M_{\BK})$ defined as follows.

\begin{nott} 
Let $s$ be a non--negative integer, consider a real number $c \in [0, ep^s)$ and an algebraic extension $E/K_s$. We consider the ring $(\varphi_k^s)^*\oh_{E}/\mathfrak{a}_E^{> c/p^s}=k\otimes_{\varphi_k^s, k}\oh_{E}/\mathfrak{a}_E^{> c/p^s}$ (note that the condition on $c$ implies that $p \in \mathfrak{a}_E^{>c/p^s}$, making $\oh_{E}/\mathfrak{a}_E^{>c/p^s}$ a $k$--algebra). We endow this ring with an $\Es$--algebra structure via $\Es\stackrel{\mathrm{mod}\,p}{\rightarrow}k[[u]]\stackrel{\alpha}{\rightarrow} (\varphi_k^s)^*\oh_{E}/\mathfrak{a}_E^{> c/p^s}$ where $\alpha$ extends the $k$--algebra structure map by the rule $u \mapsto 1\otimes{\pi_s}.$ Then we set
$$J_c^{(s), E}(M_{\BK})=\mathrm{Hom}_{\Es,\varphi}(M_{\BK}, (\varphi_k^s)^*\oh_{E}/\mathfrak{a}_E^{>c/p^s}).$$
In the case when $E/K_s$ is Galois, the module $J_c^{(s), E}(M_{\BK})$ attains a $G_s$--action induced by the $G_s$--action on $\oh_{E}/\mathfrak{a}_E^{>c/p^s}$.

When $c, d$ are two real numbers satisfying $ep^s> c \geq d \geq 0,$ there is an obvious transition map $\rho_{c, d}^{(s), E}(M_{\BK}):J_c^{(s), E}(M_{\BK}) \rightarrow J_d^{(s), E}(M_{\BK})$, which is $G_s$--equivariant in the Galois case.
\end{nott}

The relation to $J_c(M_{\BK})$ is the following.

\begin{prop} \label{prop:ApproximateJc}
Let $s, c$ be as above. Then
\begin{enumerate}[(1)]
\item{Given an algebraic extension $E/K_s$, $J_c^{(s), E}(M_{\BK})$ naturally embeds into $J_c(M_{\BK})$ as a submodule ($G_s$--submodule when $E/K_s$ is Galois).}
\item{Given a tower of algebraic extensions $F/E/K_s$, $J_c^{(s), E}(M_{\BK})$ naturally embeds into $J_c^{(s), F}(M_{\BK})$ as a submodule ($G_s$--submodule if both $E/K_s$ and $F/K_s$ are Galois).}
\item{$J_c^{(s), \overline{K}}(M_{\BK})$ is naturally isomorphic to $J_c(M_{\BK})$ as a $G_s$--module.}
\end{enumerate}
\end{prop}

\begin{proof}
Part (2) follows immediately from the observation that the map $\oh_{E}/\mathfrak{a}_E^{>c/p^s} \rightarrow \oh_{F}/\mathfrak{a}_F^{>c/p^s}$ induced by the inclusion $\oh_{E} \hookrightarrow \oh_{F}$ remains injective (and is clearly $G_s$--equivariant in the Galois case). Similarly, part (3) follows from the fact that the map $\mathrm{pr}_s: \mathcal{O}_{\mathbb{C}_K^{\flat}}=\varprojlim_{s, \varphi}\oh_{\overline{K}}/p \rightarrow \oh_{\overline{K}}/p$ induces a ($G_s$--equivariant) isomorphism $ \oh_{\mathbb{C}_K^{\flat}}/\mathfrak{a}^{>c} \rightarrow (\varphi_k^s)^*\oh_{\overline{K}}/\mathfrak{a}_{\overline{K}}^{>c/p^s}$ when $s > \mathrm{log}_p(c/e)$ (so a fortiori when $s>\mathrm{log}_p(c)$), which is proved in \cite[Lemma~2.5.1]{CarusoLiu}. Part (1) is then obtained as a direct combination of (2) and (3).
\end{proof}

For a non--negative integer $s$, denote by $L_s$ the composite of the fields $K_s$ and $L$. The following adaptation of Theorem~4.1.1 of \cite{CarusoLiu} plays a key role in establishing the ramification bound. 

\begin{thm}\label{thm:Fuel}
Let $s$ be an integer satisfying 
$$s> M_0:=\mathrm{max}\left\{\mathrm{log}_p\left(\frac{a}{e}\right),\mathrm{log}_p\left(b-\frac{e}{p-1}\right)\right\}=\mathrm{max}\left\{\mathrm{log}_p\left(\frac{ip}{p-1}\right), \mathrm{log}_p\left(\frac{(i-1)e}{(p-1)}\right)\right\},$$ and let $E/K_s$ be an algebraic extension. Then the inclusion $\rho_{a, b}^{(s),E}(J_a^{(s), E}(M_{\BK}))\hookrightarrow \rho_{a, b}(J_a(M_{\BK})),$ facilitated by the inclusions $J_a^{(s), E}(M_{\BK}) \hookrightarrow J_a(M_{\BK})$ and $J_b^{(s), E}(M_{\BK}) \hookrightarrow J_b(M_{\BK})$ from Proposition~\ref{prop:ApproximateJc}, is an isomorphism if and only if $L_s \subseteq E$.
\end{thm}

\begin{proof}
The proof of \cite[Theorem~4.1.1]{CarusoLiu} applies in our context as well, as we now explain. In \cite[\S 4.1]{CarusoLiu}, for every $F/K_s$ algebraic, an auxillary set $\widetilde{J}_1^{(s), F}(M_{\BK})$ is constructed, together with maps of sets $\widetilde{\rho}_c^{(s), F}: \widetilde{J}_1^{(s), F}(M_{\BK}) \rightarrow J_c^{(s), F}(M_{\BK})$ for every $c \in (0, ep^s).$ Notably, the construction relies only on the fact that $M_{\BK}$ is a Breuil--Kisin module that is free over $k[[u]]$ and the assumption $s>\mathrm{log}_p(a/e)$. When $F$ is Galois over $K$, this set is naturally a $G_s$--set and the maps are $G_s$--equivariant. Moreover, the sets have the property that $\left(\widetilde{J}_1^{(s), F}(M_{\BK})\right)^{G_{F'}}=\widetilde{J}_1^{(s), F'}(M_{\BK})$ when $F/F'/K_s$ is an intermediate extension.

Subsequently, it is shown in \cite[Lemma~4.1.4]{CarusoLiu} that 
\begin{equation}\tag{$*$}\label{LiftedJ}
\widetilde{\rho}_b^{(s), F}\text{ is injective and its image is }\rho_{a, b}^{(s), F}(J_a^{(s), F}(M_{\BK})),
\end{equation}
where the only restriction on $s$ is again $s> \mathrm{log}_p(a/e)$.

Finally, one obtains a series of $G_s$--equivariant bijections:
\begin{center}
{\renewcommand{\arraystretch}{1.75}
\begin{tabular}{rclr}
$\widetilde{J}^{(s),\overline{K}}_1(M_{\BK}) $ & $ \simeq $ & $ \rho_{a, b}^{(s), \overline{K}}(J_a^{(s), \overline{K}}(M_{\BK}))$ & (by (\ref{LiftedJ}))\\
  & $ \simeq $ & $ \rho_{a, b}(J_a(M_{\BK}))$ & (Proposition~\ref{prop:ApproximateJc} (3))\\
  & $ \simeq $ & $ \rho_{a, b}^{\inf}(J_a^{\inf}(M_{\inf}))$ & (Proposition~\ref{prop:GsEquivarianceOfJc})\\
  & $ \simeq $ & $ T$ & (Proposition~\ref{prop:ActionProlongationJc} (2))
\end{tabular}
}
\end{center}
(where the step that uses Proposition~\ref{prop:GsEquivarianceOfJc} relies on the assumption $s>\mathrm{log}_p(b-e/(p-1))$ ).  Applying $(-)^{G_E}$ to both sides and using $(*)$ again then yields $$\rho_{a, b}^{(s), E}(J_a^{(s), E}(M_{\BK})) \simeq T^{G_E}.$$  Therefore, we may replace the inclusion from the statement of the theorem by the inclusion $T^{G_E} \subseteq T,$ and the claim now easily follows.
\end{proof}

Finally, we are ready to establish the desired ramification bound. Let $N_s=K_s(\zeta_{p^s})$ be the Galois closure of $K_s$ over $K$, and set $M_s=L_sN_s$. Then we have

\begin{prop}\label{prop:FontainePropertyForLs}
Let $s$ be as in Theorem~\ref{thm:Fuel}, and set $m=a/p^s$. Then the properties $(P_{m}^{L_s/K_s})$ and $(P_{m}^{M_s/N_s})$ hold.
\end{prop}

\begin{proof}
The proof of  $(P_{m}^{L_s/K_s})$ is the same as in \cite{CarusoLiu}, which refers to an older version of \cite{Hattori} for parts of the proof. Let us therefore reproduce the argument for convenience. By Corollary~\ref{cor:FontainePropertyWlogTotRamified}, it is enough to prove $(P_{m}^{L_s/K_s^{un}})$ where $K_s^{un}$ denotes the maximal unramified extension of $K_s$ in $L_s$.

Let $E/K_s^{un}$ be an algebraic extension and $f: \oh_{L_s}\rightarrow \oh_{E}/\mathfrak{a}_K^{>m}$ be an $\oh_{K_s^{un}}$--algebra map. Setting $c=a$ or $c=b$, it makes sense to consider an induced map $f_c: \oh_{L_s}/\mathfrak{a}_{L_s}^{>c/p^s}\rightarrow \oh_{E}/\mathfrak{a}_K^{>c/p^s}$, and we claim is that this map is well--defined and injective.

Indeed, let $\varpi$ be a uniformizer of $L_s$, satisfying the relation 
$$\varpi^{e'}=c_1\varpi^{e'-1}+c_2\varpi^{e'-2}+\dots + c_{e'-1}\varpi+c_{e'},$$
where $P(T)=T^{e'}-\sum_{i}c_{i}T^{e'-i}$ is an Eisenstein polynomial over $K_{s}^{un}.$ Applying $f$ one thus obtains $t^{e'}=\sum_{i}c_{i}t^{e'-i}$ in $\oh_{E}/\mathfrak{a}_K^{>m}$ where $t=f(\varpi),$ and thus, lifting $t$ to $\widetilde{t}\in \oh_E,$ we obtain the equality
$$\widetilde{t}^{e'}=c_1\widetilde{t}^{e'-1}+c_2\widetilde{t}^{e'-2}+\dots + c_{e'-1}\widetilde{t}+c_{e'}+r$$
with $v_K(r)>m>1/p^s$. It follows that $v_K(\widetilde{t})=v_K(\varpi)=1/p^se',$ and so $\varpi^n \in \mathfrak{a}_{L_s}^{>c/p^s}$ if and only if $\widetilde{t}^n \in \mathfrak{a}_{E}^{>c/p^s},$ proving that $f_c$ is both well--defined as well as injective.

The map $f_c$ then induces an injection of $k$--algebras $(\varphi_k^s)^*\oh_{L_s}/\mathfrak{a}_{L_s}^{>c/p^s}\hookrightarrow (\varphi_k^s)^*\oh_{E}/\mathfrak{a}_{E}^{>c/p^s}$ which in turn gives an injection $J_c^{(s), L_s}(M_{\BK})\rightarrow J_c^{(s), E}(M_{\BK})$, where $c=a$ or $c=b$; consequently, we obtain an injection $$\rho_{a, b}^{(s), L_s}(J_a^{(s), L_s}(M_{\BK}))\hookrightarrow \rho_{a, b}^{(s), E}(J_a^{(s), E}(M_{\BK})).$$ 
Combining this with Propositions~\ref{prop:ActionProlongationJc} and \ref{prop:ApproximateJc}, we have the series of injections
$$\rho_{a, b}^{(s), L_s}(J_b^{(s), L_s}(M_{\BK}))\hookrightarrow \rho_{a, b}^{(s), E}(J_b^{(s), E}(M_{\BK}))\hookrightarrow \rho_{a, b}^{(s), \overline{K}}(J_b^{(s), \overline{K}}(M_{\BK}))\hookrightarrow \rho_{a, b}(J_b(M_{\BK}))\simeq T.$$

Since $\rho_{a, b}^{(s), L_s}(J_b^{(s), L_s}(M_{\BK}))\simeq T$ by Theorem~\ref{thm:Fuel}, this is actually an injection $T \hookrightarrow T$ and therefore an isomorphism since $T$ is finite. In particular, the natural map $\rho_{a, b}^{(s), E}(J_b^{(s), E}(M_{\BK}))\hookrightarrow \rho_{a, b}(J_b(M_{\BK}))$ is an isomorphism, and Theorem~\ref{thm:Fuel} thus implies that $L_s \subseteq E$. This finishes the proof of (1).

Similarly as in \cite{CarusoLiu}, the property $(P_{m}^{M_s/N_s})$ is deduced from $(P_{m}^{L_s/K_s})$ as follows. Given an algebraic extension $E/N_s$ and an $\oh_{N_s}$--algebra morphism $\oh_{M_s} \rightarrow \oh_{E}/\mathfrak{a}_E^{>m},$ by restriction we obtain an $\oh_{K_s}$--algebra morphism $\oh_{L_s} \rightarrow \oh_{E}/\mathfrak{a}_E^{>m},$ hence there is a $K_s$--injection $L_s\rightarrow E$. As $N_s \subseteq E$, this can be extended to a $K_s$--injection $M_s\rightarrow E$, and since the extension $M_s/K_s$ is Galois, one obtains an $N_s$--injection $M_s\rightarrow E$ by precomposing with a suitable automorphism of $M_s$. 
\end{proof}

\begin{thm}\label{thm:FinalRamification}
Let $$\alpha=\lfloor M_0\rfloor+1=\left\lfloor \mathrm{log}_p\left(\mathrm{max}\left\{\frac{ip}{p-1}, \frac{(i-1)e}{p-1}\right\}\right)\right\rfloor+1.$$ Then 
\begin{enumerate}[(1)]
\item{
$$v_K(\mathcal{D}_{L/K})<1+e\alpha+\frac{iep}{p^{\alpha}(p-1)}-\frac{1}{p^{\alpha}}.$$}
\item{For any $\mu$ satisfying 
$$\mu>1+e\alpha+\mathrm{max}\left\{\frac{iep}{p^{\alpha}(p-1)}-\frac{1}{p^{\alpha}},
\frac{e}{p-1}\right\},$$ $G_K^{(\mu)}$ acts trivially on $T$.}
\end{enumerate}
\end{thm}

\begin{proof}
We may set $s=\alpha$ as the condition $s>M_0$ is then satisfied. Propositions~\ref{prop:RamificationEngine} and \ref{prop:FontainePropertyForLs} then imply that $v_K(\mathcal{D}_{L_s/K_s})<a/p^s$ (where $a=iep/(p-1)$ ) and thus 
\begin{align*}
v_K(\mathcal{D}_{L_s/K})&=v_K(\mathcal{D}_{K_s/K})+v_K(\mathcal{D}_{L_s/K_s})<1+es-\frac{1}{p^s}+\frac{a}{p^s}=1+e\alpha+\frac{a-1}{p^{\alpha}}.
\end{align*}
Similarly, we have $v_K(\mathcal{D}_{L/K})=v_K(\mathcal{D}_{L_s/K})-v_K(\mathcal{D}_{L_s/L})\leq v_K(\mathcal{D}_{L_s/K}),$ and the claim (1) thus follows.

To prove (2), let $M_s$ and $N_s$ be as in Proposition~\ref{prop:FontainePropertyForLs}. The fields $N_s$ and $M_s=LN_s$ are both Galois over $K$, hence Lemma~\ref{lem:transitivity} applies and we thus have
$$\mu_{M_s/K}=\mathrm{max}\left\{\mu_{N_s/K}, \phi_{N_s/K}(\mu_{M_s/N_s})\right\}.$$

By \cite[Remark~5.5]{Hattori}, we have $$\mu_{N_s/K}=1+es+\frac{e}{p-1}.$$
 
As for the second argument, Proposition~\ref{prop:RamificationEngine} gives the estimate 
$$\mu_{M_s/N_s}\leq e_{N_s/K}m=\frac{e_{N_s/K}}{p^s}a.$$

The function $\phi_{N_s/K}(t)$ is concave and has a constant slope $1/e_{N_s/K}$ beyond $t=\lambda_{N_s/K},$ where it attains the value $\phi_{N_s/K}(\lambda_{N_s/K})=\mu_{N_s/K}=1+es+e/(p-1)$. Thus, $\phi_{N_s/K}(t)$ can be estimated linearly from above as follows:
$$\phi_{N_s/K}(t)\leq 1+es+\frac{e}{p-1}+\frac{1}{e_{N_s/K}}\left(t-\lambda_{N_s/K}\right)=1+es+\frac{t}{e_{N_s/K}}-\frac{\lambda_{N_s/K}}{e_{N_s/K}}+\frac{e}{p-1}$$
There is an automorphism $\sigma \in \mathrm{Gal}(N_s/K)$ with $\sigma(\pi_s)=\zeta_{p}\pi_s$. That is, $v_K(\sigma(\pi_s)-\pi_s)=e/(p-1)+1/p^s,$ showing that 
$$\lambda_{N_s/K}\geq e_{N_s/K} \left( \frac{e}{p-1}+\frac{1}{p^s}\right),$$ 
and combinig this with the estimate of $\phi_{N_s/K}(t),$ we obtain
$$\phi_{N_s/K}(t)\leq 1+es+\frac{t}{e_{N_s/K}}-\frac{1}{p^s}.$$
Plugging in the estimate for $\mu_{M_s/N_s}$ then yields
\begin{align*}
\phi_{N_s/K}(\mu_{M_s/N_s})&\leq 1+es+\frac{a}{p^s}-\frac{1}{p^s}=1+es+\frac{\frac{iep}{p-1}-1}{p^s}.
\end{align*}  
Thus, we have
$$\mu_{L/K}\leq \mu_{M_s/K}\leq 1+e\alpha+\mathrm{max}\left\{\frac{iep}{p^\alpha(p-1)}-\frac{1}{p^{\alpha}}, \frac{e}{p-1} \right\},$$
which finishes the proof of part (2).
\end{proof}

\subsection{Comparisons of bounds}\label{subsec:Comparisons}
Finally, let us compare the bounds obtained in Theorem~\ref{thm:FinalRamification} with other results from the literature. These are summarized in the table below.

\begin{table}[h]
\noindent\begin{tabular}{| Sc| Sc  Sl |}
\hline
 &  $\mu_{L/K}\leq \cdots$ & \\
\hline 
Theorem~\ref{thm:FinalRamification} & $1+e\left(\left\lfloor \mathrm{log}_p\left(\mathrm{max}\left\{\frac{ip}{p-1}, \frac{(i-1)e}{p-1}\right\}\right)\right\rfloor+1\right)+\mathrm{max}\left\{\beta, \frac{e}{p-1}\right\},$ & $\beta<\mathrm{min}\left(e, 2p\right)$ \tablefootnote{More precisely: When $i=1$, it is easy to see that $\beta=(eip/(p-1)-1)/p^{\alpha}$ is smaller than $e/(p-1)$, and hence does not have any effect.  When $i>1$, one can easily show using $p^\alpha>ip/(p-1), p^\alpha>(i-1)e/(p-1)$ that $\beta<e$ and $\beta< pi/(i-1)\leq 2p$.} \\ 
\hline 
\cite{CarusoLiu} & $1+e\left(\left\lfloor \mathrm{log}_p\left(\frac{ip}{p-1}\right)\right\rfloor+1\right)+\mathrm{max}\left\{\beta, \frac{e}{p-1}\right\},$ &   $\beta<e$ \tablefootnote{The number $\beta$ here has different meaning than the number $\beta$ of \cite[Theorem~1.1]{CarusoLiu}.}\\
\hline  
\cite{Caruso} & $1+c_0(K)+e\left(s_0(K)+\mathrm{log}_p(ip)\right)+\frac{e}{p-1}$ & \\
\hline 
\cite{Hattori} & $\begin{cases} 1+e+\frac{e}{p-1}, \;\;\;\;\;\;\;\;\;\;i=1,\\ 1+e+\frac{ei}{p-1}-\frac{1}{p}, \;\;\;i>1, \end{cases}$ & under $ie<p-1$ \\
\hline
\cite{Fontaine2}, \cite{Abrashkin} & $1+\frac{i}{p-1}$ & under $\begin{matrix}e=1,\\ i< p-1\end{matrix}$ \\
\hline
\end{tabular}
\caption{Comparisons of estimates of $\mu_{L/K}$}\label{table}
\end{table}
\vspace{0.5cm}

\noindent\textbf{Comparison with \cite{Hattori}.} If we assume $ie<p-1$, then the first maximum in the estimate of $\mu_{L/K}$ is realized by $ip/(p-1) \in (1, p)$; that is, in Theorem~\ref{thm:FinalRamification} one has $\alpha=1$ and thus,
$$\mu_{L/K}\leq 1+e+\mathrm{max}\left\{\frac{ei}{p-1}-\frac{1}{p},\frac{e}{p-1}\right\},$$
which agrees precisely with the estimate \cite{Hattori}.

\noindent\textbf{Comparison with \cite{Fontaine2}, \cite{Abrashkin}.} Specializing to $e=1$ in the previous case, the bound becomes
$$\mu_{L/K}\leq \begin{cases} 2+\frac{1}{p-1}, \;\;\;\;\;\;\;\;i=1, \\  2-\frac{1}{p} + \frac{i}{p-1} \;\;\;i>1. \end{cases}$$
This is clearly a slightly worse bound than that of \cite{Fontaine2} and \cite{Abrashkin} (by $1$ and $(p-1)/p$, respectiely). 

\noindent\textbf{Comparison with \cite{CarusoLiu}.} From the shape of the bounds it is clear that the bounds are equivalent when $(i-1)e \leq ip,$ that is, when $e \leq p$ and some ``extra'' cases that include the case when $i=1$ (more precisely, these extra cases are when $e>p$ and $i\leq e/(e-p)$), and in fact, the terms $\beta$ in such situation agree. In the remaining case when $(i-1)e > ip,$ our estimate becomes gradually worse compared to \cite{CarusoLiu}. 

\begin{rems}
\begin{enumerate}[(1)]
\item{It should be noted that the bounds from \cite{CarusoLiu} do not necessarily apply to our situation as it is not clear when $\H^i_{\et}(\mathscr{X}_{\overline{\eta}}, \mathbb{Z}/p\mathbb{Z})$ (or rather their duals) can be obtained as a quotient of two lattices in a semi--stable representation with Hodge--Tate weights in $[0, i]$. To our knowledge the only result along these lines is \cite[Theorem~1.3.1]{EmertonGee1} that states that this is indeed the case when $i=1$ (and $X$ is a proper smooth variety over $K$ with semistable reduction). Interestingly, in this case the bound from Theorem~\ref{thm:FinalRamification} always agrees with the one from \cite{CarusoLiu}.}
\item{Let us also point out that the verbatim reading of the bound from \cite{CarusoLiu} as described  in Theorem~1.1 of \textit{loc. cit.} would have the term $\left\lceil \mathrm{log}_p(ip/(p-1)) \right\rceil$ (i.e. upper integer part) instead of the term $\left\lfloor \mathrm{log}_p(ip/(p-1)) \right\rfloor+1$ as in Table~\ref{table}, but we believe this version to be correct. Indeed, the proof of Theorem~1.1 in \cite{CarusoLiu} (in the case $n=1$) ultimately relies on the objects $J_{1,a}^{(s), E}(\mathfrak{M})$ that are analogous to $J_{a}^{(s), E}(M_{\BK})$, where $s=\left\lceil \mathrm{log}_p(ip/(p-1)) \right\rceil$. In particular, Lemma~4.2.3 of \textit{loc. cit.} needs to be applied with $c=a$, and the implicitly used fact that the ring $\oh_{E}/\mathfrak{a}_E^{>a/p^s}$ is a $k$--algebra (i.e. of characteristic $p$) relies on the \emph{strict} inequality $e>a/p^s$, equivalently $s>\mathrm{log}_p(ip/(p-1))$. In the case that $ip/(p-1)$ happens to be equal to $p^t$ for some integer $t$, one therefore needs to take $s=t+1$ rather than $s=t$. This precisely corresponds to the indicated change.}
\end{enumerate}
\end{rems}

\noindent\textbf{Comparison with \cite{Caruso}.} Let us explain the constants $s_0(K), c_0(K)$ that appear in the estimate. The integer $s_0(K)$ is the smallest integer $s$ such that $1+p^s\mathbb{Z}_p \subseteq \chi(\mathrm{Gal}(K_{p^{\infty}}/K))$ where $\chi$ denotes the cyclotomic character. The rational number $c_{0}(K) \geq 0$ is the smallest constant $c$ such that $\psi_{K/K'}(1+t) \geq 1+et-c$ (this exists since the last slope of $\psi_{K/K'}(t)$ is $e$)\footnote{To make sense of this in general, one needs to extend the definition of the functions $\psi_{L/M}, \varphi_{L/M}$ to the case when the extension $L/M$ is not necessarily Galois. This is done e.g. in \cite[\S 2.2.1]{Caruso}.}. 

In the case when $K/K'$ is tamely ramified, the estimate from \cite{Caruso} becomes 
$$\mu_{L/K}\leq 1+e\left(\mathrm{log}_p(ip)+1\right)+\frac{e}{p-1},$$
which is fairly equivalent to the bound from Theorem~\ref{thm:FinalRamification} when $e < p$ (and again also in some extra cases, e.g. when $i=1$ for any $e$ and $p$), with the difference of estimates being approximately 
$$e\left(\mathrm{log}_p\left(\frac{p}{p-1}\right)-\frac{1}{p-1}\right) \in \left(-\frac{e}{4\sqrt{p}},\, 0\right).$$ 
In general, when $e$ is big and coprime to $p$, the bound in \cite{Caruso} becomes gradually better unless, for example, $i=1$. 

In the case when $K$ has relatively large wild absolute ramification, we expect that the bound from Theorem~\ref{thm:FinalRamification} generally becomes stronger, especially if $K$ contains $p^n$--th roots of unity for large $n$, as can be seen in the following examples (where we assume $i>1$; for $i=1$, our estimate retains the shape of the tame ramification case and hence the difference between the estimates becomes even larger). 

\begin{pr} 
\begin{enumerate}[(1)]
\item{When $K=\mathbb{Q}_p(\zeta_{p^n})$ for $n\geq 2$, one has $e=(p-1)p^{n-1}$, $s_0(K)=n$ and from the classical computation of $\psi_{K/\mathbb{Q}_p}$ (e.g. as in \cite[IV \S4]{SerreLocalFields}), one obtains $$c_0(K)=[(n-1)(p-1)-1]p^{n-1}+1.$$  The difference between the two estimates is thus approximately $ne-p^{n-1}+1 >(n-1)e$.}
\item{When $K=\mathbb{Q}_p(p^{1/p^n})$ for $n\geq 3$, one has $e=p^{n}$ and $s_0(K)=1$. The description of $\psi_{K/\mathbb{Q}_p}$ in \cite[\S 4.3]{CarusoLiu} implies that $c_0(K)=np^n=ne$.  The difference between the two estimates is thus approximately $$e\left(1+\mathrm{log}_p(i)-\mathrm{log}_p(i-1)+\mathrm{log}_p(p-1)\right)\approx 2e.$$
(In the initial cases $n=1, 2$, one can check that the difference is still positive, in both cases bigger than $p$.)}
\end{enumerate}
\end{pr}

\addcontentsline{toc}{section}{References}
\bibliography{references}
\bibliographystyle{amsalpha}.
\end{document}